\documentclass[11pt]{amsart}
\usepackage{amsmath}
\usepackage{amsthm}
\usepackage{amssymb}
\usepackage{sseq}
\usepackage[pagebackref, breaklinks]{hyperref}
\usepackage{hyperref}
\usepackage[curve]{xypic}
\usepackage[left=3.2cm,right=3.2cm,top=3cm,bottom=3cm]{geometry} 
\usepackage{pdfpages}
\usepackage{stmaryrd}
\usepackage{marvosym}
\usepackage{wrapfig}
\usepackage{url}
\usepackage{breakurl}

\let\oldmarginpar\marginpar
\renewcommand\marginpar[1]{\-\oldmarginpar[\raggedleft\footnotesize #1]%
{\raggedright\footnotesize #1}}

\numberwithin{equation}{section}
\newtheorem{thm}[equation]{Theorem}

\newtheorem{lemma}[equation]{Lemma}

\newtheorem{cor}[equation]{Corollary}

\newtheorem{prop}[equation]{Proposition}
\newtheorem{claim}[equation]{Claim}

\newtheorem*{conjecture*}{Conjecture}

\newtheorem*{question*}{Question}

\theoremstyle{definition}
\newtheorem{defi}[equation]{Definition}

\newtheorem{convention}[equation]{Convention}

 \newtheorem{example}[equation]{Example}
  \newtheorem*{example*}{Example}
 \newtheorem{examples}[equation]{Examples}

\theoremstyle{remark}

\newtheorem{remark}[equation]{Remark}

\def\co{\colon\thinspace}

\newcommand{\G}{\mathbb{G}}
\newcommand{\Z}{\mathbb{Z}}
\newcommand{\F}{\mathbb{F}}
\newcommand{\Fb}{\overline{\mathbb{F}}}

\newcommand{\Q}{\mathbb{Q}}

\newcommand{\N}{\mathbb{N}}
\newcommand{\A}{\mathbb{A}}

\newcommand{\C}{\mathbb{C}}

\newcommand{\Gm}{\mathbb{G}_m}

\newcommand{\CC}{\mathcal{C}}

\newcommand{\EE}{\mathcal{E}}
\newcommand{\FF}{\mathcal{F}}
\newcommand{\GG}{\mathcal{G}}

\newcommand{\LL}{\mathcal{L}}
\newcommand{\PP}{\mathcal{P}}

\newcommand{\XX}{\mathcal{X}}

\newcommand{\YY}{\mathcal{Y}}

\newcommand{\OO}{\mathcal{O}}

\newcommand{\MM}{\mathcal{M}}

\newcommand{\MMb}{\overline{\mathcal{M}}}

\newcommand{\tensor}{\otimes}

\DeclareMathOperator{\cont}{cont}
\DeclareMathOperator{\pt}{pt}

\DeclareMathOperator{\Aut}{Aut}

\DeclareMathOperator{\Spf}{Spf}

\DeclareMathOperator{\ev}{ev}
\DeclareMathOperator{\Mat}{Mat}
\DeclareMathOperator{\UpT}{UpT}
\DeclareMathOperator{\Shv}{Shv}

\DeclareMathOperator{\modules}{\text{-}mod}

\DeclareMathOperator{\id}{id}

\DeclareMathOperator{\im}{im}

\DeclareMathOperator{\sm}{\wedge}

\DeclareMathOperator{\Tors}{Tors}
\DeclareMathOperator{\Tor}{Tor}

\DeclareMathOperator{\Gal}{Gal}
\DeclareMathOperator{\Ext}{Ext}
\DeclareMathOperator{\Hom}{Hom}
\DeclareMathOperator{\Spec}{Spec}

\DeclareMathOperator{\holim}{holim}

\DeclareMathOperator{\QCoh}{QCoh}

\DeclareMathOperator{\Pic}{Pic}

\DeclareMathOperator{\rk}{rk}

\newcommand*{\myproofname}{Proof of claim}
\newenvironment{myproof}[1][\myproofname]{\begin{proof}[#1]}{\end{proof}}

\newcommand{\lbreak}{\leavevmode \\\vspace{-0.5cm}}

\DeclareMathOperator{\Xb}{\overline{\XX}}
\DeclareMathOperator{\XXb}{\overline{\XX}}
\DeclareMathOperator{\Yb}{\overline{\YY}}

\DeclareMathOperator{\Prj}{\mathbb{P}}

\begin{document}
\title[(Topological) modular forms with level structures]{(Topological) modular forms with level structures: decompositions and duality}
\author{Lennart Meier}
\maketitle
\begin{abstract}
 We study decompositions of vector bundles on the moduli stack of elliptic curves coming from moduli of elliptic curves with level structure. These decompositions imply decomposition results for rings of modular forms and also for topological modular forms. We give explicit formulas for these decompositions and also apply them to equivariant topological modular forms. 
 
 Moreover, we study the dualizing sheaf on $\MMb_1(n)$ and characterize the numbers $n$ such that $Tmf_1(n)$ is Anderson self-dual. 
\end{abstract}

\section{Introduction}
Rings of modular forms with level structure are of great importance in number theory. An example is $M_R(\Gamma_1(n))$, the ring of (holomorphic) $\Gamma_1(n)$-modular forms over a ring $R$, in which we assume $n$ to be invertible. 

While there is a lot of information available for low $n$, in general these rings are hard to understand. For example, it is an equivalent form of a famous theorem of Mazur \cite{Maz77} that the ring $M_\Q(\Gamma_1(n))[\Delta^{-1}]$ of meromorphic modular forms only admits a ring homomorphism to $\Q$ if $n\leq 10$ or $n=12$, i.e.\ exactly if (the coarse moduli space of) $\MMb_1(n)_{\C}$ has genus $0$.  Here, $\MMb_1(n)$ is the compactification of the moduli stack $\MM_1(n)$ of elliptic curves with chosen point of exact order $n$.

The aim of the present article is instead the more modest goal of an \emph{additive} understanding of the ring of modular forms with level structure. It is an elementary observation that $M_\Q(\Gamma_1(n))$ always splits as a graded $\Q$-vector space into shifted copies of $M_\Q(\Gamma_1(2))$ (and the same works with $M_\Q(\Gamma_1(3))$ or the ring of level-$1$ modular forms $M_\Q$ itself instead). Several questions arise: Is a similar splitting possible for other base rings and is it natural in the choice of base ring? Is the splitting compatible with the $M_\Q$-module structure? 

Under rather weak assumptions we can give a positive answer to these questions in a strong form. Denote by $f_n\colon \MMb_1(n)_R \to \MMb_{ell,R}$ the projection to the compactified moduli stack of elliptic curves, where $R$ is $\Z_{(l)}$ or a field of characteristic $l$ with $(l,n) = 1$ or $l=0$, and by $\EE_n$ the pushforward $(f_n)_*\OO_{\MMb_1(n)_R}$. Furthermore, denote by $\omega$ the pushforward $p_*\Omega^1_{\CC/\MMb_{ell}}$ of the sheaf of differentials of the universal generalized elliptic curve $p\colon \CC \to \MMb_{ell}$.

\begin{thm}\label{thm:int1}
Let $n\geq 4$. If $R$ is a field or $H^1(\MMb_1(n)_{(l)};\omega)$ has no $l$-torsion,\footnote{This happens if and only if all weight $1$ mod $l$ cusp forms of level $\Gamma_1(n)$ lift as $\Gamma_1(n)$-cusp forms to $\Z_{(l)}$.} the vector bundle $\EE_n$ on $\MMb_{ell,R}$ decomposes into vector bundles of the form
\begin{itemize}
\item $\EE_3\tensor \omega^{\tensor m}$ for $l=2$,
\item $\EE_2\tensor \omega^{\tensor m}$ for $l=3$,
\item $\omega^{\tensor m}$ for $l>3$.
\end{itemize} 
If $l=0$, the vector bundle $\EE_n$ decomposes indeed both into vector bundles of the form $\EE_2\tensor \omega^{\tensor m}$ and into vector bundles of the form $\EE_3\tensor \omega^{\tensor m}$ (if $n\geq 5$). 

For the corresponding splitting on the uncompactified moduli stack the $H^1$-condition is unnecessary if we are willing to $l$-complete everything in the case $R = \Z_{(l)}$. 
\end{thm} 

Note that $M_R(\Gamma_1(n))$ are the global sections of $\bigoplus_{i\in\Z}\EE_n\tensor \omega^{\tensor i}$ so that we get the additive splittings of the rings of modular forms alluded to above. Note further that the vector bundles $\omega^{\tensor m}$, $\EE_2$ and $\EE_3$ are well-understood, giving a rather complete understanding of all $\EE_n$.

The theorem works likewise for $\MMb(n)$ instead of $\MMb_1(n)$ and also for $\MMb_0(n)$ if $l\neq 2,3$ or $l=3$ does not divide $\phi(n) = |(\Z/n)^\times|$ and $n$ is squarefree. Which tensor powers $\omega^{\tensor m}$ occur how often in the splitting of $\EE_n$ is computable in terms of dimensions of spaces of modular forms and we will give explicit formulas and examples in Section \ref{sec:dec}. 

In some cases these formulas possess a remarkable symmetry. This happens if the dualizing sheaf of $\MMb_1(n)$ is isomorphic to a power of $\omega$. We will classify the values of $n$ where this occurs. These are $n\leq 8$ (genus $0$) and $n=11,14,15$ (genus $1$) and $n=23$ (genus $12$). 

In contrast, over the uncompactified moduli stack all powers of $\omega$ in the decomposition appear equally often if $n\geq 4$ and $l=2$ or $3$. This finishes the account of the algebro-geometric content of the present paper. \\

Of equal importance to the author (and perhaps to the reader) are the applications to stable homotopy theory and, more precisely, to the theory of topological modular forms. 

The homotopy theorist's version of a commutative ring is an $E_\infty$-ring spectrum, essentially a multiplicative cohomology theory that satisfies associativity and commutativity in a homotopy coherent way. The homotopy theorist's version of the ring of modular forms $M_{\Z}$ is the $E_\infty$-ring spectrum $tmf$ of topological modular forms as constructed by Goerss, Hopkins and Miller \cite{TMF}. It comes with a ring homomorphism $\pi_{2*}tmf \to M_{\Z}$ from its homotopy groups that is neither injective (the source contains torsion) nor surjective (it does not hit $\Delta$), but is an isomorphism after inverting $6$. 

There are many variants of the spectrum $tmf$. While $tmf$ is connective (i.e.\ $\pi_ktmf = 0$ for $k<0$), there is also a nonconnective version $Tmf$, which takes the cohomology of $\omega^{\tensor k}$ for negative $k$ into account. There is also $TMF = tmf[\Delta^{-1}]$ corresponding to meromorphic modular forms or the uncompactified moduli stack of elliptic curves. 

We can also construct versions with level structures. In the uncompactified situation, $TMF_1(n)$, $TMF(n)$ and $TMF_0(n)$ were already available from the original construction of $TMF$. The corresponding $Tmf_1(n)$, $Tmf(n)$ and $Tmf_0(n)$ were constructed by Goerss--Hopkins (as written down in \cite{Sto12}) and in full generality by Hill--Lawson \cite{HL13}. Our algebraic results are easily applicable to the topological situation and imply the following theorem using the descent spectral sequence.

\begin{thm}
Let $n\geq 4$ and $l$ be a prime not dividing $n$. If $H^1(\MMb_1(n)_{(l)};\omega)$ has no $l$-torsion, the $Tmf_{(l)}$-module $Tmf_1(n)_{(l)}$ splits into $Tmf_{(l)}$-modules of the form
\begin{itemize}
\item $\Sigma^{2m} Tmf_1(3)_{(l)}$ for $l=2$,
\item $\Sigma^{2m} Tmf_1(2)_{(l)}$ for $l=3$,
\item $\Sigma^{2m} Tmf_{(l)}$ for $l>3$.
\end{itemize} 
For the corresponding splitting of $TMF_1(n)$ the condition on $l$-torsion in $H^1$ is not necessary if we are willing to $l$-complete everything.
\end{thm} 

There are analogous theorems for $Tmf(n)$ and also for $Tmf_0(n)$ (if $l\neq 2,3$ or $l=3$ does not divide $\phi(n)$ and $n$ is squarefree). These conditions are equivalent to $Tmf_1(n)$, $Tmf(n)$ or $Tmf_0(n)$ (for $n$ squarefree) having torsionfree homotopy groups -- this is an obvious necessary condition as $Tmf_1(3)_{(2)}$, $Tmf_1(2)_{(3)}$ and $Tmf_{(l)}$ for $l>3$ have torsionfree homotopy groups. Indeed, these spectra and their homotopy groups are well-understood, which gives the theorem its strength. The exact suspensions occurring in the theorem can be computed explicitly. 

Our theorem on the dualizing sheaf of $\MMb_1(n)$ has an implication for the \emph{Anderson dual} $I_{\Z[\frac1n]}Tmf_1(n)$ of $Tmf_1(n)$. The Anderson dual of a spectrum $X$ is defined so that one has a short exact sequence
\[0 \to \Ext^1_{\Z}\left(\pi_{-k-1}X, \Z[\tfrac1n]\right) \to \pi_kI_{\Z[\frac1n]}X \to \Hom_{\Z}\left(\pi_{-k}X, \Z[\tfrac1n]\right) \to 0. \]
 If $X$ is Anderson self-dual (up to suspension), then one obtains a convenient universal coefficient sequence for $X$. 

\begin{thm}The Anderson dual $I_{\Z[\frac1n]}Tmf_1(n)$ of $Tmf_1(n)$ is equivalent to a suspension of $Tmf_1(n)$ if and only if $1\leq n \leq 8$ or $n=11,14,15$ or $n=23$.
\end{thm}
The case $n=1$, which is in some sense the most difficult, was already obtained in \cite{Sto12}. 

We want to finish with a few words of motivation. Topological modular forms with low level (say $tmf_1(3)$) have been important in the history of $tmf$ from the very beginning, e.g.\ in the computation of the homology of $tmf$ (as written up in \cite{Mathom}). But topological modular forms with high level -- to whose understanding the present paper contributes -- have also become more and more important in the last years and have for example been used to understand $TMF$-cooperations \cite{BOSS}. We want especially to stress applications to equivariant $TMF$ though. As will be explained in Section \ref{sec:eq}, in Lurie's model the $(\Z/n)$-fixed points $TMF^{\Z/n}$ of $\Z/n$-equivariant $TMF$ satisfy
$$TMF^{\Z/n}_{(l)} \simeq \prod_{k|n}TMF_1(k)_{(l)}$$
for $l$ not dividing $n$. Our results help us thus to understand equivariant $TMF$ for a cyclic group.\\

The structure of the present article is as follows. We begin in Section \ref{sec:weighted} with a classification result about vector bundles on weighted projective lines over normal rings. This applies in particular to $\MMb_1(2)$, $\MMb_1(3)$ and to $\MMb_{ell, (l)}$ for $l>3$. We continue in Section \ref{sec:ModularCurves} with background on the various moduli stacks we will consider.  In Section \ref{sec:decexistence}, we prove our algebraic decomposition results, first in the case of a field and then over $\Z_{(l)}$. In Section \ref{sec:dec}, we will give explicit formulas for these decompositions and in Section \ref{sec:duality}, we will prove our statement about dualizing sheaves. Section \ref{sec:tmf} will apply these algebraic considerations to $TMF$ and will also discuss the relevance for equivariant $TMF$. Appendix \ref{app:Hasse} contains a proof how to lift the Hasse invariant to a characteristic zero modular form in the presence of a level structure. Appendix \ref{sec:tables} contains tables of 
decompositions.

\subsection*{Acknowledgments}
I want to thank Viktoriya Ozornova for many discussions and computations, on which several of the ideas of the present article are based, and for providing the proof of Proposition \ref{prop:degreecomp}. Moreover, her comments on earlier versions of this paper have been a great help. Dimitar Kodjabachev also caught an oversight. I furthermore thank the mathoverflow community for many helpful questions and answers; in particular, the user Electric Penguin for his argument for lifting the Hasse invariant, which is crucial for significant parts of the present paper. 

\subsection*{Conventions}
The symbol $/$ will always denote \emph{stack} quotients. The dual of a module $M$ will be denoted by $M^{\vee}$ and similarly for sheaves. 

While we use the standard $\square$ for the end of proofs of theorems, propositions etc., we will use the symbol $\blacksquare$ for the end of proofs of claims inside bigger proofs.

\section{Vector bundles on weighted projective lines}\label{sec:weighted}
The aim of this section is to generalize some well-known facts about coherent sheaves and vector bundles on projective spaces to weighted projective stacks. This is relevant for our purposes because several compactified moduli stacks of elliptic curves (with level structure) are weighted projective stacks (see Example \ref{exa:moduli}).

\begin{defi}\label{def:wps} For $a_0,\dots, a_n$ positive integers  and a commutative ring $R$, the \textit{weighted projective stack} $\PP_R(a_0,\dots, a_n)$ is the (stack) quotient of $\A_R^{n+1}-\{0\}$ by the multiplicative group $\G_m$ under the action which is the restriction of the map
\begin{eqnarray*} \phi\co \G_m\times \A_R^{n+1} &\to& \A_R^{n+1} \\
 \Z[t, t^{-1}]\otimes R[t_0,\dots, t_n] &\leftarrow& R[t_0,\dots, t_n] \\
 t^{a_i}\tensor t_i &\mapsfrom& t_i
\end{eqnarray*}
to $\G_m\times (\A^{n+1}_R-\{0\})$. Here, $\A_R^{n+1}-\{0\}$ denotes the complement of the zero point, i.e.\ of the common vanishing locus of all $t_i$. On geometric points, the action corresponds to the map $(t, t_0,\dots, t_n) \mapsto (t^{a_0}t_0,\dots, t^{a_n}t_n)$. In the special case of $n=1$ we speak of a \textit{weighted projective line}.\end{defi}

As explained in \cite[Section 2]{Mei13}, this is a smooth and proper Artin stack over $\Spec R$, Deligne--Mumford if all $a_i$ are invertible on $R$. 

Recall that a grading on a commutative ring $A$ is equivalent to a $\Gm$-action on $\Spec A$. Moreover, there is an equivalence between graded $A$-modules and quasi-coherent sheaves on $\Spec A/\Gm$ given by pullback to $\Spec A$ and this equivalence is compatible with $\tensor$. The map $\phi$ above gives a $\G_m$-action on $\A^n_R$ and this corresponds to the grading $|t_i| = a_i$. The category of quasi-coherent modules on $\A^{n+1}_R/\G_m$ is thus equivalent to graded $R[t_0,\dots, t_n]$-modules. 

For $M$ a graded module, denote by $M[m]$ the graded module with $M[m]_k = M_{m+k}$. Then $R[t_0,\dots, t_n][m]$ is a graded $R[t_0,\dots, t_n]$-module, which corresponds to a line bundle on $\A^{n+1}_R/\G_m$ whose restriction to $\PP_R(a_0,\dots, a_n)$ we denote by $\OO(m)$. As usual, we set $\FF(m) = \FF \tensor \OO(m)$. It is easy to see that for a quasi-coherent sheaf $\FF$ on $\PP_R(a_0,\dots, a_n)$, the graded global sections $\Gamma_*(\FF) = H^0(\PP_R(a_0,\dots, a_n); \bigoplus_{m\in\Z}\FF(m))$ are exactly 
the graded $R[t_0,\dots, t_n]$-module corresponding to $\FF$. 

The following theorem summarizes some of the fundamental properties of $\OO(m)$:

\begin{thm}\label{thm:fundamentalweighted}Let $X = \PP_R(a_0,\dots, a_n)$.
 \begin{enumerate}
  \item \label{item:P1}The sheaf $\OO(1)$ is ample in the sense that for every coherent sheaf $\FF$ on $X$, there is a surjection from a sum of sheaves of the form $\OO(m)$ to $\FF$.
    \item \label{item:P1.5}For any coherent sheaf $\FF$, there exist an $m$ such that $H^i(X;\FF(m)) = 0$ for all $i>0$.
  \item \label{item:P2}The sheaf $\OO(-\sum_{i=0}^n a_i)$ is dualizing in the sense that there are natural isomorphisms
  \[\Hom_{\OO_X}(\FF,\OO(-\sum_{i=0}^n a_i)) \xrightarrow{\cong} \Hom_R(H^n(X; \FF),R)\]
  for all coherent sheaves $\FF$ on $X$. 
 \end{enumerate}
\end{thm}
\begin{proof}
 The proofs are analogous to the classical proofs. In some more detail:
 
 Let $\FF$ be a coherent sheaf on $\PP_R(a_0,\dots, a_n)$ and set $M = \Gamma_*(\FF)$. The stack $X$ is covered by the non-vanishing loci $D(t_i)$, where $t_i\in H^0(X; \OO(a_i))$. Furthermore, 
 $$D(t_i) \simeq \Spec R[t_0,\dots, t_n][t_i^{-1}] / \G_m.$$
 The restriction of $\FF$ to $D(t_i)$ corresponds to the graded $R[t_0,\dots, t_n][t_i^{-1}]$-module $M[t_i^{-1}]$.  Choose generating elements $s_{ij}$ in $M[t_i^{-1}]$. By multiplying with a power of $t_i$ we can assume that all $s_{ij}$ are actually in $M$ and thus define elements in $\Hom_X(\OO(m), \FF)$ for some $m$. Taking the sum of all these maps defines a surjection, proving (\ref{item:P1}). 
 
  For (\ref{item:P1.5}), we can argue by downward induction on $i$ as in \cite[Thm III.5.2]{Har77}, once we know that $X$ has cohomological dimension $\leq n$. This is clear as $X$ can be covered by the $(n+1)$ open substacks $D(t_i)$, on which the global sections functor is exact (because it corresponds to taking the $0$-graded piece of a graded module). 
 
 That $\OO(-\sum_{i=0}^n a_i)$ acts as a dualizing sheaf for all line bundles of the form $\OO(m)$ was shown in \cite[Prop 2.5]{Mei13}. The general case follows as in \cite[Thm III.7.1]{Har77} because $\OO(1)$ is ample.
\end{proof}

We also want to recall the cohomology of $\OO(m)$ on $\PP_R(a,b)$ from \cite[Prop 2.5]{Mei13}.

\begin{prop}\label{prop:projcoh}
 Let $B(m)$ be the set of pairs $(\lambda, \mu)$ of negative integers with $\lambda a + \mu b = m$. Then $H^1(\PP_R(a,b);\OO(m))$ is isomorphic to the free $R$-module on $B(m)$. 
\end{prop}

By a result that I learned from Angelo Vistoli \cite[Prop 3.4]{Mei13}, we have the following:
\begin{prop}\label{prop:ClassField}Let $K$ be an arbitrary field, $a_0, a_1\in\N$. Then every vector bundle $\FF$ on $\PP_K(a_0, a_1)$ is a direct sum of line bundles of the form $\OO(m)$.\end{prop}

We want to prove a generalization to weighted projective lines over more general rings, which is in the spirit of \cite[Theorem 1.4]{H-S99b}. First we need three lemmas. 

\begin{lemma}\label{lem:divisor}
 Let $k$ be an algebraically closed field and let $s$ be a section of a vector bundle $\EE$ on $X = \PP_k(a,b)$. Assume that $s$ vanishes at some geometric point of $X$. Then $s$ is in the image of a morphism $\EE(-j) \to \EE$ for some $j>0$. 
\end{lemma}
\begin{proof}
 This would easily follow from a suitable formalism of divisors on Artin stacks. We will argue in a more elementary way. 
 
 Let $M$ be the global sections of the pullback of $\EE$ to $\A_k^2-\{0\}$. This pullback can be extended to a vector bundle $\FF$ on $\A^2_k$ with the same global sections. The module $M$ is a (finite rank free) graded module over the polynomial ring $k[x,y]$ with $|x| = a$ and $|y| = b$. Assume that $s$ vanishes at a geometric point that is the image of $(u,v) \in \A^2_k-\{0\}$. Then (the pullback of) $s$ also vanishes on $f(\mathbb{G}_{m,k})$ for $f\colon \A^1_k \to \A^2_k$ the map described by the formula $\lambda \mapsto (\lambda^au, \lambda^bv)$ for $\lambda \in k$. We claim that $f(\mathbb{G}_{m,k})$ is closed in $\A^2_k-\{0\}$.
 
 First assume that $u=0$ or $v=0$, say $v=0$. Then $f$ can on $\mathbb{G}_{m,k}$ be written as the composition 
 $$\mathbb{G}_{m,k} \to \mathbb{G}_{m,k} \cong \mathbb{G}_{m,k} \times \{0\} \to \A^2_k-\{0\},$$
 where the first map is the surjection $\lambda \mapsto \lambda^au$ and the last map is obviosly a closed immersion. 
 
 If $u$ and $v$ are nonzero, let $g$ be $\gcd(a,b)$. Because 
 $$\mathbb{G}_{m,k} \to \mathbb{G}_{m,k}, \qquad \lambda \mapsto \lambda^g$$
 is surjective, we can assume that $a$ and $b$ are coprime. Thus, $f$ defines closed immersions $\mathbb{G}_{m,k} \to \Spec k[x^{\pm 1},y]$ and $\mathbb{G}_{m,k} \to \Spec k[x,y^{\pm 1}]$. Hence $f(\mathbb{G}_{m,k})$ is closed in $\A^2_k-\{0\}$. 
 
 It follows that $A = f(\mathbb{A}^1_k)$ is closed and irreducible in $\A^2_k$ and thus must be the closure of $f(\mathbb{G}_{m,k})$. Thus, $s$ vanishes on $A$. The set $A$ corresponds to a prime ideal $\mathfrak{p} \subset k[x,y]$ of height $1$. As $k[x,y]$ is factorial, $\mathfrak{p}$ contains a prime element $q$ and thus $\mathfrak{p} = (q)$. 
 As $q = q(x,y)$ and $q(\lambda^ax, \lambda^by)$ for $\lambda \in k^\times$ have both the zero set $A$, they must be unit multiple of each other and it follows that $q$ is homogeneous of some positive degree $j$. Thus, the element $m\in M$ corresponding to $s$ must be of the form $qm'$ for $m'\in M$, where $|m'| = |m|-j$. 

\end{proof}

\begin{lemma}\label{lem:basechange}
 Let $\EE$ be a quasi-coherent sheaf on $X = \PP_R(a_0,a_1,\dots, a_n)$ that is flat over $R$. Let $R\to S$ be a morphism of commutative rings and denote by $f$ the projection $$Y=X\times_{\Spec R}\Spec S\xrightarrow{f} X.$$
 If $H^i(X;\EE)$ is a flat $R$-module for $i> p$, then 
 $$H^p(Y;f^*\EE) \cong H^p(X;\EE)\tensor_R S.$$ 
 The assumption is in particular fulfilled for $p=n$. 
\end{lemma}
\begin{proof}
 Let $\{U_i\to X\}_{0\leq i \leq n}$ be the covering by the standard opens and $\check{C}(\EE)$ be the corresponding \v{C}ech complex, whose cohomology is $H^*(X;\EE)$. We can compute $H^*(Y,f^*\EE)$ as the cohomology of $\check{C}(\EE)\tensor_R S$. Then the resulting K{\"u}nneth spectral sequence
 $$\Tor_s^R(H^t(X;\EE), S) \Rightarrow H^{t-s}(Y;f^*\EE)$$
 implies the result. 
\end{proof}
\begin{remark}
 The same lemma holds, of course, much more generally, e.g.\ for any quasi-compact semi-separated scheme.
\end{remark}

\begin{lemma}\label{lem:constvector}
Let $\XX$ be a normal noetherian Artin stack and $\FF$ a coherent sheaf on $\XX$. Assume that there is an integer $n$ such that for every point $x\colon \Spec k \to \XX$, the pullback $x^*\FF$ is free of rank $n$. Then $\FF$ is a vector bundle.
\end{lemma}
\begin{proof}
 By taking a smooth cover, we reduce to the case of a noetherian normal scheme $X$. As a coherent sheaf over a noetherian scheme is a vector bundle if and only if its stalks are free over the stalks of the structure sheaf, we can assume that $\XX = \Spec A$ for a noetherian local domain $A$. Here, the statement is part of \cite[Thm 2.9]{Mil80}.
\end{proof}

\begin{thm}\label{thm:VectorBundle}
 Let $\EE$ be a vector bundle on $X = \PP_R(a_0,a_1)$ for $R$ a noetherian and normal ring. Then the following conditions are equivalent. 
 \begin{enumerate}
  \item \label{item:free} Both $H^0(X;\EE(m))$ and $H^1(X;\EE(m))$ are free $R$-modules for all $m\in\Z$. 
  \item \label{item:dec} The vector bundle $\EE$ decomposes into a sum of line bundles of the form $\OO(m)$. 
 \end{enumerate}

\end{thm}
\begin{proof}
By Proposition \ref{prop:projcoh}, the part (\ref{item:dec}) implies (\ref{item:free}) because the cohomology of $\OO(m)$ is a free $R$-module. 
 
 Now we assume (\ref{item:free}) and want to prove (\ref{item:dec}). The proof will be similar to one of the standard proofs for an unweighted projective line over a field. We will argue by induction on the rank of $\EE$ and assume that the theorem has been proven for all ranks $\leq r$ and that $\EE$ has rank $r+1$. 
 
 Denote by $\EE^\vee$ the $\OO_X$-dual of $\EE$. By Theorem \ref{thm:fundamentalweighted}, there is a maximal $m$ such that $H^1(X; \EE^\vee(m)) \neq 0$. Setting $m_0 = -m-a_0-a_1$, we claim that $m_0$ is the smallest index such that $H^0(X;\EE(m_0)) \neq 0$. Indeed: By Lemma \ref{lem:basechange}, we have for $j\colon \Spec k \to \Spec R$ (for $k$ a field) and every $i\in\Z$ an isomorphism $H^1(X; \EE^\vee(i))\tensor_R k \cong H^1(X_k; j^*\EE^{\vee}(i))$. Thus, for every $i>m$, the group $H^1(X_k; j^*\EE^{\vee}(i))$ vanishes and there exists a point $j$ of $\Spec R$ such that $H^1(X_k; j^*\EE^{\vee}(m))$ is nonzero. Serre duality implies that $H^0(X_k; j^*\EE(m_0)) \neq 0$ for this $j$ and $H^0(X_k;j^*\EE(i)) = 0$ for every $j\colon \Spec k \to \Spec R$ if $i<m_0$. By Lemma \ref{lem:basechange} again, $$H^0(X_k;j^*\EE(i)) \cong H^0(X;\EE(i)) \tensor_R k$$
  because $H^1(X;\EE(i))$ is a free $R$-module, which shows the claim that $m_0$ is minimal with $H^0(X;\EE(m_0)) \neq 0$. 
 
 By possibly tensoring $\EE$ with $\OO(-m_0)$, we can assume that $m_0 = 0$. Choose now an element $s \in H^0(X;\EE)$ that is part of an $R$-basis. Then we consider the sequence
 \begin{align}\label{align:vector}\OO_X \xrightarrow{s} \EE \to \FF \to 0.\end{align}
 We want to show that $s$ defines an injection and that its cokernel $\FF$ is a vector bundle. By Lemma \ref{lem:basechange}, we see that $s$ is still nonzero after base change to an arbitrary geometric point $j\colon \Spec k \to \Spec R$. We claim that $s$ does not vanish at any geometric point of $X_k$. Indeed, if $s$ had a zero on $X_k$, then $s$ would by Lemma \ref{lem:divisor} define a nonzero section of $j^*\EE(i)$ for some $i<0$. But by Lemma \ref{lem:basechange},
 $$H^0(X_k;j^*\EE(i)) \cong H^0(X;\EE(i))\tensor_R k = 0$$
 for $i<0$. 

Thus, $\OO_X \xrightarrow{s} \EE$ is an injection and $\FF$ has rank $r$ over every geometric point and is thus a vector bundle again by Lemma \ref{lem:constvector}. Thus $\FF \cong \OO(b_1)\oplus \cdots \oplus \OO(b_r)$ by induction. By shifting the sequence (\ref{align:vector}) by $(-i)$, it is easy to see that $H^0(X;\FF(-i)) = 0$ for $0<i<a_0+a_1$. Furthermore, for every $b>0$, take $i$ with $0<i\leq a_0$ and $i\equiv b \mod a_0$; then $H^0(X; \OO(b-i)) \neq 0$ because $H^0(X; \OO(\ast)) \cong R[t_0, t_1]$ with $|t_j| = a_j$. Thus, we see that $b_j\leq 0$ for all $1\leq j\leq r$. 

Therefore we get
\[\Ext^1_X(\FF, \OO_X) \cong \bigoplus_{j=1}^{r} H^1(X; \OO(-b_j)) = 0\]
by Proposition \ref{prop:projcoh}. Hence, (\ref{align:vector}) is a \emph{split} short exact sequence. 
\end{proof}

\section{Background on modular curves}\label{sec:ModularCurves}
Most of the material in this section is well-known to the experts. We will nevertheless provide some proofs and references to the literature for the convenience of the reader. 

\subsection{Basics and examples}
Denote by $\MM_{ell}$ the uncompactified moduli stack of elliptic curves and by $\MMb_{ell}$ its compactification. We define the stacks
$\MM_0(n)$, $\MM_1(n)$ and $\MM(n)$ by
\begin{align*}
 \MM_0(n)(S) &= \text{ Elliptic curves }E \text{ over }S\text{ with chosen cyclic subgroup }H\subset E(S)\text{ of order }n \\
 \MM_1(n)(S) &= \text{ Elliptic curves }E \text{ over }S\text{ with chosen point }P\in E(S)\text{ of exact order }n \\
 \MM(n)(S) &= \text{ Elliptic curves }E \text{ over }S\text{ with chosen isomorphism }(\Z/n)^2\cong E[n](S), 
\end{align*}
where we always assume $n$ to be invertible on $S$ and where $E[n]$ denotes the $n$-torsion points. More precisely, we demand for $\MM_1(n)$ that for every geometric point $s\colon \Spec K \to S$ the pullback $s^*P$ spans a cyclic subgroup of order $n$ in $E(K)$ and similarly for $\MM_0(n)$. 

We can define the compactified versions $\MMb_0(n)$, $\MMb_1(n)$ and $\MMb(n)$ as the normalization of $\MMb_{ell}$ in $\MM_0(n)$, $\MM_1(n)$ and $\MM(n)$, respectively \cite[IV.3]{D-R73}. These are all Deligne--Mumford stacks. For the corresponding modular interpretations see also \cite{Con07} and \cite{Ces15}. These moduli interpretations are based on the notion of a \emph{generalized elliptic curve}, which we will recall only over an algebraically closed field. By \cite[Lemme II.1.3]{D-R73}, a generalized elliptic curve is in this case either a (smooth) elliptic curve or a N\'eron $k$-gon. The \emph{N\'eron $k$-gon} $C$ over a scheme $S$ is the scheme quotient of $\Z/k\times \Prj^1_S$ where we identify $(i, \infty)$ with $(i+1,0)$ for all $i$. Its smooth part $C^{reg}$ is isomorphic to $\Z/k \times \G_{m,S}$. With its obvious group structure, $C^{reg}$ acts on $C$. See \cite[II.1]{D-R73} or \cite[Section 2.
1]{
Con07} for more details.

The stack $\MMb_1(n)$ classifies generalized elliptic curves $E$ with a chosen point of exact order $n$ in the smooth part of $E$ satisfying the following condition: Over every geometric point of the base scheme every irreducible component of $E$ contains a multiple of $P$. 
For $n$ \emph{squarefree} $\MMb_0(n)$ classifies generalized elliptic curves $E$ with a chosen cyclic subgroup $H$ of order $n$ in the smooth part of $E$ satisfying the analogous condition: Over every geometric point every irreducible component of $E$ intersects $H$ nontrivially. By definition, such an $H$ is \'etale locally isomorphic to the constant group scheme $\Z/n$. It follows that in the squarefree case $\MMb_0(n)$ is equivalent to the quotient of $\MMb_1(n)$ by the obvious $(\Z/n)^\times$-action. 

In particular, $\MMb_{ell}$ just classifies generalized elliptic curves, which are over an algebraically closed field either smooth or a N\'eron $1$-gon. A cubic curve of the form
$$y^2 +a_1xy + a_3y = x^3+a_2x^2+a_4x+a_6$$
is called a Weierstra{\ss} curve and defines a generalized elliptic curve if and only if certain quantities $\Delta$ and $c_4$ are nowhere vanishing \cite[Prop. III.1.4]{Sil09}.

For the universal generalized elliptic curve $p\colon \CC\to \MMb_{ell}$ define $\omega = p_*\Omega^1_{\CC/\MMb_{ell}}$, which is known to be a line bundle and actually to generate $\Pic(\MMb_{ell})$ \cite{F-O10}. There is another interpretation: Consider the moduli stack $\MMb_{ell}^1$ of generalized elliptic curves with a chosen invariant differential. Quasi-coherent sheaves on $\MMb_{ell}$ correspond to graded quasi-coherent sheaves on $\MMb_{ell}^1$ and $\omega$ corresponds to $\OO_{\MMb_{ell}^1}$ viewed as concentrated in degree $1$. 


\begin{examples}\label{exa:moduli}
 Denote by $\PP_R(a,b)$ the weighted projective stack $(\mathbb{A}^2_R -\{0,0\})/\Gm$ (as in Definition \ref{def:wps}). Then we have equivalences
 \begin{align*}
  \MMb_{ell, \Z[\frac16]} &\simeq \PP_{\Z[\frac16]}(4,6) \\
  \MMb_1(2) &\simeq \PP_{\Z[\frac12]}(2,4) \\
  \MMb_1(3) &\simeq \PP_{\Z[\frac13]}(1,3) \\
  \MMb(2) &\simeq \PP_{\Z[\frac12]}(2,2) \\
  \MMb_1(4) &\simeq \PP_{\Z[\frac12]}(1,2).
 \end{align*}
 and in each case the pullback of $\omega$ to the weighted projective line is isomorphic to $\OO(1)$. 
These are classically well-known and we obtain the corresponding uncompactified moduli by taking the non-vanishing locus of $\Delta$. Proofs of (most of) the second, third and fourth equivalence can be found, for example, in \cite[Sec 1.3]{Beh06}, \cite[Prop 4.5]{H-M15} and \cite[Prop 7.1]{Sto12} respectively. We give a sketch of the fifth one as this is probably the hardest one to find in the literature. Given an elliptic curve with a chosen invariant differential and a point $P$ of exact order $4$, we can write it uniquely in the form 
$$y^2 +a_1xy +a_3y = x^3+a_2x^2$$
such that $P = (0,0)$ and $\frac{dx}{2y+a_1x+a_3}$ is the chosen invariant differential; this is sometimes called the \emph{homogeneous Tate normal form} (see \cite[Section 4.4]{HusElliptic} or \cite[Section 1]{B-O16}). The condition that $(0,0)$ is a point of order $4$ is equivalent to $a_3 = a_1a_2$. Thus, we obtain an equivalence 
$$\MM_1(4) \simeq (\Spec \Z[\frac12][a_1,a_2, \Delta^{-1}])/\Gm$$
with $\Delta = a_1^2a_2^4(a_1^2-16a_2)$. 

The map 
$$(\Spec \Z[\frac12][a_1,a_2, \Delta^{-1}])/\Gm \to \MM_{ell, \Z[\frac12]}$$
extends to a map 
$$f\colon (\Spec \Z[\frac12][a_1,a_2])/\Gm \to \MM_{cub,\Z[\frac12]},$$
where $\MM_{cub}$ is the stack classifying all curves defined by a cubic equation \cite[Section 3.1]{Mathom} and $f$ classifies the cubic curve
$$y^2 +a_1xy +a_1a_2y = x^3+a_2x^2.$$
Let $A = \Z[\frac12][a_1,a_2,a_3,a_4,a_6]$ and consider the fpqc morphism $\Spec A \to \MM_{cub}$ classifying the universal Weierstra{\ss} curve
$$y^2 +a_1xy + a_3y = x^3+a_2x^2+a_4x+a_6.$$
Then 
$$\Spec A \times_{\MM_{cub}} (\Spec \Z[\frac12][a_1,a_2])/\Gm
\simeq \Z[\frac12][a_1,a_2][r,s,t]$$
as the morphisms of elliptic curves (preserving an invariant differential) are classified by parameters $r,s,t$ (see \cite[Section III.1]{Sil09}). Thus $f$ is representable and affine.  Using that $c_4 = a_1^4 - 16a_1^2a_2 + 16a_2^2$ it is easy to see that $c_4(a_1,a_2) = \Delta(a_1,a_2) = 0$ if and only if $a_1 = a_2 = 0$ with $a_1,a_2$ in a field of characteristic $\neq 2$. The pullback of $f$ along $\MMb_{ell,\Z[\frac12]} \to \MM_{cub,\Z[\frac12]}$ is thus a map 
$$f'\colon \PP_{\Z[\frac12]}(1,2) \to \MMb_{ell,\Z[\frac12]}.$$
Clearly, $f'$ is still affine and it is also proper by the valuative criteria \cite[Section 7]{LMB00} because the source is proper \cite[Section 2]{Mei13} and the target separated over $\Z[\frac12]$. Thus, $f'$ is finite. As $\PP_{\Z[\frac12]}(1,2)$ is normal, this implies that $\MMb_1(4) \simeq \PP_{\Z[\frac12]}(1,2)$ by the uniqueness of normal compactifications (see e.g.\ \cite[Lemma 4.4]{H-M15}). We have an isomorphism $f^*\omega \cong \OO(1)$ because $f'$ is induced by a $\Gm$-equivariant map $\A^2_{\Z[\frac12]} -\{0\} \to \MMb_{ell,\Z[\frac12]}^1$. 
\end{examples}

We call a Deligne--Mumford stack $\XX$ \emph{tame} if the automorphism group of every geometric point $\Spec \overline{K} \to \XX$ has order prime to the characteristic of $\overline{K}$. If $\XX$ is separated, it has by \cite{Con05} a coarse moduli space $X$ and we denote the canonical map $\XX \to X$ by $p$. Then $\XX$ is tame if and only if the pushforward functor
$$p_*\colon \QCoh(\XX) \to \QCoh(X)$$
is exact as proven in \cite{AOV08} (note that while they work with Artin stacks, their theory simplifies in the case of Deligne--Mumford stacks because automorphism group schemes of geometric points are in this case \'etale and hence constant). For example $\PP_R(a_0,\dots, a_n)$ is tame if and only if all $a_i$ are invertible in $R$ by \cite[Rem. 2.2]{Mei13}. In particular, all the Examples \ref{exa:moduli} are tame. 

\begin{lemma}\label{lem:tame}
Let $f\colon \XX \to \YY$ be a representable morphism into a tame Deligne--Mumford stack. Then $\XX$ is tame as well. 
\end{lemma}
\begin{proof}
Let $x\colon \Spec K \to \XX$ a geometric point and $y$ its image in $\YY$. This defines a geometric point in the pullback $\Spec K \times_{\YY}\XX$ whose (trivial) automorphism group is the kernel of $\Aut(x) \to \Aut(y)$. Thus $\Aut(x) \subset \Aut(y)$ and $\XX$ is tame. 
\end{proof}

We will mainly work with moduli stacks of elliptic curves in the tame or even representable case and specifically with the class singled out in the following convention.

\begin{convention}\label{convention}In the following, let $\Xb$ be either $\MMb_1(n)_R$ (with $n\geq 2$), $\MMb(n)_R$ (with $n\geq 2$) or $\MMb_0(n)_R$ (with $n\geq 2$ squarefree and $\phi(n)$ or $6$ invertible on $R$) over a noetherian ring $R$ where $n$ is invertible. We denote the projection $\Xb \to \MMb_{ell, R}$ by $g$ and by $\XX$ the interior $g^{-1}(\MM_{ell,R})$. 
\end{convention}

\begin{prop}\label{prop:basicprops}We have the following properties of $\Xb$ and $g$:
\begin{enumerate}
 \item The map $g$ is finite, representable and flat.
 \item The stack $\Xb$ is tame. In fact, $\MMb_1(n)_R$ (for $n\geq 5$) and $\MMb(n)_R$ (for $n\geq 3$) are even representable by projective $R$-schemes. In this case, $\XX$ is affine.
 \item The map $\XXb \to \Spec R$ is in the representable case smooth of relative dimension $1$. 
 \item For every quasi-coherent sheaf $\FF$ on $\Xb$, we have $H^i(\Xb;\FF) = 0$ for $i\geq 2$. For every quasi-coherent sheaf $\FF$ on $\XX$, we have $H^i(\XX;\FF) = 0$ for $i\geq 1$. 

\end{enumerate}
\end{prop}
\begin{proof}\lbreak
\begin{enumerate}
\item The map $g$ being integral and representable follows from the definition of normalization. By \cite[03GR]{STACKS} it is also finite because $\MMb_{ell}$ has a smooth cover by a Nagata scheme, e.g.\ by the union of the non-vanishing loci of $c_4$ and $\Delta$ in $\Spec \Z[a_1,a_2,a_3,a_4,a_6]$ . 

Furthermore, both $\Xb$ and $\MMb_{ell}$ are smooth over $\Spec R$ by Theorem 3.4 of \cite{D-R73}. Every finite map between Deligne--Mumford stacks that are smooth over $\Spec R$ is automatically flat if $R$ is regular. By choosing an \'etale cover, this follows from the affine case, which in turn follows from \cite[Prop 6.1.5]{EGAIV.2}. As the universal case $R = \Z[\frac1n]$ is regular, flatness follows for all $R$. 

\item As $\MMb_0(n)_R$ is the quotient of $\MMb_1(n)_R$ by a finite group with invertible order (in the case $\phi(n)$ is invertible) or representable over a tame stack (in the case $6$ is invertible), we only have to show tameness or representability for $\Xb = \MMb_1(n)$ and $\MMb(n)$. By the examples from \ref{exa:moduli}, we see that $\MMb_1(n)$ is tame for $2\leq n \leq 4$ and $\MMb(n)$ is tame for $n=2$.  

Next we will show that the automorphism groups of $K$-valued geometric points for $\MMb_1(n)$ for $n\geq 5$ and for $\MMb(n)$ for $n\geq 3$ are trivial. In the interior, this follows from \cite[Cor.\ 2.7.2]{K-M85}. Now consider a geometric point of $\MMb_1(n)$ not in the interior. This corresponds to a N\'eron $k$-gon with a point $P = (i,x)$ of exact order $n$ in the smooth part such that $i$ is a generator of $\Z/k$. Every automorphism of the N\'eron $k$-gon that preserves the group operation is of the form $u_{\zeta}$ or $\tau u_{\zeta}$ for a $k$-th root of unity $\zeta$ \cite[Prop.\ II.1.10]{D-R73}. On $(i,x) \in \Z/k \times \Gm(K)$ these are defined as 
$$\tau\colon (i,x) \mapsto (-i, x^{-1}) \quad \text{ and } \quad u_{\zeta}\colon (i,x) \mapsto (i, \zeta^i x).$$
As $n\geq 2$, we have $k=1$ or $i \neq [0]$ and thus $P$ cannot be fixed by $u_{\zeta}$ if $\zeta \neq 1$. If $P$ is fixed by $\tau u_{\zeta}$, then $i = -i$ and thus $k= 1$ or $2$, which implies $\zeta^2 = 1$. As $x = \zeta^ix^{-1}$, this shows that $x^4 = 1$. Thus, $P$ is a $4$-torsion point, contrary to the assumption that $n\geq 5$. Thus, we see that all automorphisms of geometric points of $\MMb_1(n)$ are trivial if $n\geq 5$. By the same arguments, the analogous statement follows for $\MMb(n)$ if $n\geq 3$ because an isomorphism $(\Z/n)^2 \cong C^{reg}[n]$ for a N\'eron $k$-gon $C$ implies that $k=n$.

By \cite[Theorem 2.2.5]{Con07} it follows that $\Xb$ is an algebraic space for $\Xb = \MMb_1(n)_R$ for $n\geq 5$ or $\MMb(n)_R$ for $n\geq 3$. The coarse moduli space of $\MMb_{ell, R}$ is $\mathbb{P}_R^1$ by \cite[VI.1]{D-R73} for $R=\Z$ and \cite[Prop 3.3.2]{Ces15} in the general case. As the map $\Xb \to \MMb_{ell, R}$ is finite, the composition $\Xb \to \mathbb{P}^1_R$ with the map $\MMb_{ell,R} \to \mathbb{P}_R^1$ is proper and quasi-finite as the map into the coarse moduli space is proper and quasi-finite \cite{Con05}. Thus, $\Xb$ is a scheme by \cite[Cor 6.16]{Knutson} and then automatically a projective scheme over $R$ as a proper and quasi-finite map of schemes is finite and hence projective.  

If $\Xb$ is representable by a scheme, then $\XX$ is as well. The coarse moduli scheme of $\MM_{ell,R}$ is $\A^1_R$ and the composition $\XX \to \MM_{ell,R} \to \A^1_R$ is finite again. Thus, $\XX$ is an affine scheme if $\Xb$ is representable. 

\item By \cite[Thm 3.4]{D-R73}, $\Xb_R$ is smooth over $\Spec R$ and clearly of relative dimension $1$. 

\item Under our assumptions, the case $\MMb_0(n)_R$ reduces to $\MMb_1(n)_R$ as follows: If $6$ is invertible, then $\MMb_{ell,R} \simeq \PP_R^1(4,6)$ itself has cohomological dimension $1$ and $\MMb_0(n)_R$ is finite over $\MMb_{ell,R}$ and thus $\MMb_0(n)_R$ has cohomological dimension $1$ as well. Thus assume that $\phi(n)$ is invertible and denote by $p$ the canonical map $\MMb_1(n)_R \to \MMb_0(n)_R$ (which is a $(\Z/n)^\times$-Galois cover). Furthermore, let $\FF$ be a quasi-coherent sheaf on $\MMb_0(n)_R$. In this case, the descent spectral sequence
$$H^j((\Z/n)^\times, H^i(\MMb_1(n); p^*\FF)) \Rightarrow H^{i+j}(\MMb_0(n);\FF)$$
collapses to isomorphisms 
$$H^i(\MMb_0(n)_R; \FF) \cong H^i(\MMb_1(n); p^*\FF)^{(\Z/n)^\times}.$$


The cases $\MMb_1(n)_R$ and $\MMb(n)_R$ are either treated in the Examples \ref{exa:moduli} (where one clearly has cohomological dimension $1$) or are representable. In the latter case, our statement for $\XXb$ follows from the item 3 (e.g.\ by reducing via \cite[Prop 9.3]{Har77} to the case of $R$ being a field). 

Similarly, we can reduce the case $\XX$ to $\MM_1(n)_R$ and $\MM(n)_R$. In the representable case, these are affine. The Examples \ref{exa:moduli} can be treated by hand again.

\qedhere
\end{enumerate}
\end{proof}

\begin{example}\label{exa:5-12}
 For $n=5,\dots, 10$ or $n=12$, we have an equivalence $\MMb_1(n) \simeq \Prj^1_{\Z[\frac1n]}$. Indeed, by the last proposition, $\MMb_1(n)$ is representable by a projective $\Z[\frac1n]$-scheme. Over $\C$, the scheme $\MMb_1(n)$ is connected of genus zero (for the genus formula see for example \cite[Section 3.9]{D-S05}). As in the discussion in \cite[Section 3.3]{H-L10b}, this implies that $\MMb_1(n) \simeq \Prj^1_{\Z[\frac1n]}$ as soon as we have exhibited a $\Q$-valued point of $\MMb_1(n)$. This is easily done as a N\'eron $n$-gon with $\Gamma_1(n)$-level structure already exists over $\Q$. 
\end{example}

\begin{lemma}\label{lem:sectionsirreducible}
 Let $f\colon X \to S$ be a smooth proper morphism with geometrically connected fibers and $S$ locally noetherian. Then $\OO_S \to f_*\OO_X$ is an isomorphism. 
\end{lemma}
\begin{proof}
 The case of $S = \Spec k$ for a field $k$ is \cite[Cor 3.21]{Liu}. We can assume that $S = \Spec R$ affine and noetherian. By cohomology and base change (\cite[Sec 5, Cor 2]{Mum08}), we see that 
 $$H^0(X,\OO_X)\tensor_R k \to H^0(X_k, \OO_{X,k})\cong k$$
 is an isomorphism for every $\Spec k \to S$. As $H^0(X,\OO_X)$ is a finitely generated $R$-module, this implies that the canonical map $R \to H^0(X,\OO_X)$ is an isomorphism. 
\end{proof}

\begin{prop}\label{prop:irreducible}
Define $\MMb(n)_{R,\zeta}$ to be the stack of generalized elliptic curves with a level-$n$-structure over $R$-schemes, where the Weil pairing of the generators of the level structure is a chosen primitive $n$-th root of unity $\zeta\in R$. 
\begin{enumerate}
 \item For every field $k$ and $n\geq 5$ the scheme $\MMb_1(n)_k$ is geometrically irreducible. 
 \item For every field $k$ with a primitive $n$-th root of unity $\zeta \in k$ and $n\geq 3$ the scheme $\MMb(n)_{k,\zeta}$ is geometrically irreducible.
 \item We have 
 $$H^0(\MMb_0(n)_R, \OO_{\MMb_0(n)_R}) \cong H^0(\MMb_1(n)_R, \OO_{\MMb_1(n)_R}) \cong R$$
 and
 $$H^0(\MMb(n)_R, \OO_{\MMb(n)_R}) \cong R[\zeta_n] = R[x]/\Phi_n(x)$$
 for all $n$ (squarefree for $\MMb_0(n)$) and noetherian $R$. Here, $\Phi_n(x)$ denotes the $n$-th cyclotomic polynomial. 
\end{enumerate}
\end{prop}
\begin{proof}
 By \cite[038H]{STACKS}, $\MMb_1(n)_k$ is geometrically irreducible for all $k$ if it is irreducible for $k = \overline{\F}_p$ (for all primes $p$) and for $k = \C$ and likewise for $\MM(n)_{k,\zeta}$. They are irreducible for $k = \C$ because they can be uniformized by the upper half plane and they are thus smooth and connected in the complex topology. 
 
 By Proposition \ref{prop:basicprops}, $\MMb_1(n)$ and $\MMb(n)_{\Z[\frac1n,\zeta_n],\zeta_n}$ are smooth and proper over their base schemes. Thus by \cite[Thm 4.17]{DeligneMumford}, the schemes are thus also irreducible over $\overline{\mathbb{F}}_p$ and thus the first two items are proved. 
 
 By Lemma \ref{lem:sectionsirreducible}, it follows that the inclusion $R \to H^0(\MMb_1(n)_R, \OO_{\MMb_1(n)_R})$ of constant functions is an isomorphism for $n\geq 5$ and the same is true for $n=1,\dots, 4$ by Examples \ref{exa:moduli}. Because of the isomorphism
 $$H^0(\MMb_0(n)_R, \OO_{\MMb_0(n)_R}) \cong H^0(\MMb_1(n)_R, \OO_{\MMb_1(n)_R})^{(\Z/n)^\times}$$
 our claim follows for $n$ squarefree also for $\MMb_0(n)_R$. 
 
 The Weil pairing defines a morphism $\MMb(n)_R \to \Spec R[\zeta_n]$, which induces a morphism $R[\zeta_n] \to H^0(\MMb(n)_R, \OO_{\MMb(n)_R})$. We claim that this morphism is an isomorphism after $\tensor_R R[\zeta_n]$ and thus an isomorphism itself. Indeed, $\MMb(n)_{R[\zeta_n]} \cong (\Z/n)^\times \times \MMb(n)_{R[\zeta_n],\zeta_n}$ and $H^0(\MMb(n)_{R[\zeta_n],\zeta_n}, \OO_{\MMb(n)_{R[\zeta_n],\zeta_n}}) \cong R[\zeta_n]$ by Lemma \ref{lem:sectionsirreducible}. 
 \end{proof}

\subsection{Line bundles}
 In this subsection, we will show how to compare line bundles in an integral context by just comparing them over  $\C$. Furthermore, we recall the definition of cusp forms. 

\begin{prop}\label{prop:liu}
 Let $A$ be a discrete valuation ring or a field and $X \to \Spec A$ a proper morphism (with geometrically integral fibers). Let $L$ be a field extension of the field of fractions $K$ of $A$ and assume $K$ is perfect. Then $\Pic(X) \to \Pic(X_L)$ is injective. 
\end{prop}
\begin{proof}
 The injectivity of $\Pic(X) \to \Pic(X_K)$ is proven in (the proof of) \cite[Thm 3, Section 8.4]{BLR90}. See also \cite{mathoverflowLiu} for a version where the fibers are not assumed to be geometrically integral. 
 
 In general, we can assume that $L$ is algebraically closed. For $L=\overline{K}$ the algebraic closure of $K$, we can argue as follows: Consider the descent spectral sequence for \'etale cohomology
 $$H^p_{\cont}(\Gal(\overline{K}|K), H^q_{\text{\'et}}(X_{\overline{K}},\Gm)) \Rightarrow H_{\text{\'et}}^{p+q}(X_K, \Gm). $$
 We know that $H^1_{\cont}(\Gal(\overline{K}|K), H^0_{\text{\'et}}(X_{\overline{K}};\Gm)) \cong H^1_{\cont}(\Gal(\overline{K}|K), \overline{K}^\times) = 0$ by Hilbert 90. (Here we use $H^0(X_{\overline{K}}; \OO_{X_{\overline{K}}}) \cong \overline{K}$ because $X \to \Spec A$ is proper \cite[Prop 3.18]{Liu}.) Thus, $\Pic(X_K) \cong H^1_{\text{\'et}}(X_K, \Gm)$ is the subgroup of permanent cycles in the invariants of $\Pic(X_{\overline{K}}) \cong H^1_{\text{\'et}}(X_{\overline{K}}, \Gm)$.
 
 For the general case assume that $\LL \in \Pic(X_{\overline{K}})$ is send to $0$ in $\Pic(X_L)$. We will now argue via a version of the Lefschetz principle. There is a trivialization $\OO_{X_L} \xrightarrow{\cong} \LL_L$. Cover $X_{\overline{K}}$ with affine charts $U_1 = \Spec A_1,\dots, U_n = \Spec A_n$ where $\LL$ is trivial. Then the trivialization of $\LL_L$ is determined in these charts by invertible elements $\sum a_j\tensor t_j \in A_i \tensor_{\overline{K}} L$ with inverses $\sum b_j \tensor s_j$. There is a finitely generated ring extension $\overline{K} \subset R$ with $R\subset L$ such that all $s_j, t_j$ are in $R$. Thus, $\LL_R \in \Pic(X_R)$ is already trivial. By Hilbert's Nullstellensatz, $\Spec R$ has a $\overline{K}$-point. Thus, $X_R \to X_{\overline{K}}$ has a section. This implies that $\LL$ itself is already trivial in $\Pic(X_{\overline{K}})$.  
\end{proof}
We can use this to obtain the following well-known corollary (for example stated in \cite{Kat73}).
\begin{cor}\label{cor:Omegaomega}
Let $\Gamma$ be $\Gamma_1(n), \Gamma_0(n)$ or $\Gamma(n)$. 
 Denote by $f_n\colon \MMb(\Gamma)_R \to \MMb_{ell,R}$ the canonical map for $R$ a $\Z_{(l)}$-algebra with $l$ not dividing $n$. Then 
 $$(f_n)^*\omega^{\tensor 2} \cong \Omega^1_{\MMb(\Gamma)_R/\Spec R} \tensor \LL(\mathrm{cusps}),$$
 where $\mathrm{cusps}$ is the complement of $\MM(\Gamma)_R$ in $\MMb(\Gamma)_R$ and $\LL(\mathrm{cusps})$ denotes the line bundle associated to this divisor.
\end{cor}
\begin{proof}
 There is a holomorphic isomorphism $\MM(\Gamma)_{\C}$ to the orbifold quotient $\mathbb{H}/\Gamma$ for $\mathbb{H}$ the upper half plane; this extends (for example by the uniqueness of normal compactifications) to an equivalence $\MMb(\Gamma)_{\C}$ to the usual compactification $\overline{\mathbb{H}/\Gamma}$. Using this equivalence, Deligne constructs in Section 2 of \cite{Del71} a holomorphic map $\Omega^1_{\MMb(\Gamma)/\C} \to (f_n)^*\omega^{\tensor 2}_{\C}$, which has simple zeros precisely at the cusps. As $\MMb(\Gamma)_{\C}$ has a finite flat cover by a complex projective algebraic variety (by Proposition \ref{prop:basicprops}), we can apply Serre's GAGA and descent to see that the holomorphic map above is actually algebraic. This implies the result over $\C$ and thus by the last proposition over $\Z_{(l)}$ and thus for any $\Z_{(l)}$-algebra.
\end{proof}
We can use this to get information about cusp forms in the following sense:
\begin{defi}
 For $\Gamma \in \{\Gamma_1(n), \Gamma_0(n),\Gamma(n)\}$, a \emph{$\Gamma$-cusp form} with coefficients in $R$ and weight $i$ is a section of $\LL(-\mathrm{cusps}) \tensor (f_n)^*\omega^{\tensor i}$ on $\MMb(\Gamma)_R$, i.e.\ a section of $(f_n)^*\omega^{\tensor i}$ that vanishes at the cusps.
\end{defi}
\begin{cor}\label{cor:cusps}
 For $\Gamma \in \{\Gamma_1(n), \Gamma_0(n),\Gamma(n)\}$, the space of $\Gamma$-cusp form with coefficients in $R$ and weight $i$ is isomorphic to 
 $$H^0(\MMb(\Gamma)_R; \Omega^1_{\MMb(\Gamma)_R/\Spec R} \tensor (f_n)^*\omega^{\tensor (i-2)}).$$
 If $\MMb(\Gamma)_R$ is representable, this in turn is isomorphic to $\Hom_R(H^1(\MMb(\Gamma)_R; (f_n)^*\omega^{\tensor (2-i)}), R)$ by Grothendieck duality. 
\end{cor}
This is in accordance with the definition given in \cite[Def 2.8]{Del71}. We end with another proposition about $\omega$. 
\begin{prop}
Let $\Gamma$ be $\Gamma_1(n)$ or $\Gamma(n)$ and assume that $\MMb(\Gamma)_R$ is representable. Then the line bundle $(f_n)^*\omega$ is ample on $\MMb(\Gamma)_R$. 
\end{prop}
\begin{proof}
Denote the map $\MMb_{ell,R} \to \mathbb{P}_R^1$ to the coarse moduli space by $\pi$. We claim first that $\pi^*\OO(1) \cong \omega^{\tensor 12}$. Because $\Pic(\MMb_{ell}) \cong \Z$ is generated by $\omega$ \cite{F-O10}, we just have to show which $\omega^{\tensor m}$ the pullback $\pi^*\OO(1)$ is isomorphic to and we can do it over $\C$. By Example \ref{exa:5-12}, we know that $\MMb_1(5)_{\C} \simeq \mathbb{P}^1_{\C}$ and by \cite[Sec.\ 3.8+3.9]{D-S05} we know that the composition 
$$\MMb_1(5)_{\C} \xrightarrow{g} \MMb_{ell,\C} \xrightarrow{\pi} \mathbb{P}_{\C}^1$$
 has degree $12$. Thus, $(\pi g)^*\OO(1) \cong \OO(12)$. As there are modular forms for $\Gamma_1(5)$, the line bundle $g^*\omega$ must have positive degree. Thus, it follows that $\pi^*\OO(1) \cong \omega^{\tensor 12}$.

By Proposition \ref{prop:basicprops}, the composition $\pi \colon \MMb(\Gamma)_R \to \mathbb{P}_R^1$ is finite as $\MMb(\Gamma)_R \to \MMb_{ell,R}$ is finite and $\MMb_{ell,R} \to \mathbb{P}_R^1$ is quasi-finite and projective. Thus, $f_n^*\omega^{\tensor 12} \cong (\pi f_n)^*\OO(1)$ is ample and thus $f_n^*\omega$ is ample as well (see \cite[Prop 13.83]{G-W10}). 
\end{proof}

\subsection{Cohomology}
In this subsection, we will collect some information about the cohomology of $g^*\omega^{\tensor m}$ on $\Xb$ and of $\omega^{\tensor m}$ on $\MMb_{ell}$. Here, $\Xb$ and $g$ are as in Convention \ref{convention}.

\begin{prop}\label{prop:coh}
 We have
 \begin{enumerate}
  \item \label{cor:nonegative} $H^0(\Xb;g^*\omega^{\tensor m}) = 0$ for $m<0$ (i.e.\ there are no modular forms of negative weight),
  \item \label{lem:vanishingcoh}  $H^1(\Xb;g^*\omega^{\tensor m}) = 0$ for $m\geq 2$,
  \item \label{lem:torsionfree} $H^1(\Xb; g^*\omega^{\tensor m})$ is torsionfree for all $m\neq 1$ if $R$ is torsionfree. 
 \end{enumerate}
\end{prop}
\begin{proof}Throughout the proof, we will write $\omega$ for $g^*\omega$ when the context is clear. 
\begin{enumerate}
 \item  As $H^0(\MMb_0(n)_R; \omega^{\tensor m}) \cong H^0(\MMb_1(n)_R; \omega^{\tensor m})^{(\Z/n)^\times}$, we only have to deal with $\MMb_1(n)_R$ and $\MMb(n)_R$. The non-representable cases from Example \ref{exa:moduli} can be dealt with by hand. 

 Assume now that $\Xb$ is representable. By \cite[VI.4.4]{D-R73}, the line bundle $\omega$ on $\MMb_{ell}$ has degree $\frac1{24}$. Thus, $\omega^{\tensor m}$ has negative degree on $\Xb_k$ for $m<0$ and every field $k$. Thus, $H^0(\Xb_k;\omega^{\tensor m}_k) = 0$ by \cite[IV.1.2]{Har77}. Thus, the pushforward of $g^*\omega^{\tensor m}$ to $\Spec R$ vanishes at every point of $\Spec R$ and thus vanishes completely. 
 
 \item This is shown for $\MMb(n)_R$ in \cite[Thm 1.7.1]{Kat73}. We will give the proof in the case of $\MMb_1(n)_R$ to add some details. 
 
 By dealing with the cases $n\leq 4$ by hand, we can assume again that $\MMb_1(n)_R$ is representable by a projective $R$-scheme. By cohomology and base change (see e.g.\ \cite[Theorem 12.11]{Har77}), we see that it is enough to show the claim in the case where $k= R$ is an algebraically closed field. By Corollary \ref{cor:Omegaomega},
 $$\omega^{\tensor m} \cong \Omega^1_{\MMb_1(n)_k/k} \tensor \omega^{\tensor m-2} \tensor \LL(\mathrm{cusps}).$$
 Because $\omega$ has positive degree \cite[VI.4.4]{D-R73} and $m\geq 2$, we see that the degree of $\omega^{\tensor m}$ is bigger than the degree of $\Omega^1_{\MMb_1(n)_k/k}$. By Serre duality
 $$H^1(\MMb_1(n)_k; \omega^{\tensor m}) \cong H^0(\MMb_1(n)_k;\omega^{\tensor -m} \tensor \Omega^1_{\MMb_1(n)_k/k})$$
 and this vanishes as $\omega^{\tensor -m} \tensor \Omega^1_{\MMb_1(n)_k/k}$ has negative degree. 
 
It remains to prove the claim for $\MMb_0(n)_R$ if $\phi(n)$ or $6$ is invertible on $R$. As we already know that $H^1(\MMb_1(n)_R;\omega^{\tensor m}) = 0$ for $m\geq 2$, it follows by the descent spectral sequence that 
$$H^i(\MMb_0(n)_R; \omega^{\tensor m}) \cong H^i((\Z/n)^\times; H^0(\MMb_1(n)_R; \omega^{\tensor m})).$$
Begininng with $i$, this is periodic in $i$. But we know that these cohomology groups vanish for $i>1$ by Proposition \ref{prop:basicprops}. Thus, they have to vanish for $i=1$ as well. 
 
 \item Let $l$ be a prime that does not divide $n$. Consider the short exact sequence
 \[0 \to H^0(\Xb; \omega^{\tensor m})/l \to H^0(\Xb_{R/l};\omega^{\tensor m}) \to \Tors_l H^1(\Xb;\omega^{\tensor m}) \to 0.\]
First note that the middle group is zero for $m<0$ by Item (\ref{cor:nonegative}) and hence also  $$\Tors_l H^1(\Xb;\omega^{\tensor m}) = 0$$
for $m < 0$. 

The morphism $H^0(\Xb; \omega^{\tensor m})/l \to H^0(\Xb_{R/l};\omega^{\tensor m})$ is an isomorphism for $m = 0$ by Proposition \ref{prop:irreducible}. Thus, $H^1(\Xb;\omega^{\tensor m})$ can only have torsion for $m=1$ as it vanishes for $m\geq 2$.\qedhere
\end{enumerate}
\end{proof}

At last, we have to collect some facts about the cohomology of $\MMb_{ell}$ itself, which is certainly not tame if we do not invert $6$. The cohomology of the sheaves $\omega^{\tensor m}$ on $\MMb_{ell}$ was computed by \cite{Konter}, based on \cite{Bau08}. We need essentially only the following.
\begin{prop}\label{prop:cohmell}We have isomorphisms
 \begin{align*}\
  H^1(\MMb_{ell}; \omega) &\cong \Z/2 \cdot \eta, \\
 H^1(\MMb_{ell};\omega^{\tensor 2}) & \cong \Z/12 \cdot \nu.
 \end{align*}
\end{prop}
These classes have a rather classical description as obstructions to lift the Hasse invariant. Indeed, denote by 
$$A_p \in H^0(\MMb_{ell,\F_p}; \omega^{\tensor (p-1)}) \cong H^0(\MMb_{ell,\Z_{(p)}}; \omega^{\tensor (p-1)}/p) $$
the mod $p$ Hasse invariant (see Appendix \ref{app:Hasse} for a definition). The short exact sequence
$$0 \to \omega^{\tensor (p-1)} \xrightarrow{p} \omega^{\tensor (p-1)} \to \omega^{\tensor (p-1)}/p \to 0$$
on $\MMb_{ell}$ induces a long exact sequence
$$\cdots \to H^0(\MMb_{ell,\Z_{(p)}};\omega^{\tensor (p-1)}) \to H^0(\MMb_{ell,\Z_{(p)}};\omega^{\tensor (p-1)}/p) \xrightarrow{\partial} H^1(\MMb_{ell,\Z_{(p)}};\omega^{\tensor (p-1)}) \to \cdots.$$
Because the $H^1$-term vanishes for $p>3$, there is no obstruction to lift $A_p$ to characteristic zero for $p>3$. As the Hasse invariant does not lift to characteristic zero for $p=2,3$ (there does not even exist a nonzero integral modular form in these degrees), we must have $\partial(A_2) = \eta$ and $\partial(A_3)$ at least a nonzero multiple of $\nu$.

\begin{remark}
 There is also a topological interpretation of the classes $\eta$ and $\nu$. These cohomology classes detect the Hurewicz images of the Hopf maps of the same name in $\pi_*Tmf$ in the descent spectral sequence (see \cite{Bau08} for a related statement). 
\end{remark}

\begin{prop}\label{prop:eta}
 The image of $\eta$ in $H^1(\Xb; g^*\omega)$ is zero. 
\end{prop}
\begin{proof}
 It is enough to show $g^*\eta = 0$ for $\Xb = \MMb_1(n)_{(2)}$ and $n$ odd. Indeed, consider the composite
 \[H^1(\MMb_{ell,\Z_{2)}};\omega) \to H^1(\MMb_0(n)_{\Z_{2)}}; g^*\omega) \to H^1(\MMb_1(n)_{\Z_{2)}};g^*\omega) \to H^1(\MMb(n)_{\Z_{2)}};g^*\omega).
 \]
 Now we only have to use that the second map is actually an injection (namely the inclusion of the $(\Z/n)^\times$-fixed points by the proof of Proposition \ref{prop:basicprops}). 
 
It is enough to show $g^*\eta = 0$ after base change to $C= \Z_{(2)}[\zeta]$ for a $\zeta$ a $2^m$-th root of unity. Consider the commutative square
 \[
  \xymatrix{
  0 = H^0(\MMb_{ell,C}; \omega) \ar[r]\ar[d]^{g^*} & H^0(\MMb_{ell,C}; \omega/2) \ar[r]^-{\partial} \ar[d]^{g^*} & H^1(\MMb_{ell,C}; \omega) \ar[d]^{g^*} \\
   H^0(\MMb_1(n)_C; g^*\omega) \ar[r] & H^0(\MMb_1(n)_C; g^*\omega/2) \ar[r]^-{\partial_n}  & H^1(\MMb_1(n)_C; g^*\omega)
  }
 \]
 As $\eta = \partial(A_2)$ is still true over $C$, it is enough to show that $\partial_ng^*A_2 = 0$, i.e.\ that $g^*A_2$ lifts to $H^0(\MMb_1(n)_C; g^*\omega)$. This is exactly the content of Proposition \ref{prop:Hasse} for suitable $\zeta$.
\end{proof}

\section{The existence of decompositions}\label{sec:decexistence}
In this section, we want to prove Theorem \ref{thm:int1}. We will first do it over a field. Then we will use deformation theory to deduce it over the formal spectrum of $\Z_l$ in the case of uncompactified moduli of elliptic curves. Afterwards, we will use our results on vector bundles on weighted projective lines to deduce it also on compactified moduli under certain conditions. 
\subsection{Working over a field}\label{sec:workingfield}
We fix an integer $l$ that is either a prime or $0$ and localize throughout this section implicitly at the ideal $(l)$. Denote by $\MMb'$ for $l=3$ the moduli stack $\MMb_1(2)$, for $l=2$ the moduli stack $\MMb_1(3)$ and for $l\neq 2,3$ the moduli stack $\MMb_{ell}$ itself. We will denote the canonical morphism $\MMb' \to \MMb_{ell}$ by $f$. We denote by $\MM'$ correspondingly $\MM_1(2)$, $\MM_1(3)$ or $\MM_{ell}$. Recall furthermore the notation from Convention \ref{convention}. Our goal is to show the following proposition.

\begin{prop}\label{prop:split}
Let $k$ be a field of characteristic $l$ and let $\EE$ be a vector bundle on $\Xb_k$ with $\Xb_k$ as in Convention \ref{convention}. Then $g_*\EE$ decomposes into a direct sum of vector bundles of the form $f_*\OO_{\MMb_k'} \tensor \omega^{\tensor i}$. 
\end{prop}
The proof for $l\neq 2,3$ is very easy. In this case, we have an equivalence $\MMb_{ell,k} \simeq \PP_k(4,6)$ with $\omega$ corresponding to $\OO(1)$. By Proposition \ref{prop:ClassField}, \emph{every} vector bundle on $\PP_k(4,6)$ decomposes into a sum of the line bundles $\OO(i)$. Thus, we will assume $l=2$ or $3$ in the following. 

We will often use the following standard lemma:
\begin{lemma}\label{lem:pushpull}
 Let 
 \[
  \xymatrix{X' \ar[r]^{v}\ar[d]^{q} & X \ar[d]^p \\
  Y' \ar[r]^u & Y }
 \]
 be a cartesian diagram of Deligne--Mumford stacks, where $p$ is representable and affine, and let $\FF$ be a quasi-coherent sheaf on $X$. Then the natural map 
 $$u^*p_*\FF \to q_*v^*\FF$$
 is an isomorphism.
\end{lemma}
\begin{proof}
Note first that the lemma is true without any assumptions on $p$ if $u$ is \'etale because on $U\to Y'$ \'etale, both source and target can be identified with $\FF(U\times_Y X)$. Thus we can work \'etale locally and assume that $Y = \Spec A$ and $Y' = \Spec A'$ are affine schemes and hence also $X = \Spec B$ and $X' = \Spec B\tensor_{A}A'$. If $\FF$ corresponds to the $B$-module $M$, our assertion just becomes
\[M\tensor_A A' \cong M\tensor_B (B\tensor_A A').\qedhere\]
\end{proof}

Our strategy will be to analyze the consequences of this lemma in the case of the pullback square
 \[\xymatrix{
  \Yb = \Xb_k \times_{\MMb_{ell,k}} \MMb'_k \ar[r]^-{g'} \ar[d]^{f'} & \MMb'_k \ar[d]^f \\
  \Xb_k \ar[r]^{g} &  \MMb_{ell,k}
 }\]
 using a Krull--Schmidt theorem and the following indecomposability statement.

\begin{lemma}\label{lem:indec}
 The vector bundle $(f_S)_*\OO_{\MM'_S}$ is indecomposable for $S$ an arbitrary $\Z_{(l)}$-scheme admitting a morphism $\Spec \Fb_l \to S$ (if $l\leq 3$). 
\end{lemma}
\begin{proof}
 Consider the case $l=2$ and the elliptic curve $E\colon y^2 +y = x^3$ over $\overline{\F}_2$. This has, according to \cite{Sil09}, III.10.1, automorphism group $G$ of order $24$.  By \cite[2.7.2]{K-M85}, the morphism $G\to GL_2(\F_3)$ (given by the operation of $S$ on $E[3]$) is injective. Using elementary group theory, $GL_2(\F_3)$ has a unique subgroup of order $24$, namely $SL_2(\F_3)$; thus $G$ embeds onto $SL_2(\F_3)$. This induces a map $\Spec \Fb_2 / SL_2(\F_3) \to \MM_{ell}$. 
 
 By Lemma \ref{lem:pushpull}, pulling $f_*\OO_{\MM'}$ back along this map gives an $8$-dimensional $\Fb_2$-vector space $V$ with $SL_2(\F_3)$-action; this is isomorphic to the permutation representation defined by the action of $SL_2(\F_3)$ on the eight points of exact order $3$ in $\F_3^2$. The quaternion subgroup $Q\subset SL_2(\F_3)$ acts freely and transitively on these points; thus, $V$ restricted to $Q$ is isomorphic to the regular representation of $Q$. 
 
 The regular representation $K[P]$ of a finite $p$-group $P$ over a field $K$ of characteristic $p$ is always indecomposable. Indeed, $K[P]^P \cong K$, but every $K$-representation of $P$ has a 
nonzero $P$-fixed vector by \cite[Prop 26]{Ser77}. Thus every nontrivial decomposition of $K[P]$ would yield that $\dim K[P]^P \geq 2$. 

Thus, also $(f_{\Fb_2})_*\OO_{\MM'_{\Fb_2}}$ is indecomposable. The pullback of $(f_S)_*\OO_{\MM'_S}$ to $\MM_{ell,{\Fb_2}}$ agrees with $(f_{\Fb_2})_*\OO_{\MM'_{\Fb_2}}$ by the last lemma; thus we see that also the vector bundle $(f_S)_*\OO_{\MM'_S}$ is indecomposable. 
 
 For $l = 3$, we can either do an analogous argument or cite \cite[Cor 4.8]{Mei13}. 
\end{proof}

\begin{prop}[Krull--Schmidt]\label{prop:Krull-Schmidt}
 Let $\YY$ be a proper Artin stack over a field $k$. Then the Krull--Schmidt theorem holds for coherent sheaves on $\YY$. This means that every coherent sheaf on $\YY$ decomposes into finitely many indecomposables and that this decomposition is unique up to permutation of the summands.
\end{prop}
\begin{proof}
 As shown by Atiyah in \cite{Ati56}, a $k$-linear abelian category has a Krull--Schmidt theorem if all $\Hom$-vector spaces are finite dimensional. By a theorem of Faltings \cite{Fal03}, the global sections of any coherent sheaf on $\YY$ form a finite-dimensional $k$-vector space. We can apply this to the Hom-sheaf $\mathcal{H}om_{\OO_\YY}(\FF,\GG)$ for two coherent $\OO_\YY$-modules, which is coherent itself.
\end{proof}

\begin{example}\label{exa:KS}
 We give a counterexample to the Krull--Schmidt theorem if we do not assume properness. It follows from \cite[Prop 4.14]{Mathom} that $f_*\OO_{\MM_1(2)}$ splits as $\OO_{\MM_{ell}} \oplus \omega^{\tensor 2}\oplus \omega^{\tensor 4}$ after rationalization. By \cite[Sec 1.3]{Beh06}, there is an equivalence 
 $$\MM_1(2) \simeq \Spec \Z[\tfrac12][b_2,b_4][\Delta^{-1}]/\G_m,$$ 
 where $\Delta$ is a polynomial in $b_2$ and $b_4$. As $\Delta$ is divisible by $b_4$, the ring $\Z[\tfrac12][b_2,b_4][\Delta^{-1}]$ is $4$-periodic and we deduce that $f^*\omega^{\tensor 4} \cong \OO_{\MM_1(2)}$. In particular, it follows that 
$$\OO_{\MM_{ell}} \oplus \omega^{\tensor 2}\oplus \omega^{\tensor 4} \cong f_*\OO_{\MM_1(2)} \cong f_*\OO_{\MM_1(2)} \tensor \omega^{\tensor 4} \cong \omega^{\tensor 4}\oplus \omega^{\tensor 6} \oplus \omega^{\tensor 8}$$
on $\MM_{ell,\Q}$, contradicting a possible Krull--Schmidt theorem in the uncompactified case. With some extra work one can remove the summand $\omega^{\tensor 4}$ from both sides. 
\end{example}

\begin{lemma}\label{lem:splitting}
The vector bundle $g^*f_*\OO_{\MMb'}$ is a direct summand of a sum of line bundles of the form $g^*\omega^{\tensor n}$. For $l=3$ or after base-changing to a field, $g^*f_*\OO_{\MMb'}$ is the sum of such line bundles itself.
\end{lemma}
\begin{proof}
 We will first do the case $l=3$. By \cite[Prop 4.14]{Mathom}, there are extensions
  \[ 0 \to \OO_{\MMb_{ell}} \to f_*\OO_{\MMb_1(2)} \to E_{\nu} \tensor \omega^{\tensor (-2)} \to 0\]
  and 
  \[0 \to \OO_{\MMb_{ell}} \to E_\nu \to \omega^{\tensor (-2)} \to 0\]
  classified by $\tilde{\nu} \in \Ext^1(E_\nu \tensor \omega^{\tensor (-2)}, \OO_{\MMb_{ell}})$ and $\nu \in \Ext^1(\omega^{\tensor (-2)},\OO_{\MMb_{ell}})$ respectively, where $\nu$ is as in Proposition \ref{prop:cohmell} and $\tilde{\nu}$ is a lift of $\nu$. The classes $g^*\nu$ and $g^*\tilde{\nu}$ are zero by Proposition \ref{prop:coh}. Thus $g^*f_*\OO_{\MMb_1(2)}$ splits into line bundles of the form $g^*\omega^{\tensor n}$. 
  
  The same argument -- only more complicated -- works at the prime $2$. In \cite[Section 4.1]{Mathom}, Mathew constructs a vector bundle $\FF(Q)$ on $\MMb_{ell}$, which arises via two short exact sequences
  $$ 0 \to E_{\eta} \to \FF(Q) \to \omega^{\tensor (-3)} \to 0$$
  and
  $$ 0\to \OO_{\MMb_{ell}} \to E_{\eta} \to \omega^{\tensor (-1)} \to 0.$$
  The latter extension is classified by $\eta \in H^1(\MMb_{ell}; \omega) \cong \Ext^1(\omega^{\tensor (-1)}, \OO_{\MMb_{ell}})$, while the former is classified by a lift $\tilde{\nu}\in\Ext^1(\omega^{\tensor (-3)}, E_\eta)$ of $\nu$. By Proposition \ref{prop:eta}, we know that $g^*\eta = 0$ so that $g^*E_{\eta} \cong \OO_{\XXb}\oplus g^*\omega^{\tensor(-1)}$. The other extension splits as well because by Proposition \ref{prop:coh} we see that $\Ext^1(g^*\omega^{\tensor (-3)}, g^*E_{\eta}) = 0$. Thus, $$g^*\FF(Q) \cong \OO_{\XXb}\oplus g^*\omega^{\tensor(-1)} \oplus g^*\omega^{\tensor (-3)}.$$ 
  
    By \cite[Prop 4.7, Cor 4.11]{Mathom}, the cokernel of the coevaluation map 
    $$q\colon \OO_{\MMb_{ell}} \to \FF(Q) \tensor \FF(Q)^{\vee}$$
    is isomorphic to $f_*\OO_{\MMb_1(3)}$.\footnote{Topologically, this corresponds to the fact that $DA(1)\sm Tmf \simeq Tmf_1(3)$ for an $8$-cell complex $DA(1)$.} The composition of $q$ with the evaluation map $\FF(Q) \tensor \FF(Q)^{\vee} \to \OO_{\MMb_{ell}}$ equals multiplication by $3 = \rk(\FF(Q))$, which is invertible. Hence, $f_*\OO_{\MMb_1(3)}$ splits off $\FF(Q) \tensor \FF(Q)^{\vee}$ as a direct summand. Thus, $g^*f_*\OO_{\MMb_1(3)}$ is a direct summand of a sum of line bundles of the form $g^*\omega^{\tensor i}$. After base-change to a field, Krull--Schmidt implies that $g^*f_*\OO_{\MMb_1(3)}$ itself splits into line bundles of the form $g^*\omega^{\tensor i}$.
\end{proof}

\begin{proof}[Proof of Proposition \ref{prop:split}:]
 Consider again the pullback diagram
 \[\xymatrix{
  \Yb = \Xb_k \times_{\MMb_{ell,k}} \MMb'_k \ar[r]^-{g'} \ar[d]^{f'} & \MMb'_k \ar[d]^f \\
  \Xb_k \ar[r]^{g} &  \MMb_{ell,k}.
 }\]
 We have an isomorphism
$$g_*f'_*(f')^*\EE \cong f_*g'_* (f')^*\EE.$$
 Every vector bundle on $\MMb'_k$ decomposes into line bundles of the form $f^*\omega^{\tensor ?}$ by the Examples \ref{exa:moduli} and Proposition \ref{prop:ClassField}. By the projection formula, $f_*g'_* (f')^*\EE$ is thus a sum of vector bundles of the form $f_*\OO_{\MMb'_k} \tensor \omega^{\tensor ?}$. As these are indecomposable by Lemma \ref{lem:indec}, this is also the decomposition of $g_*f'_*(f')^*\EE$ into indecomposables. By the Krull--Schmidt theorem it is thus enough to show that $g_*\EE \tensor \omega^{\tensor n}$ is a summand of $g_*f'_*(f')^*\EE$ for some $n\in\Z$. 
 
 Note first that by the last lemma and Lemma \ref{lem:pushpull}, we can write $g^*f_*\OO_{\MMb_k'} \cong f'_*\OO_{\Yb}$ as $\bigoplus_i g^*\omega^{\tensor n_i}$. By the projection formula, we have a chain of isomorphisms
 \begin{align*}
  g_*(f')_*(f')^*\EE &\cong g_*((f')_*\OO_{\Yb}\tensor \EE) \\
		   &\cong g_*(\bigoplus_i g^*\omega^{\tensor n_i} \tensor \EE) \\
		   &\cong \bigoplus_i (g_*\EE) \tensor \omega^{\tensor n_i},
 \end{align*}
The result follows. 
\end{proof}

\begin{remark}
 The explicit decomposition (i.e.\ the powers of $\omega$) can be worked out by computing graded global sections over $k$, as we will do later in the case $k=\C$ in Section \ref{sec:dec}. 
\end{remark}

\subsection{Working integrally and uncompactified}
In the following, we assume again that $\Xb$ is as in Convention \ref{convention} with $R = \Z_{(l)}$. Set $\XX = \Xb \times_{\MMb_{ell}} \MM_{ell}$ again and let $\MM'$ be as in the last subsection. Furthermore, let $\EE$ be a vector bundle on $\XX$. We aim to prove the following theorem:

\begin{thm}\label{thm:decuncompact}
 Under these assumptions, $g_*\EE$ splits into a sum of vector bundles of the form $f_*\OO_{\MM'} \tensor \omega^{\tensor i}$ on the $l$-completion $\widehat{\MM}_{ell,l}$. 
\end{thm}

We refer to \cite{Con} for a treatment of completions in the generality of Artin stacks. We will deduce the theorem from the corresponding theorem over $\F_l$ by the following result from obstruction theory, which is analogous to \cite[Cor 8.5.5]{FAG}. 

\begin{prop}
 Let $A$ be a noetherian ring with a maximal ideal $I$. Let $\XX$ be a locally noetherian Artin stack over $A$. We denote by $\hat{\XX}$ the completion with respect to $I\OO_\XX$ and by $\XX_0$ the closed substack corresponding to $I\OO_{\XX}$. Let $\FF,\GG$ be vector bundles on $\XX$ such that $H^1(\XX_0; \mathcal{H}om(\FF|_{\XX_0}, \GG|_{\XX_0})) = 0$. Then the completions $\widehat{\FF}$ and $\widehat{\GG}$ are isomorphic on $\widehat{\XX}$ if and only if $\FF|_{\XX_0}\cong \GG|_{\XX_0}$.
\end{prop}
\begin{proof}The only-if direction is clear. We will prove the if direction.

 Set $\XX_n = \XX\times_{\Spec A} \Spec A/I^{n+1}$ and $\FF_n =\FF|_{\XX_n}$ and $\GG_n = \GG|_{\XX_n}$. Choose an isomorphism $f_0\colon \FF_0 \to \GG_0$. We want to show inductively that it lifts to a morphism $f_n\colon \FF_n \to \GG_n$ (whose pullback to $\XX_0$ agrees with $f_0$). Assume that a lift $f_{n-1}$ of $f_0$ to $\XX_{n-1}$ is already constructed. The ideal defining $\XX_{n-1}\subset \XX_n$ is $$I^n\OO_\XX/I^{n+1}\OO_\XX \cong \OO_\XX \tensor_A I^n/I^{n+1}.$$
 By \cite[Thm 8.5.3]{FAG}, The obstruction for lifting $f_{n-1}$ to $f_n\colon \FF_n \to \GG_n$ lies in 
 \begin{align*}
  H = H^1(\XX_{n-1}; I^n\OO_\XX/I^{n+1}\OO_\XX\tensor_{\OO_{\XX_{n-1}}} \mathcal{H}om_{\XX_{n-1}}(\FF_{n-1},\GG_{n-1})). 
 \end{align*}
Because the coefficient sheaf is killed by $I$, it is the pushforward of its pullback along the closed immersion $\XX_0\to \XX_n$. Thus, 
\begin{align*}
H &\cong
 H^1(\XX_0; I^n\OO_\XX/I^{n+1}\OO_\XX\tensor_{\OO_{\XX_0}} \mathcal{H}om_{\XX_0}(\FF_0,\GG_0)) \\
 &\cong H^1(\XX_0; I^n/I^{n+1}\tensor_{A/I} \mathcal{H}om_{\XX_0}(\FF_0,\GG_0)) \\
 &\cong I^n/I^{n+1} \tensor_{A/I} H^1(\XX_0; \mathcal{H}om_{\XX_0}(\FF_0,\GG_0))  \\
 &\cong 0.
\end{align*}
Thus, we get a compatible system of morphisms $(f_n)_{n\in\N_0}$. All of these are isomorphisms by the Nakayama lemma. Thus, $(f_n)_{n\in\N_0}$ defines an isomorphism of adic systems and thus one of completions by \cite{Con}.
\end{proof}

\begin{proof}[Proof of theorem:]
We first want to argue that $\EE$ is the restriction of a vector bundle on $\Xb$. Indeed, by Lemma 3.2 from \cite{Mei13}, there is a reflexive sheaf $\EE'$ on $\Xb$ restricting to $\EE$. Choose a surjective \'etale map $U \to \Xb$ from a scheme $U$. By Proposition \ref{prop:basicprops}, the scheme $U$ is smooth of relative dimension $1$ over $\Z_{(l)}$ and thus a regular scheme of dimension $2$. By \cite[Cor 1.4]{Har80}, the pullback of $\EE'$ to $U$ is a vector bundle and thus $\EE'$ itself as well. 

 By Proposition \ref{prop:split}, it follows that $g_*\EE$ and a vector bundle $\FF$ of the form 
 $$\bigoplus_{i\in I} f_*\OO_{\MM'} \tensor \omega^{\otimes n_i}$$ 
 are isomorphic after base change to $\F_l$. Set $\EE_0 = \EE|_{\XX_{\F_l}}$ and $\FF_0 = \FF|_{\MM_{ell, \F_l}}$. By Lemma \ref{lem:pushpull} and the last proposition, it remains to show that 
 $$H^1(\MM_{ell, \F_l}; \mathcal{H}om(\FF_0, g_*\EE_0)) = 0.$$
 We have
 \begin{align*}
  H^1(\MM_{ell, \F_l}; \mathcal{H}om(\FF_0, g_*\EE_0)) &\cong H^1(\MM_{ell,\F_l}; g_* \mathcal{H}om(g^*\FF_0, \EE_0)) \\
  &\cong H^1(\XX_{\F_l}; \mathcal{H}om(g^*\FF_0, \EE_0))= 0
 \end{align*}
 as the cohomology of every quasi-coherent sheaf on $\XX_{\F_l}$ vanishes by Proposition \ref{prop:basicprops}. 
\end{proof}

We can be more explicit in many cases, using the following two lemmas.
\begin{lemma}\label{lem:trivial}
 The line bundle $\omega$ is trivial on $\MM_1(n)$ for $n\geq 4$ and non-trivial for $n=2,3$.
\end{lemma}
\begin{proof}
 The non-triviality is known by Example \ref{exa:moduli}. Now assume that $n\geq 4$. Then $\MM_1(n)$ is an affine scheme by Proposition \ref{prop:basicprops} and integral by Proposition \ref{prop:irreducible}. Indeed, $\MMb_1(n)$ is smooth over $\Spec \Z[\frac1n]$ and connected (as it is connected after base change to $\C$). 
 
 Zariski locally, we can write the universal elliptic curve $E$ over $\MM_1(n)$ first in Weierstra{\ss} form and then transform it to Kubert--Tate normal form 
 $$y^2+(1-c)xy -by = x^3-bx^2$$
 with $b\neq 0$ everywhere and $(0,0)$ as the chosen $n$-torsion point (see \cite[Section 4.4]{HusElliptic} or \cite[Section 1]{B-O16}). We claim that there can be no non-trivial coordinate change fixing $(0,0)$ and the form of the equation. Indeed, the only coordinate change fixing $(0,0)$ is $x = u^2x'$ and $y = u^3y'+u^2sx'$ (see \cite[III.1]{Sil09} for the coordinate change formulas). We have 
 $sb = u^4a_4' = 0$. Thus, $s=0$. Furthermore, we have 
 $$u^2a_2' = -b =u^3a_3'$$ 
 and $a_2' = a_3'$. Thus, $u=1$. This implies that we can find a global Weierstra{\ss} form (of the form above) for $E$. Its invariant differential provides a trivialization of $\omega$ on $\MM_1(n)$. 
\end{proof}
If we had a Krull--Schmidt theorem for $\MM_{ell}$ or $\widehat{\MM}_{ell,l}$, this would directly imply a periodicity in the decomposition of $g_*\OO_{\MM_1(n)}$. Due to Example \ref{exa:KS}, we have to be more careful though and formulate a weaker substitute. To do this, note first that for $l=2$ or $3$ we have
$$f_*\OO_{\MM'} \cong f_*\OO_{\MM'} \otimes \omega^{\tensor (l+1)}.$$
Indeed, recall that $$\MM_1(3) \simeq \Spec \Z[\tfrac13][a_1,a_3,\Delta^{-1}]/\G_m$$
with $\Delta = a_3^3(a_1^3-27a_3)$ (\cite[Prop 3.2]{M-R09}); in this case $a_3$ defines the required trivialization of $\omega^{\tensor 3}$ on $\MM_1(3)$. Furthermore recall that 
$$\MM_1(2) \simeq \Spec \Z[\tfrac12][b_2,b_4,\Delta^{-1}]/\G_m$$
with $\Delta = 16b_4^2(b_2^2-4b_4)$ (\cite[Sec 1.3]{Beh06}); in this case $b_4$ defines the required trivialization of $\omega^{\tensor 4}$ on $\MM_1(2)$.

\begin{lemma}
 Let $l=2$ or $3$. Then 
 $$\bigoplus_{[k] \in \Z/(l+1)}\bigoplus_{m_k} f_*\OO_{\MM'_S}\tensor \omega^{\tensor k} \cong \bigoplus_{[k] \in \Z/(l+1)}\bigoplus_{m_k'} f_*\OO_{\MM'_S}\tensor \omega^{\tensor k}$$
 on $\MM_{ell,S}$ for $S$ a (formal) scheme with an $\Fb_l$-point implies $m_k = m'_k$ for all $[k]\in\Z/(l+1)$. 
\end{lemma}
\begin{proof}
We can assume $S = \Spec \overline{\F}_l$ and we will leave it implicit in the notation. 

 Consider first $l=3$ with $\MM' = \MM_1(2)$. Let $E$ be given by the equation $y^2 = x^3-x$ over $\overline{\F}_3$ and let $G= \langle t\rangle$ be the group of order $4$ generated by  the automorphism
 $$t\colon (x,y) \mapsto (-x, iy)$$
 of $E$, where $i\in\overline{\F}_3$ is a chosen primitive fourth root of unity. Let 
 $$h\co \Spec \overline{\F}_3/G \to \MM_{ell}$$
 be the map classifying this elliptic curve. The $G$-representation $P$ over $\overline{\F}_3$ corresponding to $h^*f_*\OO_{\MM_1(2)}$ is the $3$-dimensional permutation representation associated to the action of $G$ on the points of exact order $2$. Our next goal is to compute $h^*(f_*\OO_{\MM_1(2)}\tensor \omega^{\tensor k})$ for $0\leq k \leq 3$.  
 
 The curve $E$ has an invariant differential $\frac{dx}{2y}$ and $t$ acts on it as multiplication by $i$. Thus, $h^*\omega^{\tensor k}$ corresponds to the $1$-dimensional representation $L_k$ where $t$ acts as $i^k$. The points of exact order $2$ of $E$ are $(0,0), (1,0)$ and $(-1,0)$. Thus, the $G$-representation $P$ splits as $L_0 \oplus L_0 \oplus L_2$, where the summands are spanned by $(0,0), (1,0)+(-1,0)$ and $(1,0)-(-1,0)$. The vector bundle $h^*(f_*\OO_{\MM_1(2)}\tensor \omega^{\tensor k})$ corresponds to the representation $L_k \oplus L_k \oplus L_{k+2}$. 
 
 Let now 
 $$\EE \cong \bigoplus_{[k] \in \Z/4} \bigoplus_{m_k} f_*\OO_{\MM_1(2)}\tensor \omega^{\tensor k}.$$
  Then $h^*\EE$ decomposes as $\bigoplus_{[k]\in\Z/4} L_k^{\oplus n_k}$ with $n_k = 2m_k+ m_{k+2}$. By the Krull--Schmidt theorem for $G$-representations over $\overline{\F}_3$, the sequence of $n_k$ is uniquely determined by $h^*\EE$. We can recover the $m_k$ by the formula $m_k = \frac23 n_k -\frac13 n_{k+2}$.
 
 Consider now $l=2$ with $\MM' =\MM_1(3)$. Let $E$ be given by the equation $y^2 +y = x^3$ over $\overline{\F}_2$ (as in Lemma \ref{lem:indec}) and let $G = \langle s\rangle$ be the group of order $3$ generated by the automorphism  
 $$s\colon (x,y) \mapsto (\zeta x, y)$$
 for a third root of unity $\zeta \in \overline{\F}_2$.  
 Let 
 $$h\co \Spec \overline{\F}_2/G \to \MM_{ell}$$
 be the map classifying this elliptic curve. 
 The $G$-representation $P$ over $\overline{\F}_l$ corresponding to $h^*f_*\OO_{\MM_1(3)}$ is the $8$-dimensional permutation representation associated to the action of $G$ on the points of exact order $3$. 
 
 The curve $E$ has invariant differential $dx$ and $s$ acts on it as multiplication by $\zeta$. Thus, $h^*\omega^{\tensor k}$ corresponds to the $1$-dimensional representation $L_k$ where $t$ acts as $\zeta^k$. The points of exact order $3$ of $E$ are 
$$(0,0), (0,1), (1,\zeta), (1,\zeta^2), (\zeta,\zeta), (\zeta, \zeta^2), (\zeta^2, \zeta), (\zeta^2,\zeta^2).$$
Thus, the $G$-representation $P$ splits as $L_0^{\oplus 4} \oplus L_1^{\oplus 2}\oplus L_2^{\oplus 2}$, where the summands are spanned by 
$$(0,0), (0,1), (1,\zeta) + (\zeta,\zeta)+(\zeta^2,\zeta), (1,\zeta^2) + (\zeta,\zeta^2)+(\zeta^2,\zeta^2), (1,\zeta) + \zeta^2(\zeta,\zeta)+\zeta(\zeta^2,\zeta),\dots$$
 The vector bundle $h^*(f_*\OO_{\MM_1(3)}\tensor \omega^{\tensor k})$ corresponds to the representation $L_k^{\oplus 4} \oplus L_{k+1}^{\oplus 2} \oplus L_{k+2}^{\oplus 2}$. 
 
  Let now 
 $$\EE \cong \bigoplus_{[k] \in \Z/3} \bigoplus_{m_k} f_*\OO_{\MM_1(3)}\tensor \omega^{\tensor k}.$$
  Then $h^*\EE$ decomposes as $\bigoplus_{[k]\in\Z/3} L_k^{\oplus n_k}$ with $n_k = 4m_k+ 2m_{k+1}+2m_{k+2}$. By the Krull--Schmidt theorem, the sequence of $n_k$ is uniquely determined by $h^*\EE$. We can recover the $m_k$ by the formula $m_k = \frac38 n_k -\frac18 n_{k+1}-\frac18n_{k+2}$.
\end{proof}

\begin{cor}\label{cor:equal}
 Let $\XX = \widehat{\MM_1(n)}_l$ or $\widehat{\MM(n)}_l$ for $n\geq 4$. Then $g_*\OO_{\widehat{\XX}_l}$ decomposes into copies of $f_*\OO_{\widehat{\MM'}_l} \oplus \cdots \oplus (f_*\OO_{\widehat{\MM'}_l}\otimes \omega^{\tensor l})$ for $l=2,3$. 
\end{cor}
\begin{proof}
We will implicitly complete everywhere at $l$. By Theorem \ref{thm:decuncompact}, we can write
 $$g_*\OO_\XX \cong \bigoplus_{[k] \in \Z/(l+1)} \bigoplus_{m_k} f_*\OO_{\MM'}\tensor \omega^{\tensor k}.$$
 We know that $g^*\omega \cong \OO_\XX$ by Lemma \ref{lem:trivial}. This implies that $g_*\OO_\XX \cong g_*\OO_\XX \otimes \omega$ and thus that 
 $$\bigoplus_{[k] \in \Z/(l+1)} \bigoplus_{m_k} f_*\OO_{\MM'}\tensor \omega^{\tensor k} \cong \bigoplus_{[k] \in \Z/(l+1)} \bigoplus_{m_{k-1}} f_*\OO_{\MM'}\tensor \omega^{\tensor k}.$$
 The previous lemma implies the result (taking $S = \Spf \Z_l$).
\end{proof}

\subsection{Working integrally and compactified}
Consider again the diagram
\[\xymatrix{\Yb = \XXb \times_{\MMb_{ell}}\MMb' \ar[d]^{f'}\ar[r]^-  {g'} & \MMb' \ar[d]^f \\
\XXb \ar[r]^{g} & \MMb_{ell}
 }
\]
where $\Xb$ is as in Convention \ref{convention} with $R = \Z_{(l)}$ and $\MMb'$ as in Section \ref{sec:workingfield}.

\begin{thm}\label{thm:deccompact}
 Assume that $H^1(\Xb; g^*\omega)$ has no $l$-torsion. Then $g_*\OO_{\XXb}$ decomposes into a sum of vector bundles of the form $f_*\OO_{\MMb'}\tensor \omega^{\tensor i}$.
\end{thm}
\begin{proof}
 By Proposition \ref{prop:split}, there is an isomorphism
 \[\overline{h}\colon g_*\OO_{\XXb_{\F_l}} \to f_*\FF_{\F_l}\]
 for some $\FF = \bigoplus_{i\in\Z}\bigoplus_{k_i} f^* \omega^{\tensor i}$. By adjunction, this corresponds to a morphism
 \[\overline{\varphi}\colon (g')_*\OO_{\Yb_{\F_l}} \cong f^*g_*\OO_{\XXb_{\F_l}} \to \FF_{\F_l},\]
 where we use Lemma \ref{lem:pushpull} again. By Proposition \ref{prop:basicprops}, $g$ and hence $g'$ is finite and flat and thus $(g')_*\OO_{\Yb}$ is a vector bundle. We want to show that it decomposes as a sum $\GG = \bigoplus_{i\in \Z}\bigoplus_{n_i}f^*\omega^{\tensor i}$ for some $n_i$. As $\MMb'$ is a weighted projective line by Examples \ref{exa:moduli}, Proposition \ref{thm:VectorBundle} and the flatness of $\MMb'$ over $\Z_{(l)}$ imply that it is enough to show that $H^1(\MMb';(g')_*\OO_{\Yb}\tensor f^*\omega^{\tensor j})$ is a free $\Z_{(l)}$-module. We have the following chain of isomorphisms:
 \begin{align*}
  H^1(\MMb';(g')_*\OO_{\Yb}\tensor f^*\omega^{\tensor j}) &\cong H^1(\MMb_{ell};f_*(g')_*(g')^*f^*\omega^{\tensor j}) \\
  &\cong H^1(\MMb_{ell};g_*(f')_*(g')^*f^*\omega^{\tensor j}) \\
  &\cong  H^1(\MMb_{ell};g_*g^*f_*f^*\omega^{\tensor j})\\
  &\cong H^1(\Xb;g^*f_*f^*\omega^{\tensor j})\\
  &\cong H^1(\Xb;(g^*f_*\OO_{\MMb'}) \tensor g^*\omega^{\tensor j})
 \end{align*}
By Lemma \ref{lem:splitting}, the last group is a direct summand of terms of the form $H^1(\Xb; g^*\omega^{\tensor n})$. This is $\Z$-torsionfree by assumption for $n=1$ and else as in Proposition \ref{prop:coh}. Thus, $(g')_*\OO_{\Yb}$ is of the form $\GG$.
 
 We claim that the morphism
 \[\Hom_{\OO_{\MMb'}}(\GG, \FF) \to \Hom_{\OO_{\MMb'_{\F_l}}}(\GG_{\F_l}, \FF_{\F_l})\]
 is surjective. This surjectivity follows from 
$$H^0(\MMb'; f^*\omega^{\tensor i}) \to H^0(\MMb'_{\F_l}; f^*\omega^{\tensor i}_{\F_l}) \cong H^0(\MMb'_{\F_l}; f^*\omega^{\tensor i}/l)$$ being surjective, which in turn is true by the long exact cohomology sequence associated with
$$0 \to f^*\omega^{\tensor i} \xrightarrow{l} f^*\omega^{\tensor i} \to f^*\omega^{\tensor i}/l\to 0$$
and the fact that $H^1(\MMb'; f^*\omega^{\tensor i})$ has no $l$-torsion. 
 
 Thus, we can choose a map $(g')_*\OO_{\Yb} \to \FF$, which reduces to $\overline{\varphi}$ mod $l$. By tracing through the adjunctions, this corresponds to a map $h\colon g_*\OO_{\XXb} \to f_*\FF$ whose restriction to $\MMb_{ell,\F_l}$ agrees with $\overline{h}$.  To show that $h$ is an isomorphism it is enough to check this after pullback to an fpqc cover, e.g.\ $\MMb_1(5)$ or $\MMb_1(6)$, which are isomorphic to $\mathbb{P}^1_{\Z_{(l)}}$ (see Example \ref{exa:5-12}). Thus, $h$ is a morphism between vector bundles on $\mathbb{P}^1_{\Z_{(l)}}$ whose restriction to $\mathbb{P}^1_{\F_l}$ is an isomorphism. We know that $h$ is an isomorphism on an open subset of $\mathbb{P}^1_{\Z_{(l)}}$ that contains the special fiber; its complement is a closed subset $A$. The image of $A$ under $\mathbb{P}^1_{\Z_{(l)}} \to \Spec \Z_{(l)}$ is closed and thus empty as it cannot contain the closed point. Thus, $A$ is empty as well and $h$ is an isomorphism. 
\end{proof}

\begin{remark}\label{rem:cuspy}
Again by the long exact cohomology sequence associated with
$$0 \to \omega \xrightarrow{l} \omega \to \omega/l\to 0$$
on $\Xb$ we see that $H^1(\Xb; g^*\omega)$ is $l$-torsionfree if and only if the map 
$$H^0(\Xb;g^*\omega) \to H^0(\Xb_{\F_l};g^*\omega)$$
is surjective, i.e.\ if every mod-$l$ modular form of weight $1$ can be lifted to a characteristic zero form of the same kind. If there are such non-liftable modular forms, there are indeed also mod-$l$ cusp forms of weight $1$ non-liftable to characteristic zero cusp forms (of the same level). Indeed by the Semicontinuity Theorem \cite[Thm 12.8, 12.9]{Har77}, $H^1(\Xb; g^*\omega)$ having $l$-torsion implies that the rank of $H^1(\Xb_{\F_l}; g^*\omega)$ is bigger than that of $H^1(\Xb_{\C};g^*\omega)$ and these ranks agree with those of $\F_l$-valued and $\C$-valued cusp forms of weight $1$ by Corollary \ref{cor:cusps}.

As in \cite[Lemma 2]{Buz}, it is easy to show by hand that non-liftable mod-$l$ weight $1$ cusp forms do not exist for $\Xb  = \MMb_1(n)$ and $n\leq 28$. Further explorations benefit from computer help and were done in \cite{Edi06}, \cite{Buz}, \cite{SchThesis} and \cite{Sch14}. Some small examples from these sources where $l$-torsion occurs in $H^1(\MMb_1(n); g^*\omega)$ are 
 $$ (l,n) = (2,1429),\quad (3,74),\quad (3,133),\quad (5, 141) \;\text{ and }\; (199,82).$$
 There is evidence \cite{Sch14} that the torsion in $H^1(\MMb_1(n);g^*\omega)$ grows at least exponentially in $n$. 
\end{remark}

\section{Computing the decompositions and duality}
In this section, we will be more concrete and actually give formulas how to decompose vector bundles and also determine when this decomposition is ``symmetric''. As we will mostly work over a field of characteristic zero, the hard work of the last section is almost entirely unnecessary for this section though the results of this section have strong implications for the integral decompositions from the last section. The work in this section is partially joint with Viktoriya Ozornova. 
\subsection{Decompositions}\label{sec:dec}
 We work (implicitly) over a field $K$ of characteristic $0$. Let $\Gamma$ be one of the congruence subgroups $\Gamma_0(n)$, $\Gamma_1(n)$ or $\Gamma(n)$. Let $f_n\colon \MMb(\Gamma) \to \MMb_{ell}$ be the projection.  Recall that $\MMb_{ell} \simeq \PP(4,6)$ with $\OO(1) \cong \omega$. By \cite[Theorem 4.1.1]{Con07} or \cite[Theorem 5.5.1]{K-M85}, the $\OO_{\MMb_{ell}}$-module $(f_n)_*\OO_{\MMb(\Gamma)}$ is locally free of finite rank. Thus it decomposes by Proposition \ref{prop:ClassField} as 
 \begin{align}\label{eq:dec} (f_n)_*\OO_{\MMb(\Gamma)} \cong \bigoplus_{i\in\Z} \bigoplus_{l_i} \omega^{\tensor (-i)}.\end{align}
 
 Our aim is to determine the sequence of $l_i$ (which is well-defined by the Krull--Schmidt Theorem \ref{prop:Krull-Schmidt}). We will sometimes call it the \emph{decomposition sequence} of $(f_n)_*\OO_{\MMb(\Gamma)}$. 
 
 Denote by $m_i$ the dimension of the space of weight $i$ modular forms for $\Gamma$ and by $s_i$ the dimension of the space of cusp forms for the same weight and group. 
 
 \begin{prop}\label{prop:dec5}We have 
  \begin{enumerate}
   \item $l_i = 0$ for $i<0$ and $i> 11$,
   \item $l_i = m_i-m_{i-4}-m_{i-6}+m_{i-10}$ for $i\leq 11$; in particular, $l_i = m_i$ for $i\leq 3$,
   \item $l_{12-i} = s_i$ for $i\leq 4$,
   \item $l_{10}$ is the genus of $\MMb(\Gamma)$, i.e.\ $\dim_K H^0(\MMb(\Gamma); \Omega^1_{\MMb(\Gamma)/K})$. 
  \end{enumerate}
 \end{prop}
 \begin{proof}
  The number $m_k$ is by definition the dimension of 
  \[H^0(\MMb(\Gamma); (f_n)^*\omega^{\tensor k}) \cong H^0(\MMb_{ell}; (f_n)_*\OO_{\MMb(\Gamma)}\tensor \omega^{\tensor k}).\]
  Denote by $d_i = \dim_K H^0(\MMb_{ell};\omega^{\tensor i}$ the dimension of the space of holomorphic modular forms of weight $i$ for $SL_2\Z$. Then \eqref{eq:dec} implies that
  \begin{align}\label{sum}m_k = \sum_{i\in \Z} l_id_{k-i}.\end{align}
  In particular, $l_i = 0$ for $i<0$ because $m_i \geq l_i = l_id_{i-i}$ and there are no modular forms of negative weight (Proposition \ref{prop:coh}). 
  
  To get more precise results, we want to use Serre duality on $\MMb_{ell}$ and $\MMb(\Gamma)$. By Theorem \ref{thm:fundamentalweighted}, the stack $\MMb_{ell} \simeq \PP_K(4,6)$ has dualizing sheaf $\omega^{\tensor (-10)}$. Using this and that $f_n$ is affine and thus $(f_n)_*$ is exact, we get the following chain of isomorphisms: 
  
  \begin{align*}
   H^0(\MMb_{ell}; ((f_n)_*\OO_{\MMb(\Gamma)})^\vee \tensor \omega^{\tensor -10-k}) &\cong H^0(\MMb_{ell}; ((f_n)_*\OO_{\MMb(\Gamma)}\tensor \omega^{\tensor k})^\vee \tensor \omega^{\tensor -10}) \\
   &\cong H^1(\MMb_{ell}; (f_n)_*\OO_{\MMb(\Gamma)}\tensor \omega^{\tensor k})^{\vee} \\
   &\cong H^1(\MMb_{ell}; (f_n)_*(\OO_{\MMb(\Gamma)}\tensor (f_n)^*\omega^{\tensor k}))^{\vee}\\
   &\cong H^1(\MMb(\Gamma); (f_n)^*\omega^{\tensor k})^{\vee} \\
   &\cong H^0(\MMb(\Gamma); \Omega^1_{\MMb(\Gamma)/K} \tensor (f_n)^*\omega^{\tensor -k})
  \end{align*}
 By Proposition \ref{prop:coh}, this vanishes for $k\geq 2$.
  
  Since the rank of $(f_n)_*\OO_{\MMb(\Gamma)}$ is finite, only finitely many $l_i$ can be nonzero. Let $j$ be the largest index such that $l_j \neq 0$. Then $H^0(\MMb_{ell};((f_n)_*\OO_{\MMb(\Gamma)})^\vee \tensor \omega^{\tensor -j})$ has dimension $l_j$. In particular, we see by the computation above that $j\leq 11$, proving part (1) of our proposition. Using that the ring of holomorphic modular forms for $SL_2\Z$ is isomorphic to $K[c_4,c_6]$ and thus $d_0 = d_4 = d_6 = d_8 = d_{10} = 1$ and $d_i = 0$ for all other $i\leq 11$, we obtain the recursive equation
  $$l_i = m_i-l_{i-4}-l_{i-6}- l_{i-8} - l_{i-10}$$
  from Equation (\ref{sum}). Part (2) of our proposition follows by a straightforward computation.
  
  The equality $\dim_K H^0(\MMb_{ell};((f_n)_*\OO_{\MMb(\Gamma)})^\vee \tensor \omega^{\tensor -i}) = l_i$ holds even for all $j-3\leq i$ (and in particular for $i\geq 8$) as $d_k = 0$ for $0<k<4$. Part (3) follows then from Corollary \ref{cor:cusps}. Part (4) follows from the previous computations and the definition of the genus. 
 \end{proof}
 
 \begin{example}\label{exa:m12}
  Let us consider the case $n=2$. By Example \ref{exa:moduli}, the ring of modular forms for $\Gamma_1(2)$ is isomorphic to $K[b_2, b_4]$. Furthermore, we know that the rank of $(f_2)_*\OO_{\MMb_1(2)}$ is $3$. Thus, there can be at most three nonzero $l_i$ and these are $l_0 =l_2 = l_4 = 1$. In other words: $(f_2)_*\OO_{\MMb_1(2)} \cong \OO_{\MMb_{ell}} \oplus \omega^{\tensor -2} \oplus \omega^{\tensor -4}$.
 \end{example}
 
 \begin{example}\label{exa:m13}
  Now consider the case $n=3$. By Example \ref{exa:moduli}, the ring of modular forms for $\Gamma_1(3)$ is isomorphic to $K[a_1,a_3]$. By the last proposition, it follows easily that in this case
  $$l_0 =1,\; l_1=1,\; l_2=1,\; l_3=2,\; l_4=1,\; l_5 =1,\; l_6 = 1$$
  and all the other $l_i$ are zero. 
 \end{example}
 
We will need the following standard fact, which follows directly from Riemann--Roch, Proposition \ref{prop:coh} and Corollary \ref{cor:cusps}. 

\begin{prop}\label{prop:riemann}
Assume that $\MMb(\Gamma)$ is representable. Let $g = s_2$ be the genus of $\MMb(\Gamma)$. Then 
$$\deg(f_n^*\omega)i + 1-g = \begin{cases}m_1-s_1 & \text{ if }i =1\\
m_i & \text{ if } i\geq 2.
\end{cases}.$$
\end{prop}

Let us be more explicit about $\deg(f_n^*\omega)$ for $\Gamma = \Gamma_1(n)$. 

\begin{lemma}\label{lem:degomega}
The degree $\deg (f_n)^*\omega$ equals $\frac1{24} d_n$ for $d_n$ the degree of the map $\MM_1(n) \to \MM_{ell}$. We have 
\begin{align*}
 d_n &= \sum_{d|n}d\phi(d)\phi(n/d) \\
 &= n^2\prod_{p|n}(1-\frac1{p^2})
\end{align*}
\end{lemma}
\begin{proof}
By \cite[VI.4.4]{D-R73}, the line bundle $\omega$ has degree $\frac1{24}$. As  $\deg (f_n)^*\omega = \deg(f_n) \deg \omega$, the first result follows. For the formulas for $d_n$ see \cite[Sec 3.8+3.9]{D-S05}; note that the map of stacks has twice the degree of the map of coarse moduli spaces as the generic point of $\MM_{ell}$ has automorphism group of order $2$. 
\end{proof}

 \begin{cor}\label{cor:dec3}
 Let $\Gamma$ be $\Gamma_1(n)$ for $n\geq 4$ or $\Gamma(n)$ for $n\geq 3$. Then we have a decomposition
 \[(f_n)_*\OO_{\MMb(\Gamma)} \cong \bigoplus_{i\in\Z} \bigoplus_{k_i} (f_2)_*\OO_{\MMb_1(2)} \tensor \omega^{\tensor -i}.\]
 The $k_i$ are uniquely determined and satisfy
 \begin{enumerate}
  \item $k_i = 0$ for $i<0$ and $i>7$,
  \item $k_i = m_i-m_{i-2}-m_{i-4}+m_{i-6}$; in particular, $k_i = m_i$ for $i\leq 1$,
  \item $k_7 = s_1$ and $k_6=s_2$ is the genus of $\MMb(\Gamma)$,
  \item $k_0 + k_4 = k_1 +k_5 = k_2+k_6 = k_3+k_7$.
 \end{enumerate}
\end{cor}
\begin{proof}
First we want to show the existence of a decomposition of the form 
\begin{align}\label{eq:dec2} (f_n)_*\OO_{\MMb(\Gamma)} \cong \bigoplus_{i\in\Z} \bigoplus_{k_i} (f_2)_*\OO_{\MMb_1(2)} \tensor \omega^{\tensor -i}.\end{align}
This follows if $n$ is not divisible by $3$ and $H^1(\MMb(\Gamma);\omega)$ (with $\MMb(\Gamma)$ viewed over $\Z_{(3)}$) contains no $3$-torsion by Theorem \ref{thm:deccompact}, but we will prove it more generally in this situation, where we are over a field of characteristic $0$. The idea is to show that one can group summands of the decomposition in \eqref{eq:dec} together and use Example \ref{exa:m12} to get copies of $(f_2)_*\OO_{\MMb_1(2)}$.

To that purpose set $k_i = m_i-m_{i-2}-m_{i-4}+m_{i-6}$. First we have to show that $k_i\geq 0$. We will use that $m_i = 0$ for $i<0$ and $m_0 = 1$. The two further properties of the sequence $m_i$ we need are that there are constants $a>0$ and $b$ such that $m_i = ai +b$ for all $i\geq 2$ and that $m_2 \neq 0$ by Proposition \ref{prop:riemann}.\footnote{For $\Gamma = \Gamma_1(4)$ the same holds, but with different constants $a,b$ for $i$ even or odd (see \cite[Fig 3.4]{D-S05}). This does not affect our proof.}  We see directly that $k_i = 0$ for $i>7$ or for $i<0$. As the ring of modular forms has no zero divisors, we also see that $m_i \geq m_{i-2}$ and likewise $s_i\geq s_{i-2}$ by multiplying with a nonzero modular form of weight $2$. This implies $k_i \geq 0$ for $i\leq 3$. We have 
$$k_4 = m_4-m_2-m_0 = 2a -1 \geq 0$$
as $2a = m_4-m_2$ for $i\geq 2$ implies that $a \in \tfrac12 \Z$.
From Proposition \ref{prop:dec5} and the formula $m_i = ai+b$, we can compute 
\begin{align*}
k_7 &= m_7-m_3-m_5+m_1\\
&= m_{11}-m_7-m_5+m_1\\
&= s_1
\end{align*}
and likewise $k_6 = s_2$ and $k_5 = s_3-s_1$, which are also clearly at least $0$. By Example \ref{exa:m12}, we just have to check now that $l_i = k_i+k_{i-2}+k_{i-4}$, which is a straightforward computation. Thus, we obtain the existence of a decomposition into $(f_2)_*\OO_{\MMb_1(2)} \tensor \omega^{\tensor -i}$.  

Next we show that the formula from item 2 has to hold for every decomposition of the form \ref{eq:dec2}. This follows by a straightforward computation from the fact that $k_i = 0$ for $i<0$ and $i>7$ (as follows from Proposition \ref{prop:dec5} and Example \ref{exa:m12}) and from the equation
\[k_i = m_i-k_{i-2}-2k_{i-4} -2k_{i-6},\]
which we obtain from the dimension of the space of modular forms for $\Gamma_1(2)$ being $1,1,2$ and $2$ in weights $0$, $2$, $4$ and $6$ respectively and zero else in this range. 
 Item 2 implies the following equations:
 \begin{align*}
  k_0+k_4 &= m_4-m_2 \\
  k_1 +k_5 &= m_5-m_3 \\
  k_2+k_6 &= m_6-m_4\\
  k_3+k_7 &= m_7-m_5
 \end{align*}
As $i\mapsto m_i$ is affine linear in $i$ for $i\geq 2$, this implies item 4. 
\end{proof}

 \begin{cor}\label{cor:dec2}
 Let $\Gamma$ be $\Gamma_1(n)$ for $n\geq 5$ or $\Gamma(n)$ for $n\geq 3$. We have a decomposition
 \[(f_n)_*\OO_{\MMb(\Gamma)} \cong \bigoplus_{i\in\Z} \bigoplus_{k_i} (f_3)_*\OO_{\MMb_1(3)} \tensor \omega^{\tensor -i}.\]
 The $k_i$ are uniquely determined and satisfy
 \begin{enumerate}
  \item $k_i = 0$ for $i<0$ and $i>5$,
  \item $k_i = m_i-m_{i-1}-m_{i-3}+m_{i-4}$,
  \item $k_5 = s_1$ and $k_4 =s_2-s_1$,
  \item $k_0 + k_3= k_1 +k_4 = k_2+k_5$.
 \end{enumerate}
\end{cor}
\begin{proof}
The proof is quite analogous to that of the last corollary. 
First we want to show the existence of a decomposition of the form 
\[(f_n)_*\OO_{\MMb(\Gamma)} \cong \bigoplus_{i\in\Z} \bigoplus_{k_i} (f_3)_*\OO_{\MMb_1(3)} \tensor \omega^{\tensor -i}.\]
To that purpose set $k_i = m_i-m_{i-1}-m_{i-3}+m_{i-4}$. First we have to show that $k_i\geq 0$. The two essential properties of the sequence $m_i$ we need is that there are constants $a>0$ and $b$ such that $m_i = ai +b$ for all $i\geq 2$ by Proposition \ref{prop:riemann} and that $m_1 \neq 0$. For the latter, we use Theorem 3.6.1 from \cite{D-S05}, which shows that $2m_1$ is at least the number of regular cusps for $\Gamma$ and p.\ 103 of loc.\ cit.\ where it is shown that in our cases all cusps are regular. 

We see directly that $k_i = 0$ for $i>5$ or for $i<0$. As the ring of modular forms has no zero divisors, we also see that $m_i \geq m_{i-1}$ and likewise $s_i\geq s_{i-1}$ by multiplying with a nonzero modular form of weight $1$. This implies $k_i \geq 0$ for $i\leq 2$. We have 
$$k_3 = m_3-m_2-m_0 = a -1 \geq 0$$
using that $a \in\Z$. 
From Proposition \ref{prop:dec5} and the formula $m_i = ai+b$, we can compute $k_5 = s_1$ and $k_4 = s_2-s_1$, which are also clearly at least $0$. By Example \ref{exa:m13}, we just have to check now that 
$$l_i = k_i+k_{i-1}+k_{i-2}+2k_{i-3}+k_{i-4}+k_{i-5}+k_{i-6},$$
which is a straightforward computation. Thus, we obtain the existence of a decomposition into $(f_3)_*\OO_{\MMb_1(3)} \tensor \omega^{\tensor -i}$.  

Next we show that the formula from item 2 has to hold for every such decomposition. This follows by a straightforward computation from $k_i = 0$ for $i>5$ or for $i<0$ (as follows from Proposition \ref{prop:dec5} and Example \ref{exa:m12}) and from the equation
\[k_i = m_i-k_{i-1}-k_{i-2} -2k_{i-3}-2k_{i-4}-2k_{i-5},\]
which we obtain from the dimension of the space of modular forms for $\Gamma_1(3)$ being $1,1,1,2,2$ and $2$ in weights $0$ to $5$, respectively. 
 Item 2 implies the following equations:
 \begin{align*}
  k_0+k_3&= m_3-m_2 \\
  k_1 +k_4 &= m_4-m_3 \\
  k_2+k_5 &= m_5-m_4
 \end{align*}
As the function $i\mapsto m_i$ is linear in $i$ for $i\geq 2$, this implies item 4. 
\end{proof}

The reader might wonder about the significance of the items (4) of the last two corollaries. It is twofold: First, it implies that over the uncompactified moduli stack all powers of $\omega$ appear equally often (which was already proved in an $l$-complete setting in Corollary \ref{cor:equal}). Here, we use that $\omega^{\tensor 3}$ is trivial on $\MM_1(3)$ and that $\omega^{\tensor 4}$ is trivial on $\MM_1(2)$. For the second consequence we need some preparation and in particular the following lemma, which was proven jointly with Viktoriya Ozornova. 

\begin{lemma}
 Let $K$ be an algebraically closed field and $A$ a $\Z_{\geq 0}$-graded integral domain over $K$. Set $m_i = \dim_K A_i$ and assume that $m_0 = 1$. Then $m_2 \geq 2m_1-1$.
\end{lemma}
\begin{proof}
 We will work throughout this proof over the field $K$ and set $n = m_1$. Let $P =K[x_1,\dots, x_n]$ with the $x_i$ of degree $1$. Without loss of generality we can assume that $A$ is generated in degree $1$ and can thus be written as $P/I$ with $I$ a homogeneous ideal generated in degrees $\geq 2$. An element in $P_2$ can be written as $\sum_{i\leq j} a_{ij}x_ix_j$ and thus we can view it as an upper triangular matrix $(a_{ij})$. Let $R$ be the space of upper triangular matrices associated with the elements of $P_2 \cap I$. For an arbitrary matrix $(a_{ij})$ we set 
 $U((a_{ij})) = (m_{ij})$ with 
 $$m_{ij} = \begin{cases}
  a_{ii} &\text{ if }i=j \\
  a_{ij}+a_{ji} &\text { if }i<j\\
  0 & \text{ else.}
 \end{cases}$$
 Thus, $U$ defines a linear map $\Mat_{n\times n} \to \UpT$ from all $n\times n$-matrices to the upper triangular ones. If we view linear homogeneous polynomials in the $x_i$ as column vectors $\mathbf{v},\mathbf{w}$ in $P_1 = A_1$, then their product corresponds to the upper triangular matrix $U(\mathbf{v}\mathbf{w}^T)$. As $A$ is an integral domain,
 no nonzero element in $R$ is of this form. 
 
  We can reformulate this in terms of the Segre embedding. Recall that this is the map $\iota\colon \mathbb{P}^{n-1}\times \mathbb{P}^{n-1} \to \mathbb{P}(\Mat_{n\times n})$ sending $([\mathbf{v}],[\mathbf{w}])$ to $[\mathbf{v}\mathbf{w}^T]$. Let $V = \mathbb{P}(\Mat_{n\times n}) \setminus \mathbb{P}(\ker(U))$. As $P$ is an integral domain, $\iota$ factors over $V$. Furthermore, $U$ defines an algebraic map $V \to \mathbb{P}(\UpT)$. The composite map $\kappa\colon \mathbb{P}^{n-1}\times \mathbb{P}^{n-1} \to \mathbb{P}(\UpT)$ is proper as the source is proper over $K$. Furthermore, $\kappa$ is quasifinite because $P$ is a unique factorization domain and thus every point in $\mathbb{P}(\UpT)$ has at most two preimages in $\mathbb{P}^{n-1}\times \mathbb{P}^{n-1}$. Thus, $\kappa$ defines a finite map $ \mathbb{P}^{n-1}\times \mathbb{P}^{n-1} \to \im(\kappa)$ and thus $\im(\kappa)$ is a projective variety of dimension $2n-2$. As $\im(\kappa)\cap \mathbb{P}(R) = \varnothing$, it follows by \cite[Thm 7.2]{Har77} 
that $R$ has dimension at most $\dim_K \UpT - (2n-1)$ and thus that 
  \[m_2 = \dim_K \UpT - \dim_K R \geq 2n-1. \qedhere\] 
\end{proof}

\begin{prop}
Let $\Gamma = \Gamma_0(n), \Gamma_1(n)$ or $\Gamma(n)$. 
\begin{enumerate}
\item $m_2\geq 2m_1-1$.
\item $s_2 \geq 2s_1$ if $\MMb(\Gamma)$ is representable.
\item Let $k_i^{(2)}$ be the $k_i$ from Corollary \ref{cor:dec3}. Then $k_3^{(2)}\geq k_2^{(2)}\geq k_1^{(2)} \geq 1$.
\item  Let $k_i^{(3)}$ be the $k_i$ from Corollary \ref{cor:dec2}. Then $k_2^{(3)}\geq k_1^{(3)} \geq 1$.
\end{enumerate}
\end{prop}
\begin{proof}
The inequality $m_2\geq 2m_1-1$ follows directly from the last lemma. 

Now assume $\MMb(\Gamma)$ to be representable. Recall that the genus $g =\dim_{K}H^0(\MMb(\Gamma);\Omega^1_{\MMb(\Gamma)})$ of $\MMb(\Gamma)$ equals $s_2$. By Proposition \ref{prop:riemann}, $m_1-s_1 = \deg(f_n^*\omega) + 1-g$ and thus 
$$2(m_1-s_1) = m_2 +1 -s_2.$$
Together with part (1) this implies that $s_2 \geq 2s_1$. 

The inequalities for $k_i^{(2)}$ translate into
\begin{align*}
m_1 &\geq 1 \\
m_2 &\geq m_1 +1 \\
m_3 - m_2 &\geq m_1 -1\\
\end{align*}
First note that all these inequalities are true for $\Gamma = \Gamma_1(4)$ by inspection, so assume that $\MMb(\Gamma)$ is representable. 
For the first inequality note that as in Corollary \ref{cor:dec2}, the quantity $2m_1$ is at least the number of cusps. This is clearly $\geq 1$ and it is indeed easy to see by \cite[Figure 3.3]{D-S05} that it is $\geq 4$ and thus $m_2\geq 2$. The second inequality follows form part (1) and $m_1 \geq 2$. By Proposition \ref{prop:riemann}, $m_3 - m_2 = \deg(f_n^*\omega)$ and $m_1-1 = \deg(f_n^*\omega)-s_2+s_1$. By part (2), we know that $s_2 \geq s_1$, so the third inequality follows. 

The inequalities for $k_i^{(3)}$ translate into
\begin{align*}
m_1 &\geq 2 \\
m_2 &\geq m_1 -1 
\end{align*}
Both were already shown above. 
\end{proof}

\begin{cor}\label{cor:decwild}
 Let $\Gamma$ be $\Gamma_1(n)$ for $n\geq 5$ or $\Gamma(n)$ for $n\geq 3$.
 We have a decomposition
 \[(f_n)_*\OO_{\MMb(\Gamma)} \cong \bigoplus_{i\in\Z} \bigoplus_{\kappa_i} (f_q)_*\OO_{\MMb_1(q)} \tensor \omega^{\tensor -i},\]
 for $q=4$, $5$ and $6$. For every $q>6$, there is an an arbitrarily large $n$ such that $(f_n)_*\OO_{\MMb(\Gamma_1(n))}$ does not decompose in the manner above. 
\end{cor}
\begin{proof}
 Consider first the case $q=4$ and the decomposition
 \[(f_n)_*\OO_{\MMb(\Gamma)} \cong \bigoplus_{i\in\Z} \bigoplus_{k_i} (f_2)_*\OO_{\MMb_1(2)} \tensor \omega^{\tensor -i}.\]
 For example, if $\Gamma=\Gamma_1(4)$ we get $k_0 = k_1=k_2=k_3= 1$ and all other $k_i = 0$. In the general case, we set $\kappa_0 = k_0 = 1$, $\kappa_1 = k_1-k_0$, $\kappa_2 = k_2-k_1$, $\kappa_3 = k_3-k_2$ and $\kappa_4 = k_7$ and all other $\kappa_i = 0$. We have to check that $k_i = \kappa_i + \kappa_{i-1} +\kappa_{i-2} + \kappa_{i-3}$, which is a straightforward computation from item 4 of Corollary \ref{cor:dec3}. Furthermore, $\kappa_i \geq 0$ by the last proposition. 
 
  Consider now the case $q=5$ or $q=6$ and the decomposition
 \[(f_n)_*\OO_{\MMb(\Gamma)} \cong \bigoplus_{i\in\Z} \bigoplus_{k_i} (f_3)_*\OO_{\MMb_1(3)} \tensor \omega^{\tensor -i}.\]
 For example, if $\Gamma=\Gamma_1(5)$ or $\Gamma =\Gamma_1(6)$ we get $k_0 = k_1=k_2= 1$ and all other $k_i = 0$. In the general case, we set $\kappa_0 = k_0 = 1$, $\kappa_1 = k_1-k_0$, $\kappa_2 = k_2-k_1$ and $\kappa_3 = k_5$. We have to check that $k_i = \kappa_i + \kappa_{i-1} +\kappa_{i-2}$, which is a straightforward computation from item 4 of Corollary \ref{cor:dec2}. Furthermore, $\kappa_i \geq 0$ by the last proposition. 
 
 It remains to show that for $q>6$ the vector bundle $(f_n)_*\OO_{\MMb(\Gamma)}$ cannot decompose into a sum of vector bundles of the form $(f_q)_*\OO_{\MMb_1(q)}\tensor \omega^{\tensor ?}$ in general. Let now $d_M$ be the degree of the map $\MM_1(m) \to \MM_{ell}$, which is also the rank of $(f_m)_*\OO_{\MMb_1(m)}$. By Lemma \ref{lem:degomega}, the function $d_m = d(\Gamma_1(m))$ is multiplicative and for primes $p$, we have $d_{p^k} = p^{2k-2}(p^2-1$); for example $d_5 = d_6 = 24$ and $d_7 = d_8 = 48$ and $d_9=72$. Moreover, $d_p > 24$ for a prime $p>5$. These facts imply that $d_q>24$ for every $q>6$. 
 
 For a decomposition as in the statement of the corollary, it is necessary that $d_q$ divides $d(\Gamma)$. Thus, we only need to show that for every $D>24$, there are infinitely many $p$ with $d_p$ not 
divisible by $d$. Every $D>24$ has a divisor of the form $d= 16$, $d=9$ or $d$ a prime that is at least $5$. Pick an $a$ that is coprime to $d$ and not congruent to $\pm 1 \mod d$; for $d=16$ we take $a=3$ and for $d=9$ we take $a=2$. By Dirichlet's prime number theorem, there are infinitely many primes $p$ such that $p\equiv a \mod d$. If $d$ is prime, this implies that $d$ does not divide $d_p = (p-1)(p+1)$. If $d =16$, this implies that $d_p \equiv 8 \mod d$, and for $d=9$, this implies $d_p \equiv 3\mod d$. In any case, $d$ does not divide $d_p$ for infinitely many primes $p$ and thus $D$ does not as well. 
\end{proof}

\begin{remark}
The only obstruction presented in the last proof for decomposing $(f_n)_*\OO_{\MMb_1(n)}$ into copies of $(f_m)_*\OO_{\MMb_1(m)}\tensor \omega^{\tensor ?}$ was that $d_m|d_n$. But in general it is not true that $d_m|d_n$ implies the possibility of such a decomposition. For example $d_7|d_{31}$, but $(f_{31})_*\OO_{\MMb_1(n)}$ does not decompose into copies of $(f_7)_*\OO_{\MMb_1(7)}\tensor \omega^{\tensor ?}$.
\end{remark}

\begin{remark}
 We remark that the last corollary also has implications over $\Z_{(l)}$ in the sense that if we have, for example, a splitting of $(f_n)_*\OO_{\MMb_1(n)}$ with $n\geq 4$ into summands of the form $(f_2)_*\OO_{\MMb_1(2)}\tensor \omega^{\tensor i}$  over $\Z_{(3)}$, we obtain also a splitting into summands of the form $(f_4)_*\OO_{\MMb_1(4)}\tensor \omega^{\tensor i}$ as we have just argued with the combinatorics of the decomposition numbers. 
\end{remark}

\subsection{Duality}\label{sec:duality}
In this section, we still work implicitly over a field $K$ of characteristic zero until further notice. Our goal is to explain that in certain cases the decomposition sequence $(l_i)$ of $(f_n)_*\OO_{\MMb_1(n)}$ displays a symmetry, which is easily noticeable in the tables of Appendix \ref{sec:tables}. Recall here that the $l_i$ are defined via the decomposition
\[(f_n)_*\OO_{\MMb_1(n)} \cong \bigoplus_{i\in\Z} \bigoplus_{l_i} \omega^{\tensor (-i)}.\]

\begin{prop}\label{prop:symmetry}
 If $f_n^*\omega^{\tensor k-10}$ is the dualizing sheaf for $\MMb_1(n)$, then we have 
 $$((f_n)_*\OO_{\MMb_1(n)})^\vee \cong (f_n)_*\OO_{\MMb_1(n)} \tensor \omega^{\tensor k}$$
 and this isomorphism holds if and only if the sequence of $l_i$ is of length $(k+1)$ and ``symmetric'', i.e.\ $l_{k-i} = l_i$. 
\end{prop}
\begin{proof}
 The ``if and only if'' is clear by the Krull--Schmidt theorem for $\MMb_{ell}$ (Proposition \ref{prop:Krull-Schmidt}).
 
Now assume that $f_n^*\omega^{\tensor k-10}$ is the dualizing sheaf for $\MMb_1(n)$. We have:
 \begin{align*}
  H^0\left(\MMb_{ell}; ((f_n)_*\OO_{\MMb_1(n)}\tensor\omega^{\tensor k}) \tensor \omega^{\tensor i}\right) &\cong
  H^0\left(\MMb_1(n); (f_n)^*\omega^{\tensor (k-10)} \tensor (f_n)^*\omega^{\tensor (10+i)}\right)\\
  &\cong  H^1\left(\MMb_1(n);(f_n)^*\omega^{\tensor (-10 - i)}\right)^{\vee} \\
  &\cong H^1\left(\MMb_{ell}; (f_n)_*\OO_{\MMb_1(n)}\tensor \omega^{\tensor (-10 - i)}\right)^{\vee} \\
  &\cong H^0\left(\MMb_{ell}; ((f_n)_*\OO_{\MMb_1(n)})^\vee \tensor \omega^{\tensor i}\right)
 \end{align*}
 Here, we use the projection formula, Serre duality on $\MMb_{ell}$ and $\MMb_1(n)$ and a degenerate form of the Leray spectral sequence. For vector bundles $\FF, \GG$ on $\MMb_{ell}$, we have $\FF\cong \GG$ iff $H^0(\MMb_{ell};\FF\tensor \omega^{\tensor i}) \cong H^0(\MMb_{ell};\GG\tensor \omega^{\tensor i})$ for all $i\in \Z$. This uses that $\FF$ and $\GG$ are of the form $\bigoplus_{j\in\Z} \bigoplus_{l_i}\omega^{\tensor (-i)}$. Thus, we see that
 \[((f_n)_*\OO_{\MMb_1(n)})^\vee \cong (f_n)_*\OO_{\MMb_1(n)}\tensor\omega^{\tensor k}.\qedhere\] 
\end{proof}

Now we turn to the question: When does it happen that $(f_n)^*\omega^{\tensor i}$ is the dualizing sheaf $\Omega^1_{\MMb_1(n)/K}$ for $\MMb_1(n)$ for some $i$? 

Recall that $H^1(\MMb_1(n); \Omega^1_{\MMb_1(n)/K}) \cong K$ and $\dim_K H^0(\MMb_1(n);\Omega^1_{\MMb_1(n)/K})$ is the genus $g$ of $\MMb_1(n)$. 
As $H^1(\MMb_1(n);(f_n)^*\omega^{\tensor i}) = 0$ for $i\geq 2$ by Proposition \ref{prop:coh}, 
\begin{align}\label{eq:dualiso}(f_n)^*\omega^{\tensor i}\cong\Omega^1_{\MMb_1(n)/K}\end{align}
can only happen for $i\leq 1$. Our strategy will be to treat our question step by step for $i\leq -1$, for $i=0$ and for $i=1$. 

Proposition \ref{prop:coh} also implies that \eqref{eq:dualiso} can only be true for $i\leq -1$ if $g = 0$. 
The genus zero cases are only $1\leq n \leq 10$ and $n =12$. For $n=9,10$ or $12$ we do not have symmetric decomposition sequences as seen in the first table in Section \ref{sec:tables}. In the other cases we have:
\begin{align*}
\Omega^1_{\MMb_{ell}/K} &\cong \omega^{\tensor -10} \\
\Omega^1_{\MMb_1(2)/K} &\cong (f_2)^*\omega^{\tensor -6} \\
\Omega^1_{\MMb_1(3)/K} &\cong (f_3)^*\omega^{\tensor -4} \\
\Omega^1_{\MMb_1(4)/K} &\cong (f_4)^*\omega^{\tensor -3} \\
\Omega^1_{\MMb_1(5)/K} &\cong (f_5)^*\omega^{\tensor -2} \\
\Omega^1_{\MMb_1(6)/K} &\cong (f_6)^*\omega^{\tensor -2} \\
\Omega^1_{\MMb_1(7)/K} &\cong (f_7)^*\omega^{\tensor -1} \\
\Omega^1_{\MMb_1(8)/K} &\cong (f_8)^*\omega^{\tensor -1} \\
\end{align*}
Indeed, in the non-representable cases $1\leq n\leq 4$ this is true by Example \ref{exa:moduli} and Theorem \ref{thm:fundamentalweighted}. For $5\leq n\leq 8$, we have $\MMb_1(n) \simeq\mathbb{P}^1$; on these stacks a line bundle is determined by its degree. The table above of canonical sheaves follows now from Lemma \ref{lem:degomega}.

A projective curve has genus $1$ if and only if the canonical sheaf agrees with the structure sheaf. The curve $\MMb_1(n)$ has genus $g=1$ if and only if $n=11,14,15$. This is well-known and can be easily proven using the genus formula
$$g = 1+ \frac{d_n}{24} - \frac14 \sum_{d|n}\phi(d)\phi(n/d)$$
from \cite[Sec 3.8+3.9]{D-S05} and easier analogues of the methods of Proposition \ref{prop:degreecomp}. 

Now assume that $\Omega^1_{\MMb_1(n)/K} \cong (f_n)^*\omega$. Then, in particular, we have $2g-2 = \deg (f_n)^*\omega$. We want to prove the following proposition, whose proof we learned from Viktoriya Ozornova.

\begin{prop}\label{prop:degreecomp}
 We have $2g-2 = \deg (f_n)^*\omega$ if and only if $n = 23,32,33,35,40$ or $42$.
\end{prop}

By Lemma \ref{lem:degomega} and the genus formula above, we see that we have just have to solve for $n$ in the equation 
\[
 \frac1{12}\sum_{d|n}d\phi(d)\phi(n/d)=\sum_{d|n} \varphi(d)\varphi\left(\frac{n}{d}\right).\qedhere
\]

 \begin{lemma}
The inequality
\[
 \frac{1}{12}\sum_{d|n} d\varphi(d)\varphi\left(\frac{n}{d}\right)>\sum_{d|n} \varphi(d)\varphi\left(\frac{n}{d}\right)
\]
holds for every natural number $n>144$.
 \end{lemma}
\begin{proof}We have the following chain of (in)equalities:
 \[
  \begin{aligned}
    \frac{1}{12}\sum_{d|n} d\varphi(d)\varphi\left(\frac{n}{d}\right)&= \frac{1}{12}\sum_{d|n} \frac{1}{2}\left(d+\frac{n}{d}\right)\varphi(d)\varphi\left(\frac{n}{d}\right)\\
&\geq   \frac{1}{12}\sum_{d|n} \sqrt{n}\varphi(d)\varphi\left(\frac{n}{d}\right)\\
&\stackrel{{\tiny \sqrt{n}>12}}{>}\sum_{d|n} \varphi(d)\varphi\left(\frac{n}{d}\right).\qedhere
  \end{aligned}
 \]
\end{proof}
\begin{proof}[Proof of proposition:]
 The proof can be easily finished by a computer search of all values up to $144$. For a proof by hand, one can argue as follows: Set $f(n) = \sum_{d|n} d\varphi(d)\varphi\left(\frac{n}{d}\right)$ and $g(n) = \sum_{d|n} \varphi(d)\varphi\left(\frac{n}{d}\right)$. Both functions are multiplicative for coprime integers; thus, we only have to understand them on prime powers smaller or equal to $144$. For $p$ a prime, we have $\frac1{12}f(p) = g(p)$ exactly for $p=23$ and $\frac1{12}f(p)>g(p)$ for $p>23$; thus, the equation cannot have solutions $n$ with prime factors $>23$. The other solutions can now easily be deduced from the following table with the relevant values of $\frac{g(p^k)}{f(p^k)}$:
\begin{center}
\renewcommand{\arraystretch}{1.2}
\begin{tabular}{|c|c|c|c|c|c|c|c|c|}
\hline
$k$, $p$ & $2$ & $3$ & $5$ & $7$ & $11$ & $13$ & $17$ & $19$\\ \hline
$1$ & $\frac{2}{3}$ & $\frac{1}{2}$ & $\frac{1}{3}$ & $\frac{1}{4}$ & $\frac{1}{6}$ & $\frac{1}{7}$ & $\frac{1}{9}$ & $\frac{1}{10}$\\ \hline
$2$ & $\frac{5}{12}$ & $\frac{2}{9}$ & $\frac{7}{75}$ & $\frac5{98}$ & $\frac8{363}$ &  &  & \\ \hline
$3$ & $\frac{1}{4}$ & $\frac{5}{54}$ & $\frac{3}{125}$ &  &  &  &  & \\ \hline
$4$ & $\frac{7}{48}$ & $\frac{1}{27}$ &  &  &  &  &  & \\ \hline
$5$ & $\frac{1}{12}$ &  &  &  &  &  &  & \\ \hline
$6$ & $\frac3{64}$ & & & & & & & \\ \hline
$7$ & $\frac{5}{192}$ & & & & & & & \\ \hline
\end{tabular}
\end{center}
\end{proof}

A line bundle $\LL$ on $\MMb_1(n)$ is isomorphic to $\Omega^1_{\MMb_1(n)/K}$ if and only if $\deg \LL = \deg \Omega^1_{\MMb_1(n)/K}$ and $\dim_K H^1(\MMb_1(n); \LL) = 1$ (indeed, then $\Omega^1_{\MMb_1(n)/K} \tensor \LL^{-1}$ is a line bundle of degree $0$ with a nonzero global section by Serre duality and has thus to be the structure sheaf by \cite[Lemma IV.1.2]{Har77}). In the case of $\LL = (f_n)^*\omega$, we have $\dim_K H^1(\MMb_1(n); \LL) = s_1$, the dimension of the space of weight-$1$ cusp forms for $\Gamma_1(n)$. Among the values from Proposition \ref{prop:degreecomp}, we only have $s_1 = 1$ for $n=23$ as a simple \texttt{MAGMA} computation shows. (We remark that the case $n=23$ was already treated in \cite[Section 1]{Buz} by hand.)

Thus, we have proven the following proposition:
\begin{prop}
 We have $\Omega^1_{\MMb_1(n)/K} \cong (f_n)^*\omega^{\tensor i}$ for some $i$ if and only if 
 \[n=1,2,3,4,5,6,7,8, 11, 14, 15 \text{ or }23.\]
\end{prop}

\begin{cor}\label{cor:Omega}
 If we view $\MMb_1(n)$ as being defined over $\Z[\frac1n]$, we have
 $$\Omega^1_{\MMb_1(n)/\Z[\frac1n]} \cong (f_n)^*\omega^{\tensor i}$$
 for some $i$ if and only if $n=1,2,3,4,5,6,7,8, 11, 14, 15 \text{ or }23.$
\end{cor}
\begin{proof}
 The restriction functor $\Pic(\MMb_1(n)) \to \Pic((\MMb_1)_\Q)$ is injective by Proposition \ref{prop:liu}. 
\end{proof}

\section{Applications to topological modular forms}\label{sec:tmf}
Goerss, Hopkins and Miller defined a sheaf of $E_\infty$-ring spectra $\OO^{top}$ on the \'etale site of the compactified compactified moduli stack $\MMb_{ell}$ of elliptic curves (see \cite{TMF}). It satisfies $\pi_0\OO^{top} \cong \OO_{\MMb_{ell}}$ and more generally $\pi_{2k}\OO^{top} \cong \omega^{\tensor k}$ and $\pi_{2k-1}\OO^{top} = 0$. As $\MM_0(n), \MM_1(n)$ and $\MM(n)$ are \'etale over $\MMb_{ell}$, we can define
\begin{align*}
 TMF_1(n) &= \OO^{top}(\MM_1(n)), \\
 TMF(n) &= \OO^{top}(\MM(n)), \\
 TMF_0(n) &= \OO^{top}(\MM_0(n)).
\end{align*}

In contrast, $\MMb_0(n)$ etc.\ are not \'etale over $\MMb_{ell}$. As a remedy, Hill and Lawson extended in \cite{HL13} the sheaf $\OO^{top}$ to the log-\'etale site of the compactified moduli stack $\MMb_{ell}$. Thus, we can additionally define
\begin{align*}
 Tmf_1(n) &= \OO^{top}(\MMb_1(n)), \\
 Tmf(n) &= \OO^{top}(\MMb(n)), \\
 Tmf_0(n) &= \OO^{top}(\MMb_0(n)).
\end{align*}
Everywhere here, the integer $n$ is implicitly inverted. The next step would be to define connective versions, which we will leave to a future treatment. 

The goal of this section is to draw topological corollaries from our algebraic decomposition and duality theorems and explain the significance to equivariant $TMF$. 

\subsection{Decompositions}
Our algebraic decomposition theorems rather easily imply topological counterparts. For simplicity of notation, we will formulate the next theorem only for $TMF_1(n)$ and $Tmf_1(n)$, but we will indicate later how to formulate analogous theorems for $TMF(n)$ and $TMF_0(n)$ (if $n$ is squarefree and $l>3$ or $l$ does not divide $\phi(n)$) and likewise for their ``compactified'' versions.

For a spectrum $E$, we will denote by $\widehat{E}_l = \holim_n E/l^n$ its $l$-completion (see e.g.\ \cite[Section II.7.3]{SAG} for a conceptual treatment). 
\begin{thm}\label{thm:main}
 Let $n\geq 4$ and $l$ be a prime not dividing $n$. Denote furthermore by $d_n$ the degree of $f_n\colon \MMb_1(n) \to \MMb_{ell}$. 
 \begin{enumerate}
  \item\label{item:uncompactified} The $\widehat{TMF}_l$-module $\widehat{TMF_1(n)}_l$ decomposes into copies of 
  \begin{itemize}
   \item $\frac{d_n}{24}$ copies of $\widehat{TMF_1(3)}_l \oplus \Sigma^2 \widehat{TMF_1(3)}_l \oplus \Sigma^4 \widehat{TMF_1(3)}_l$ if $l=2$
   \item $\frac{d_n}{12}$ copies of $\widehat{TMF_1(2)}_l\oplus \Sigma^2 \widehat{TMF_1(2)}_l\oplus \Sigma^4 \widehat{TMF_1(2)}_l \oplus \Sigma^6 \widehat{TMF_1(2)}_l$ if $l=3$
   \item $d_n$ copies of even suspensions of $\widehat{TMF}_l$.
  \end{itemize}
  \item\label{item:compactified} The $Tmf_{(l)}$-module $Tmf_1(n)_{(l)}$ decomposes into 
    \begin{itemize}
   \item $\frac{d_n}{8}$ copies of even suspensions of $Tmf_1(3)_{(l)}$ if $l=2$ (with $\Sigma^{2i}Tmf_1(3)_{(l)}$ occuring $k_i$ times with notation as in Corollary \ref{cor:dec2})
   \item $\frac{d_n}{3}$ copies of even suspensions of $Tmf_1(2)_{(l)}$ if $l=3$ (with with $\Sigma^{2i}Tmf_1(2)_{(l)}$ occuring $k_i$ times with notation as in Corollary \ref{cor:dec3})
   \item $d_n$ copies of even suspensions of $Tmf_{(l)}$ if $l>3$ (with $\Sigma^{2i}Tmf_{(l)}$ occuring $l_i$ times with notation as in Proposition \ref{prop:dec5})
  \end{itemize}
    if $\pi_1Tmf_1(n)$ has no $l$-torsion (or, equivalently, every mod $l$ weight $1$ cusp form for $\Gamma_1(n)$ lifts to an integral $\Gamma_1(n)$ weight $1$ cusp form as in Remark \ref{rem:cuspy}).
 \end{enumerate}
\end{thm}
\begin{proof}
 We implicitly localize at $l$ everywhere. We will first consider the case $l=3$ of Item \ref{item:compactified} as the cases $l=2$ and $l>3$ are very similar. Consider the quasi-coherent $\OO^{top}$-module $\EE = (f_n)_*(f_n)^*\OO^{top}$, whose global sections are $Tmf_1(n)$. We know that $\OO^{top}$ is even and $\pi_{2k}\OO^{top} \cong \omega^{\tensor k}$. Moreover, taking (sheafified) homotopy groups is compatible with pushforward and pullback along flat maps. Thus, $\EE$ is even as well and 
 $$\pi_{2k}\EE \cong (f_n)_*(f_n)^*\omega^{\tensor k} \cong ((f_n)_*\OO_{\MMb_1(n)}) \tensor \omega^{\tensor k}.$$
 As every quasi-coherent sheaf on $\MMb_1(n)$ has cohomological dimension $\leq 1$ by Proposition \ref{prop:basicprops}, the same is true for pushforwards of such sheaves to $\MMb_{ell}$ by the Leray spectral sequence as $f_n$ is affine. In particular, the descent spectral sequence is concentrated in lines $0$ and $1$ and this implies that
 $$\pi_1Tmf_1(n) \cong H^1(\MMb_{ell}, \pi_2\EE) \cong H^1(\MMb_1(n), (f_n)^*\omega)$$
 and this is $l$-torsionfree by assumption.
 
 Thus, by Theorem \ref{thm:deccompact} we obtain a splitting of the form
 $$(f_n)_*\OO_{\MMb_1(n)} \cong \bigoplus_{i\in\Z} \bigoplus_{k_i} (f_2)_*\OO_{\MMb_1(2)}\tensor \omega^{\tensor -i}.$$
 Set 
 $$\FF = \bigoplus_{i\in\Z} \bigoplus_{k_i} (f_2)_*(f_2)^*\Sigma^{2i}\OO^{top}$$
 and note that $\Gamma(\FF) = \bigoplus_{i\in\Z} \bigoplus_{k_i} \Sigma^{2i}Tmf_1(2)$. As we have just seen, there is an isomorphism $h\colon \pi_0\FF \xrightarrow{\cong} \pi_0\EE$. Note that $\FF$ is even as well and $\pi_{2k}\FF \cong \pi_0\FF \tensor \omega^{\tensor k}$ so that we actually have $\pi_*\EE \cong \pi_*\FF$.

 There is a descent spectral sequence 
 \[H^q(\MMb_{ell}; \mathcal{H}om_{\pi_*\OO^{top}}(\pi_*\FF,\pi_{*+p}\EE)) \Rightarrow \pi_{p-q}\mathrm{Hom}_{\OO^{top}}(\FF,\EE)\]
 for $\mathcal{H}om$ the Hom-sheaf; indeed, $\FF$ is locally free and thus 
$$\pi_p\mathcal{H}om_{\OO^{top}}(\FF,\EE) \cong \mathcal{H}om_{\pi_*\OO^{top}}(\pi_*\FF,\pi_{*+p}\EE).$$ 
 We claim that $\mathcal{H}om_{\pi_*\OO^{top}}(\pi_*\FF,\pi_{*+p}\EE) \cong \mathcal{H}om_{\OO_{\MMb_{ell}}}(\pi_0\FF,\pi_p\EE)$ has cohomological dimension $\leq 1$. Indeed, it is the sum of vector bundles of the form 
 $$\mathcal{H}om_{\OO_{\MMb_{ell}}}((f_2)_*\OO_{\MMb_1(2)}\tensor \omega^{\tensor j}, (f_n)_*\OO_{\MMb_1(n)}),$$
each of which is isomorphic to
$$(f_n)_*\mathcal{H}om_{\OO_{\MMb_1(n)}}((f_n)^*((f_2)_*\OO_{\MMb_1(2)}\tensor \omega^{\tensor j}), \OO_{\MMb_1(n)}).$$
Thus, there cannot be any differentials or extension issues in the spectral sequence. 
 The homomorphism $h$ defines an element in the spot $p=q=0$. As there are no differentials, $h$ lifts uniquely (up to homotopy) to a morphism $\FF \to \EE$ that induces an isomorphism on $\pi_*$ and is thus an equivalence. 
 
 The other two parts of \ref{item:compactified} are very similar. For the uncompactified case \ref{item:uncompactified}, we start with a few remarks about completions. Note first that $l$-completion commutes with homotopy limits. In particular, if we define $\widehat{\OO}^{top}_l$ by $\widehat{\OO}^{top}_l(U) = (\OO^{top}(U))^{\wedge}_l$ for every Deligne--Mumford stack $U$ \'etale over $\MMb_{ell}$, we see that $\widehat{\OO}^{top}_l$ is still a sheaf on the \'etale site of $\MMb_{ell}$ with global sections $\widehat{TMF}_l$. By \cite[Cor 7.3.5.2]{SAG}, $l$-completion is monoidal and thus $\widehat{\OO}^{top}_l$ is at least a sheaf of $A_\infty$-ring spectra. Working everywhere with $l$-completions, the proof is from this point on analogous to the one in the compactified case if we use Theorem \ref{thm:decuncompact} and Corollary \ref{cor:equal}. 
\end{proof}

For space reasons, we will only give the case $l=3$ of the following two results, but there are analogous results for $l=2$ and $l>3$. 
\begin{thm}
 Let $n\geq 2$ be not divisible by $l=3$. Then the $\widehat{TMF}_3$-module $\widehat{TMF(n)}_3$ decomposes into copies of $\Sigma^{2?}\widehat{TMF_1(2)}_3$ and more precisely into copies of 
 $$\widehat{TMF_1(2)}_l\oplus \Sigma^2 \widehat{TMF_1(2)}_l\oplus \Sigma^4 \widehat{TMF_1(2)}_l \oplus \Sigma^6 \widehat{TMF_1(2)}_l$$
 if $n\geq 3$. Likewise, the $Tmf_{(3)}$-module $Tmf(n)_{(3)}$ decomposes into copies of $\Sigma^{2?}TMF_1(2)_{(3)}$ if $\pi_1Tmf(n)$ has no $3$-torsion (or, equivalently, every mod $3$ weight $1$ cusp form for $\Gamma(n)$ lifts to an integral $\Gamma(n)$ weight $1$ cusp form as in Remark \ref{rem:cuspy}).
\end{thm}
\begin{thm}
 Let $n\geq 2$ be squarefree and assume that neither $n$ nor $\phi(n)$ are divisible by $l=3$. Then the $\widehat{TMF}_3$-module $\widehat{TMF_0(n)}_3$ decomposes into copies of $\Sigma^{2?}\widehat{TMF_1(2)}_3$. Likewise, the $Tmf_{(3)}$-module $Tmf_0(n)_{(3)}$ decomposes into copies of $\Sigma^{2?}\widehat{TMF_1(2)}_3$ if $\pi_1Tmf_0(n)$ has no $3$-torsion (or, equivalently, every mod $3$ weight $1$ cusp form for $\Gamma_0(n)$ lifts to an integral $\Gamma_0(n)$ weight $1$ cusp form as in Remark \ref{rem:cuspy}).
\end{thm}

\begin{remark}
 There are several possible extensions we already have partial results on and plan to treat in future work. The first is to extend this decomposition theorems to the connective versions $tmf_1(n)$, $tmf(n)$ and $tmf_0(n)$. The second is to take the group action by $(\Z/n)^\times$ on $TMF_1(n)$ and $Tmf_1(n)$ into account.
\end{remark}

\subsection{Duality}
In this section, we will investigate possible Anderson self-duality of $Tmf_1(n)$. Let us recall the definition of Anderson duality, which was first studied by Anderson (only published in mimeographed notes \cite{Anderson}) and Kainen \cite{Kainen}, mainly for the purpose of universal coefficient sequences, and further investigated in the context of topological modular forms in \cite{Sto12}, \cite{H-M15} and \cite{G-M16}.  

For an injective abelian group $J$, the functor
\[\mathrm{Spectra} \to \text{graded abelian groups},\quad X \mapsto \Hom_\Z(\pi_{-*}X, J)\]
is representable by a spectrum $I_J$, as follows from Brown representability. If $A$ is an abelian group and $A \to J^0 \to J^1$ an injective resolution, we define the spectrum $I_A$ to be the fiber of $I_{J^0}\to I_{J^1}$. Given a spectrum $X$, we define its \emph{$A$-Anderson dual} $I_AX$ by $F(X, I_A)$. It satisfies for all $k\in\Z$ the following functorial short exact sequence:
\[0 \to \Ext^1_\Z(\pi_{-k-1}X, A) \to \pi_kI_AX \to \Hom_\Z(\pi_{-k}X, A) \to 0.\]
Note that if $A$ is a subring of $\Q$ and $\pi_{-k-1}X$ is a finitely generated $A$-module, the $\Ext$-group is just the torsion in $\pi_{-k-1}X$. More generally, we obtain for spectra $E$ and $X$ a universal coefficient sequence
\[0 \to \Ext^1_\Z(E_{k-1}X, A) \to (I_AE)^kX \to \Hom_\Z(E_kX, A) \to 0,\]
which is most useful in the case of \emph{Anderson self-duality}, i.e.\ if $I_AE$ is equivalent to $\Sigma^mE$ for some $m$, as then the middle-term can be replaced by $E^{k+m}X$. Such Anderson self-duality is, for example, true for $E = KU$ (with $m=0$), $E = KO$ (with $m=4$) by \cite{Anderson} and $E = Tmf$ (with $m=21$) by \cite{Sto12}. We want to investigate which $Tmf_1(n)$ are Anderson self-dual. \\

Recall that in Corollary \ref{cor:Omega}, we gave conditions when the dualizing sheaf of $\MMb_1(n)$ is a power of $f_n^*\omega$ for $f_n\colon \MMb_1(n) \to \MMb_{ell,\Z[\frac1n]}$ the structure map. We will explain how this implies Anderson self-duality for $Tmf_1(n)$ in these cases once we know that $\pi_*Tmf_1(n)$ is torsionfree. We will assume throughout that $n\geq 2$. 

\begin{lemma}\label{lem:torsionfree2}
 If $\Omega^1_{\MMb_1(n)/\Z[\frac1n]} \cong (f_n)^*\omega^{\tensor i}$, then the cohomology groups $H^1(\MMb_1(n); (f_n)^*\omega^{\tensor j})$ are torsionfree for all $j$. 
\end{lemma}
\begin{proof}For $j\neq 1$, this follows by Proposition \ref{prop:coh}. It remains to show it for $j=1$. We can assume that $n\geq 5$ so that $\MMb_1(n)$ is representable as the other cases are easily dealt with by hand. By Section \ref{sec:duality} we furthermore know that $i\leq 1$. 

If $i \leq 0$, then $H^1(\MMb_1(n)_{\Q};f_n^*\omega) = 0 = H^1(\MMb_1(n)_{\F_p};f_n^*\omega)$ by Serre duality for all primes $p$ not dividing $n$ because there are no modular forms of negative weight by Proposition \ref{prop:coh} again. Thus, by cohomology and base change $H^1(\MMb_1(n);f_n^*\omega) = 0$ (see e.g.\ \cite[Corollary 12.9]{Har77}). If $i = 1$, then Grothendieck duality states that $H^1(\MMb_1(n);f_n^*\omega) \cong \Z[\frac1n]$. 
\end{proof}
This implies that $\pi_*Tmf_1(n)$ is torsionfree if $\Omega^1_{\MMb_1(n)/\Z[\frac1n]} \cong (f_n)^*\omega^{\tensor i}$ for some $i$. Indeed, as $\MMb_1(n)$ has cohomological dimension $1$, the descent spectral sequence for $Tmf_1(n)$ collapses (as in the proof of Theorem \ref{thm:main}) and we obtain
$$\pi_{2i}Tmf_1(n) \cong H^0(\MMb_1(n); (f_n)^*\omega^{\tensor i}) \;\text{ and }\; \pi_{2i-1}Tmf_1(n)\cong H^1(\MMb_1(n); (f_n)^*\omega^{\tensor i}).$$
The former is torsionfree because $\MMb_1(n)$ is flat over $\Z[\frac1n]$ and the latter by the last lemma. 

The following lemma will also be useful.
\begin{lemma}\label{lem:rational}
Let $A$ be a subring of $\Q$ and $X$ a spectrum whose homotopy groups are finitely generated $A$-modules. Then $(I_AX)_{\Q} \to I_{\Q}X_{\Q}$ is an equivalence, where $X_{\Q}$ denotes the rationalization.
\end{lemma}
\begin{proof}
Recall that $I_{\Q}X$ is defined as $F(X,I_{\Q})$, which is equivalent to $F(X_{\Q}, I_{\Q})$ as $I_{\Q}$ is rational. We have to show that the natural map $I_AX \to I_{\Q}X_{\Q}$ is an isomorphism on homotopy groups after rationalization. This boils down to the facts that for a finitely generated $A$-module $M$ we have
$$\Ext_A(M,A) \tensor \Q = 0$$ 
and that 
$$\Hom(M,A)\tensor \Q \to \Hom(M, \Q) \cong \Hom_{\Q}(M\tensor \Q, \Q)$$
is an isomorphism.
\end{proof}

\begin{prop}
 We have $I_{\Z[\frac1n]}Tmf_1(n) \simeq \Sigma^l Tmf_1(n)$ as $Tmf_1(n)$-modules for $n\geq 2$ if and only if $l$ is odd and $(f_n)^*\omega^{\tensor (-k)}$ is isomorphic to $\Omega^1_{\MMb_1(n)/\Z[\frac1n]}$ for $k = (l-1)/2$.
\end{prop}
\begin{proof}
We will denote throughout the proof the sheaf $(f_n)^*\omega$ by $\omega$. Furthermore, we abbreviate $Tmf_1(n)$ to $R$. As noted before, we have $\pi_{2i}R \cong H^0(\MMb_1(n); \omega^{\tensor i})$ and $\pi_{2i-1}R \cong H^1(\MMb_1(n); \omega^{\tensor i})$. 

First suppose that $\omega^{\tensor (-k)}$ is a dualizing sheaf for $\MMb_1(n)$. Then $H^1(\MMb_1(n); \omega^{\tensor (-k)}) \cong \Z[\frac1n]$ and the pairing
$$H^0(\MMb_1(n); \omega^{\tensor -i-k}) \tensor H^1(\MMb_1(n); \omega^{\tensor i}) \to  H^1(\MMb_1(n); \omega^{\tensor -k})\cong \Z[\frac1n]$$
is perfect by Serre duality; note here that all occuring groups are finitely generated and torsionfree by Lemma \ref{lem:torsionfree2}. As note above, this implies that $\pi_*R$ is torsionfree as well.

Choose a generator $D$ of $H^1(\MMb_1(n); \omega^{\tensor -k})$.  This is represented by a unique element in $\pi_{-2k-1}R \cong \Z\left[\tfrac1n\right]$, which we will also denote by $D$. Denote by $\delta$ the element in $\pi_{2k+1} I_{\Z\left[\tfrac1n\right]} R$ with $\phi(\delta)(D) = 1$, where $\phi\colon \pi_{2k+1} I_{\Z\left[\tfrac1n\right]} R \xrightarrow{\cong} \Hom(\pi_{-2k-1}R,\Z\left[\tfrac1n\right])$. The element $\delta$ induces an $R$-linear map $\widehat{\delta}\co \Sigma^{2k+1}R\to I_{\Z\left[\tfrac1n\right]} R$. 

We obtain a diagram
\begin{equation}\label{AndersonSquare}
\begin{gathered}
 \xymatrix@C=1.3em{\pi_{i-2k-1}R \tensor \pi_{-i}R \ar[rr]^-{\widehat{\delta}_*\tensor \id}\ar[d] && \pi_iI_{\Z\left[\tfrac1n\right]}R\tensor \pi_{-i}R \ar[rr]_-{\cong}^-{\phi\tensor\id} && \Hom(\pi_{-i}R, \Z\left[\tfrac1n\right]) \tensor \pi_{-i}R\ar[d]^{\ev} \\
\pi_{-2k-1}R\ar[rrrr]^{\phi(\delta)}_{\cong} &&&& \Z\left[\tfrac1n\right], }
\end{gathered}
\end{equation}
which is commutative up to sign. 

The left vertical map is a perfect pairing because of Serre duality (as described above), as is the right vertical map by definition. Thus, the map 
$$\widehat{\delta}_*\colon \pi_{i-2k-1}R \to \pi_kI_{\Z\left[\tfrac1n\right]}R$$
is an isomorphism for all $i$. This shows that $\widehat{\delta}$ is an equivalence of $R$-modules. 

Now assume on the other hand that there is an equivalence $I_{\Z[\frac1n]}R \simeq \Sigma^l R$ as $R$-modules. By Lemma \ref{lem:rational}, this implies an equivalence $I_{\Q}R_\Q \simeq \Sigma^l R_\Q$ of $R_\Q$-modules. In the following, we will rationalize everything implicitly.

If $l$ is even, this implies
$$H^0(\MMb_1(n); \omega^{\tensor i}) \cong \pi_{2i}Tmf_1(n) \cong (\pi_{-2i}\Sigma^lTmf_1(n))^{\vee} \cong H^0(\MMb_1(n); \omega^{\tensor (-l/2-i)})^{\vee}.$$
As there are no modular forms of negative weight, this would imply that $H^0(\MMb_1(n); \omega^{\tensor i})$ is zero for $i$ big; this is absurd as the ring of modular forms does not contain nilpotent elements. Thus, $l$ can be written as $2k+1$. 

 Let 
$$D \in H^1(\MMb_1(n); \omega^{\tensor -k}) \cong \pi_{-2k-1}R$$
be the element corresponding under the isomorphism
\begin{align*}
 \pi_{-2k-1}R &\cong \pi_0\Sigma^{2k+1}R \\
 &\cong \pi_0I_{\Q}R \\
 &\cong \Hom_{\Q}(\pi_0R, \Q) \\
 &\cong \Hom_{\Q}(\Q,\Q)
\end{align*}
to $1$. Consider now again the diagram (\ref{AndersonSquare}), but now tensored with $\Q$. Now all the horizontal arrows are isomorphisms and the right vertical map is a perfect pairing. Therefore, the left hand vertical arrow is a perfect pairing as well. This implies that 
$$H^0(\MMb_1(n); \omega^{\tensor -i-k}) \tensor H^1(\MMb_1(n); \omega^{\tensor i}) \to  H^1(\MMb_1(n); \omega^{\tensor -k})\cong \Q$$
is a perfect pairing. As $\omega$ is ample by Proposition \ref{prop:basicprops}, one can repeat the proof of \cite[Thm 7.1b]{Har77} to see that $\omega^{\tensor -k}$ is dualizing on $\MMb_1(n)_{\Q}$ and thus isomorphic to the dualizing sheaf $\Omega^1_{\MMb_1(n)_{\Q}/\Q} = \left(\Omega^1_{\MMb_1(n)/\Z[\frac1n]}\right)_{\Q}$. By Proposition \ref{prop:liu}, this implies that $\omega^{\tensor (-k)} \cong \Omega^1_{\MMb_1(n)/\Z[\frac1n]}$ also before rationalizing. 
\end{proof}

This gives the following theorem:
\begin{thm}
 We have $I_{\Z[\frac1n]}Tmf_1(n) \simeq \Sigma^l Tmf_1(n)$ as $Tmf_1(n)$-modules for some $l$ if and only if 
 
   \begin{tabular}{ll}
      $n = 1$ &with\;  $l=21$, \\
  $n = 2$ & with\;  $l=13$, \\
  $n = 3$ & with\; $l = 9$, \\
  $n=4$ & with\; $l = 7$,\\
  $n = 5,6$ & with\;  $l = 5$, \\
  $n = 7,8$ & with\; $l = 3$,\\
  $n = 11,14,15$ &with\; $l = 1$, or \\
  $n = 23$ & with\;  $l= -1$.
   \end{tabular}
\end{thm}
\begin{proof}
 The only case not dealt with by the last proposition and Corollary \ref{cor:Omega} is the case $n=1$. With $2$ inverted, this was shown in \cite{Sto12}. Without inverting $2$, this was announced in \cite{Sto14}. 
\end{proof}

\subsection{Connections to Equivariant $TMF$}\label{sec:eq}
Before we come to the equivariant part of the story, we summarize Lurie's viewpoint of $TMF$ from \cite{Lur07} in the terminology of \cite{SAG}. 

For us, the term \emph{derived stack} means a nonconnective spectral Deligne--Mumford stack in the sense of \cite[Section 1.4.4]{SAG}. As data, it consists of an $\infty$-topos $\XX$ with a sheaf $\OO_{\XX}$ of $E_\infty$-ring spectra on it. An $\infty$-topos $\XX$ has an underlying $1$-topos $\XX^{\heartsuit}$, which is the full sub-$\infty$-category of discrete objects. We call $\XX$ \emph{$1$-localic} if the natural morphism $\XX \to \Shv(\XX^{\heartsuit})$ into space-valued sheaves is an equivalence. As $\XX \mapsto \XX^{\heartsuit}$ defines by \cite[Section 6.4.5]{HTT} an equivalence between the $\infty$-categories of $1$-localic $\infty$-topoi and $1$-topoi, we will not distinguish between $1$-topoi and $1$-localic $\infty$-topoi in the following. 

Denote the underlying topos of $\MM_{ell}$ by $\MM$ and recall the notions of (pre)oriented elliptic curves over $E_\infty$-rings from \cite{Lur07}. Lurie constructs in \cite{Lur07} a derived stack $\MM_{ell}^{top} = (\MM,\OO^{top})$ representing the moduli problem of oriented elliptic curves over $E_\infty$-rings and sketches why this sheaf $\OO^{top}$ agrees with the one constructed by Goerss--Hopkins--Miller. By this moduli interpretation we obtain a universal oriented elliptic curve $\EE^{top}$ over $\MM_{ell}^{top}$, which is characterized by the property that for every $E_\infty$-ring $R$ and map $f\colon \Spec R \to \MM_{ell}^{top}$, the pullback of $\EE^{top}$ is exactly the oriented elliptic curve over $R$ classified by $f$. 

The following lemma is certainly well-known to experts. 

\begin{lemma}
  The underlying $\infty$-topos of $\EE^{top}$ is $1$-localic and agrees with the underlying $1$-topos of the universal elliptic curve $\EE^{uni}$ over $\MM_{ell}$ and we can identify $\OO_{\EE^{uni}}$ with $\pi_0$ of the structure sheaf $\OO^{top}_{\EE}$ of $\EE^{top}$. 
\end{lemma}
\begin{proof}
Denote by $\FF$ the functor that associates to each $E_\infty$-ring spectrum $A$ the space of preoriented elliptic curves over $A$. By Proposition 4.1 and the discussion thereafter in \cite{Lur07}, the functor $\FF$ is represented by a derived stack $\MM_{ell}^{pre} = (\MM, \OO^{pre})$ and $\OO^{pre}$ is connective with $\pi_0\OO^{pre} = \OO_{\MM_{ell}}$. Thus, $\MM_{ell}$ can be identified with the $0$-truncation of $\MM_{ell}^{pre}$ in the sense of \cite[Section 1.4.6]{SAG}. The induced map $\MM_{ell} \to \MM_{ell}^{pre}$ classifies the unique preorientation on a classical elliptic curve. 

We claim that the underlying $\infty$-topos of the universal preoriented elliptic curve $\EE^{pre}$ agrees with $\EE$. Indeed, let $X$ be a $0$-truncated spectral Deligne--Mumford stack with a map $X \to \EE^{pre}$. The composition $X \to \MM_{ell}^{pre}$ factors by \cite[Cor 1.4.6.4]{SAG} essentially uniquely through the $0$-truncation of $\MM_{ell}^{pre}$, which is $\MM_{ell}$. Thus, a map $X \to \EE^{pre}$ is the same as a map into the pullback of $\EE^{pre}$ along $\MM_{ell} \to \MM_{ell}^{pre}$, which is $\EE^{uni}$. Thus, the $0$-truncation of $\EE^{pre}$ coincides with $\EE^{uni}$ and we can write $\EE^{pre} = (\EE,\OO^{pre}_{\EE})$.

 The universal oriented elliptic curve is the pullback of
 \[
  \xymatrix{
  & (\EE, \OO^{pre}_{\EE})=\EE^{pre} \ar[d] \\
  (\MM_{ell},\OO^{top}) \ar[r] & (\MM_{ell},\OO^{pre}) = \MM_{ell}^{pre},
  }
 \]
where the map on underlying Deligne--Mumford stacks $\MM_{ell}\to \MM_{ell}$ is the identity by \cite[Section 4.1]{Lur07}. It follows that this pullback is given as $(\EE,\OO^{top}_{\EE})$ with 
$$\OO^{top}_{\EE} = \OO^{pre}_{\EE}\sm_{\OO^{pre}}\OO^{top}.$$ 
As $\EE^{pre} \to \MM_{ell}^{pre}$ is flat by the definition of an elliptic curve, we have \[\pi_0 \OO^{top}_{\EE} \cong \pi_0 \OO_{\EE}^{pre} \tensor_{\pi_0\OO^{pre}} \pi_0\OO^{top} \cong \OO_{\EE^{uni}}.\qedhere\]
\end{proof}

In \cite{Lur07}, Lurie uses the universal derived elliptic curve to construct equivariant elliptic cohomology and we will recall features of this construction now.  

Let $G$ be an abelian compact Lie group. Define $\MM_G$ to be the moduli stack of elliptic curves $E$ together with a morphism $\widehat{G} \to E$ from the constant group scheme of characters $\widehat{G} = \Hom(G,\C^\times)$ of $G$.
Lurie defines $\MM_G^{top}$ as the derived mapping stack $\Hom(\widehat{G}, \EE^{top})$. As $\widehat{G}$ is a finitely generated abelian group, this can concretely be constructed as follows: We set $\Hom(\Z,\EE^{top})$ to be $\EE^{top}$ itself and $\Hom(\Z/n, \EE^{top})$ to be the $n$-torsion points of $\EE^{top}$, i.e.\ the fiber product 
\[\xymatrix{\EE^{top}[n] \ar[r]\ar[d] & \EE^{top} \ar[d]^{[n]}\\
\MM_{ell}^{top} \ar[r]^e & \EE^{top}.
}\]
Because $[n]$ is flat by \cite[Theorem 2.3.1]{K-M85}, this fiber product descends to underlying Deligne--Mumford stacks so that the underlying stack of $\EE^{top}[n]$ are the $n$-torsion points of $\EE^{uni}$, i.e.\ $\EE^{uni}[n]$.
Furthermore, $\Hom(H_1\oplus H_2,\EE^{top})$ is $\Hom(H_1,\EE^{top})\times_{\MM_{ell}^{top}} \Hom(H_2, \EE^{top}).$  
It will follow from Lemma \ref{lem:FiniteFlat} that this fiber product again descends to the underlying stacks so that the underlying stack of $\MM_G^{top}$ is $\MM_G = \Hom(\widehat{G}, \EE^{uni})$. We denote the structure sheaf of $\MM_G^{top}$ by $\OO^{top}_G$.

Lurie defines in \cite{Lur07} $\infty$-functors $$\FF_G\colon (\text{finite }G\text{-CW complexes})^{op} \to \QCoh(\MM_G^{top})$$
satisfying the following properties:
\begin{enumerate}
 \item $\FF_G$ sends finite homotopy colimits of $G$-CW complexes to finite homotopy limits,
 \item $\FF_G(\pt) = \OO^{top}_G$, and
 \item For $H\subset G$ and $X$ a finite $H$-CW complex, we have $\FF_G(X\times_H G) \simeq f_*\FF_H(X)$ for $f\colon \MM_H \to \MM_G$ the morphism defined by restriction.
\end{enumerate}

The $G$-equivariant $TMF$-cohomology of some finite $G$-CW complex $X$ is then computed as the (homotopy groups of the) global sections of $\FF_G(X)$. We set the $G$-fixed points $TMF^G$ of $G$-equivariant $TMF$ to be $\Gamma(\OO^{top}_G) = \OO^{top}_G(\MM_G).$ Its homotopy groups are the value of $G$-equivariant $TMF$ at a point.

\begin{remark}
 In forthcoming work, Gepner and Nikolaus show that this definition of $TMF^G$ actually refines to a global equivariant spectrum $TMF$. 
\end{remark}

We need the following lemma. 

\begin{lemma}\label{lem:FiniteFlat}
 For $G$ finite abelian, the morphism $\MM_G \to \MM_{G/H}$ is finite and flat for every split inclusion $H\subset G$.
\end{lemma}
\begin{proof}By writing $H$ as a sum of cyclic groups and using induction, we can assume that $H$ is cyclic of order $k$. As $G \cong H \oplus G/H$, we obtain $\MM_G \simeq \MM_{G/H} \times_{\MM_{ell}} \EE^{uni}[k]$. By \cite[Theorem 2.3.1]{K-M85}, the map $\EE[k]$ is finite and flat over $\MM_{ell}$ and the result follows. 
\end{proof}

\begin{thm}
For $G$ finite abelian, the global sections functor
\[\Gamma\colon \QCoh(\MM_G^{top}) \to TMF^G\modules\]
is an equivalence of $\infty$-categories.
\end{thm}
\begin{proof}
 By construction, there is a morphism $\MM_G^{top} \to \MM_{ell}^{top}$ of derived stacks. By one of the main results of \cite{MM15}, the derived stack $\MM_{ell}^{top}$ is $0$-affine in the sense that the global sections functor 
 $$\Gamma\colon \QCoh(\MM_{ell}^{top}) \to TMF\modules$$
 is an equivalence of $\infty$-categories. By the last lemma the underlying morphism of stacks $\MM_G \to \MM_{ell}$ is finite and flat. It follows by \cite[Prop 3.29]{MM15} that $\MM_G^{top}$ is $0$-affine as well. 
\end{proof}

In particular, we do not loose information when we use the global sections functor. This gives special importance to the spectra $TMF^G$. We will determine $TMF^G$ in terms of better known spectra after localizing or completing at a prime not dividing $|G|$, or at least inverting $|G|$. To that purpose, let $\MM^G$ be the moduli stack of elliptic curves with $G$-level structure, i.e.\ for an elliptic curve $E/S$ an $S$-homomorphism $G\times S \to E$ that is injective after base change along every geometric point $s\colon \Spec k \to S$. 

\begin{lemma}
For $G$ finite abelian, there is a splitting 
$$\Hom(G,\EE^{uni})[\tfrac1{|G|}] \simeq \coprod_{K\subset G} \MM^{G/K}[\tfrac1{|G|}],$$ 
where $K$ runs over all subgroups of $G$ such that $G/K$ embeds into $(\Z/|G|)^2$.
 \end{lemma}
\begin{proof}Let $S$ be a connected scheme with $|G|$ invertible. We have to show that there is a natural decomposition $\MM_G(S) \simeq \coprod_{K\subset G} \MM^{G/K}(S)$ of groupoids (where $K$ is as above). Let $E/S$ be an elliptic curve and $f\colon G\times S \to E$ be a homomorphism over $S$. This factors over $E[|G|]$, which is a finite \'etale group scheme over $S$. Choose a geometric point $s$ of $S$. Then the category of finite \'etale group schemes over $S$ is equivalent to the category of finite groups with continuous $\pi_1(S,s)$-action of the \'etale fundamental group \cite[Thm 33]{Stix} or \cite[Exp 5]{SGA1}. For example, $E[|G|]$ is $(\Z/|G|)^2$ with a certain action and $f$ corresponds to a map $f'\colon G \to (\Z/|G|)^2$ with image in the $\pi_1(S,s)$-invariants.
 
 Let $K$ be the kernel of $f'$. The resulting map $\overline{f}'\colon G/K \to (\Z/|G|)^2$ corresponds to a homomorphism $\overline{f}\colon G/K\times S \to (\Z/|G|)^2$ over $S$ that is an injection on geometric points. 
 
 If we have conversely a subgroup $K\subset G$ and a $G/K$-level structure on $\EE$, this defines an $S$-homomorphism $G\times S \to E$ by precomposition with $G \to G/K$. 
\end{proof}

Using that $\MM_G \simeq \Hom(G,\EE^{uni})[\tfrac1{|G|}]$, this induces a corresponding splitting of the $E_\infty$-ring spectra $TMF^G[\frac1{|G|}]$. For example,
\[TMF^{\Z/n}[\tfrac1n] \simeq \prod_{k|n} TMF_1(k)[\tfrac1n]\]
as $E_\infty$-ring spectra and by Theorem \ref{thm:main} every factor with $k>3$ decomposes further into well-understood pieces. We obtain more generally: 

\begin{thm}
 Let $G$ be a finite abelian group. After completion at a prime $l$ not dividing $|G|$, the $TMF$-module $TMF^G$ splits into one copy of $TMF$ and even suspensions of $TMF_1(3)$ (for $l=2$), $TMF_1(2)$ (for $l=3$) or $TMF$ (if $l>3$). 
\end{thm}
\begin{proof}
 We $l$-complete throughout. By the last lemma, we have
 $$TMF^G \simeq \prod_{K\subset G} \OO^{top}(\MM^{G/K}).$$
 Assume that $G/K$ is non-trivial and write $G/K \cong \Z/k_1 \oplus \Z/k_2$ with $k_1\geq 2$. Then the map $h\colon \MM^{G/K} \to \MM_{ell}$ factors over $\MM_1(k_1)$. As in the proof of (1) of Proposition \ref{prop:basicprops}, the map $\MM^{G/K} \to \MM_1(k_1)$ is flat (as both source and target are finite \'etale over $\MM_{ell}$) and thus the pushforward of $\OO_{\MM^{G/K}}$ to $\MM_1(k_1)$ is a vector bundle. By Theorem \ref{thm:decuncompact}, this vector bundle $h_*\OO_{\MM^{G/K}}$ decomposes into vector bundles of the form $\omega^{\tensor j}$ (for $l>3$), $(f_2)_*\OO_{\MM_1(2)} \tensor \omega^{\tensor j}$ (for $l=3$) or $(f_3)_*\OO_{\MM_1(3)}\tensor \omega^{\tensor j}$ (for $l=2$). Analogously to the proof of Theorem \ref{thm:main}, we conclude a topological splitting from this algebraic one. 
\end{proof}

If $H^1(\MMb_1(k); \omega)$ is $l$-torsionfree for all $k|n$ and $k>1$, we can use Theorem \ref{thm:deccompact} to replace $l$-completion by localization at $l$.

\appendix
\section{Lifting the Hasse invariant}\label{app:Hasse}
This appendix does not contain an original contribution by the author. Besides a short introduction to the Hasse invariant it proves that the Hasse invariant is liftable to characteristic zero once we have chosen a $\Gamma_1(k)$-level structure for $k\geq 2$. This proof is (essentially) taken from \cite{Hasse} and the credit belongs to the mathoverflow user Electric Penguin. 

We begin by recalling the definition of the Hasse invariant from \cite[Section 2.0]{Kat73}. Let $f\colon E\to \Spec R$ be an elliptic curve. By Grothendieck duality, we have a canonical isomorphism $R^1f_*\OO_E \cong \omega_{E/R}^{\tensor -1}$ of line bundles. Choosing a trivialization of $\omega_{E/R}$ thus induces a dual $R$-basis $x$ of $H^1(E;\OO_E)$. If $R$ is an $\F_p$-algebra, the absolute Frobenius $F_{abs}$ acts on $H^1(E;\OO_E)$ so that $F_{abs}^*(x) = \lambda x$ with $\lambda \in R$. Associating to each elliptic curve $E\to \Spec R$ over an $\F_p$-algebra $R$ with chosen trivialization of $\omega_{E/R}$ the element $\lambda \in R$ defines an $\F_p$-valued (holomorphic) modular form $A_p$ of weight $p-1$, as explained in \cite{Kat73}. This is the \emph{Hasse invariant}. Katz shows that the $q$-expansion of $A_p$ in $\F_p\llbracket q\rrbracket$ is identically $1$.

We want to prove that the Hasse invariant is liftable to characteristic $0$ in the presence of a level structure. More precisely, we have the following:

\begin{prop}\label{prop:Hasse}
 For every odd $k\geq 2$, there is a modular form $F$ of weight $1$ and level $\Gamma_1(k)$ over a cyclotomic ring $\Z_{(2)}[\zeta_{2^m}]$ such that $F\equiv A_2 \mod 2$. 
\end{prop}
\begin{proof}
The proof we present is based on the mathoverflow post \cite{Hasse} by the user Electric Penguin and uses the theory of Eisenstein series. Let $\chi\colon (\Z/k)^\times \to \C^\times$ be an odd character, i.e.\ we require $\chi(-1) = -1$. Following \cite[Section 4.8]{D-S05}, we define its associated weight $1$ Eisenstein series by
\[E_1^{\chi}(\tau) = \frac12L(0,\chi) + \sum_{n=1}^{\infty}c_nq^n,\]
where $c_n = \sum_{d|n}\chi(n)$ (this is half of the normalization chosen in \cite{D-S05}). Here $L(s,\chi)$ denotes the Dirichlet $L$-series associated to $\chi$. Let $a$ be the \emph{conductor} of $\chi$, i.e.\ the smallest $a|k$ such that $\chi$ factors over $(\Z/a)^\times$. Then we have
\[L(0,\chi) = -\frac1{a}\sum_{n=1}^{a-1} n\chi(n).\]
It is proven in \cite{D-S05} that $E_1^{\chi}$ is a $\Gamma_1(k)$-modular form of weight $1$ over the ring $\C$ with character $\chi$ (although the latter fact will not be relevant for us). 

In general, if $K/\Q$ is a finite extension of degree $n$, we can extend the $2$-adic valuation from $\Q$ to $K$ by setting $v_2(x) = \frac1n v_2(N_{K/\Q}(x))$ for $x\in K$. For example, let $K/\Q$ be the cyclotomic extension $\Q(\zeta)$ with $\zeta =\zeta_{2^m}$ a primitive root of unity. Then $N_{K/\Q}(x-\zeta) = x^{2^{m-1}}+1$ (for $x$ not acted upon by the Galois group) as both have the same zeros. In particular, $v_2(1-\zeta) = \frac1{2^{m-2}}$. As $1-\zeta$ generates the maximal ideal of $\Z_{(2)}[\zeta]$, we see that $\frac1{2^{m-2}}$ is the minimal positive $2$-adic valuation in $\Q(\zeta)$ and thus every $2$-adic valuation is a multiple of $\frac1{2^{m-2}}$. 

For the proof of the proposition, we may assume $k$ to be an odd prime $p$ (as every $\Gamma_1(p)$-modular form is also a $\Gamma_1(k)$-modular form for $p|k$), which we will do in the following. Thus consider an odd character $\chi\colon (\Z/p)^\times \to \C^\times$ and the associated Eisenstein series $E_1^\chi$. We will assume that $\chi$ has order $2^m$ for $p-1 = 2^ml$ with $l$ odd. This implies that $\chi$ is surjective onto the $2^m$-th roots of unity and that $\ker(\chi)\subset (\Z/p)^\times$ has order $l$.

\begin{claim}We have
 $$v_2(L(0,\chi)) = 1-\frac1{2^{m-2}},$$
 where the valuation is taken in $\Q(\zeta_{2^m})$. 
\end{claim}
\begin{myproof}
Choose $b_0,\dots, b_{2^{m-1}-1}$ with $\chi(b_j) = \zeta^j$, where we still use the notation $\zeta = \zeta_{2^m}$. We furthermore use the notation $\overline{x}$ to denote for an integer $x$ the integer $0\leq \overline{x}\leq p-1$ it is congruent to mod $p$. 

We have
\begin{align*}
L(0,\chi) &= -\frac1p\sum_{n=1}^{p-1}n\chi(n) \\
 &= -\frac1p\sum_{j=0}^{2^{m-1}-1} \sum_{[i]\in\ker(\chi)} \left(\overline{ib_j}\chi(ib_j) + (p-\overline{ib_j})\chi(p-ib_j)\right) \\
 &= -\frac1p\sum_{j=0}^{2^{m-1}-1} (-pl + 2\sum_{[i]\in\ker(\chi)} \overline{ib_j})\zeta^j \\
 &\equiv \sum_{j=0}^{2^{m-1}-1} \zeta^j \mod 2
\end{align*}
in $\Z_{(2)}[\zeta]$. As this is not congruent to $0$ mod $2$, this implies in particular that $v_2(L(0,\chi)) < 1$. Moreover,
 \[L(0,\chi)(1-\zeta) \equiv (1+\zeta + \cdots + \zeta^{2^{m-1}-1})(1-\zeta) \equiv 1+1 \equiv 0 \mod 2\]
 and thus $v_2(L(0,\chi)(1-\zeta)) = v_2(L(0,\chi)) + \frac1{2^{m-2}}\geq 1$. As every $2$-adic valuation $\Z_{(2)}[\zeta]$ is a multiple of $\frac1{2^{m-2}}$, this implies the result. 
\end{myproof}

We see that $E =(1-\zeta)E_1^{\chi}$ is a level $p$, weight $1$ modular form with $q$-expansion in the ring $\Z_{(2)}[\zeta]$. By the $q$-expansion principle from \cite[Thm 1.6.1]{Kat73}, $E$ is thus a modular form over the ring $\Z_{(2)}[\zeta]$. Furthermore, we know that $E \equiv 1 \mod (1-\zeta)$ in $\Z_{(2)}[\zeta]\llbracket q\rrbracket$.

Write 
\[ E = \sum_{i=0}^{2^{m-1}-1}\zeta^if_i \in \Z_{(2)}[\zeta]\llbracket q\rrbracket\]
with $f_i \in \Z_{(2)}\llbracket q\rrbracket$. 

\begin{claim}
Each $f_i$ is a $\Gamma_1(p)$-modular form.
\end{claim}
\begin{myproof}
The Galois group $(\Z/2^m)^\times = \Gal(\Q(\zeta)/\Q)$ acts on $\Q(\zeta)$-valued modular forms. In particular, $\sum_{g\in \Gal(\Q(\zeta)/\Q)}g(\zeta^{-i}E)$ is a modular form, namely $2^{m-1}f_i$. 
\end{myproof}

Thus, also $$F = \sum_{i=0}^{2^{m-1}-1}f_i\in\Z_{(2)}\llbracket q\rrbracket$$
is a $\Gamma_1(p)$-modular form over the ring $\Z_{(2)}[\zeta]$ and also 
$$F \equiv E \equiv 1 \mod (1-\zeta).$$
As $(1-\zeta)\Z_{(2)}[\zeta]\llbracket q\rrbracket \cap \Z_{(2)}\llbracket q\rrbracket = (2)\Z_{(2)}\llbracket q\rrbracket$, we also get $F \equiv 1 \mod 2$. 

This proves the proposition.
\end{proof}
\newpage

\section{Tables of decomposition numbers}\label{sec:tables}
Let $f_n\colon \MMb_1(n)_{\C} \to \MMb_{ell,\C}$ be the projection and $$(f_n)_*\OO_{\MMb_1(n)_\C} \cong \bigoplus_{i\in\Z}\bigoplus_{l_i}\omega^{\tensor -i}.$$

\renewcommand{\arraystretch}{0.98}
  \begin{tabular}{lccccccccccccc}
$n$ & genus & $l_0$ & $l_1$ & $l_2$ & $l_3 $& $l_4$ & $l_5$ & $l_6$ & $l_7$ & $l_8$ & $l_9$ & $l_{10}$ & $l_{11}$ \\\hline\\
2 & 0 & 1 & 0 & 1 & 0 & 1 & 0 & 0 & 0 & 0 & 0 & 0 & 0 \\
3 & 0 & 1 & 1 & 1 & 2 & 1 & 1 & 1 & 0 & 0 & 0 & 0 & 0 \\
4 & 0 & 1 & 1 & 2 & 2 & 2 & 2& 1 & 1 & 0 & 0 & 0 & 0 \\
5 & 0 & 1 & 2 & 3 & 4 & 4 & 4 & 3 & 2 & 1 & 0 & 0 & 0 \\
6 & 0 & 1 & 2 & 3 & 4 & 4 & 4 & 3 & 2 & 1 & 0 & 0 & 0 \\
7 & 0 & 1 & 3 & 5 & 7 & 8 & 8 & 7 & 5 & 3 & 1 & 0 & 0 \\
8 & 0 & 1 & 3 & 5 & 7 & 8 & 8 & 7 & 5 & 3 & 1 & 0 & 0 \\
9 & 0 & 1 & 4 & 7 & 10 & 12 & 12 & 11 & 8 & 5 & 2 & 0 & 0 \\
10 & 0 & 1 & 4 & 7 & 10 & 12 & 12 & 11 & 8 & 5 & 2 & 0 & 0 \\
11 & 1 & 1 & 5 & 10 & 15 & 19 & 20 & 19 & 15 & 10 & 5 & 1 & 0 \\
12 & 0 & 1 & 5 & 9 & 13 & 16 & 16 & 15 & 11 & 7 & 3 & 0 & 0 \\
13 & 2 & 1 & 6 & 13 & 20 & 26 & 28 & 27 & 22 & 15 & 8 & 2 & 0 \\
14 & 1 & 1 & 6 & 12 & 18 & 23 & 24 & 23 & 18 & 12 & 6 & 1 & 0 \\
15 & 1 & 1 & 8 & 16 & 24 & 31 & 32 & 31 & 24 & 16 & 8 & 1 & 0 \\
16 & 2 & 1 & 7 & 15 & 23 & 30 & 32 & 31 & 25 & 17 & 9 & 2 & 0 \\
17 & 5 & 1 & 8 & 20 & 32 & 43 & 48 & 47 & 40 & 28 & 16 & 5 & 0 \\
18 & 2 & 1 & 8 & 17 & 26 & 34 & 36 & 35 & 28 & 19 & 10 & 2 & 0 \\
19 & 7 & 1 & 9 & 24 & 39 & 53 & 60 & 59 & 51 & 36 & 21 & 7 & 0 \\
20 & 3 & 1 & 10 & 22 & 34 & 45 & 48 & 47 & 38 & 26 & 14 & 3 & 0 \\
21 & 5 & 1 & 12 & 28 & 44 & 59 & 64 & 63 & 52 & 36 & 20 & 5 & 0 \\
22 & 6 & 1 & 10 & 25 & 40 & 54 & 60 & 59 & 50 & 35 & 20 & 6 & 0 \\
23 & 12 & 1 & 12 & 33 & 55 & 76 & 87 & 87 & 76 & 55 & 33 & 12 & 1 \\
24 & 5 & 1 & 12 & 28 & 44 & 59 & 64 & 63 & 52 & 36 & 20 & 5 & 0 \\
25 & 12 & 1 & 14 & 39 & 64 & 88 & 100 & 99 & 86 & 61 & 36 & 12 & 0 \\
26 & 10 & 1 & 12 & 33 & 54 & 74 & 84 & 83 & 72 & 51 & 30 & 10 & 0 \\
27 & 13 & 1 & 15 & 42 & 69 & 95 & 108 & 107 & 93 & 66 & 39 & 13 & 0 \\
28 & 10 & 1 & 15 & 39 & 63 & 86 & 96 & 95 & 81 & 57 & 33 & 10 & 0 \\
29 & 22 & 1 & 14 & 49 & 84 & 118 & 140 & 139 & 126 & 91 & 56 & 22 & 0 \\
30 & 9 & 1 & 16 & 40 & 64 & 87 & 96 & 95 & 80 & 56 & 32 & 9 & 0 \\
31 & 26 & 1 & 16 & 55 & 95 & 134 & 159 & 159 & 144 & 105 & 65 & 26 & 1 \\
32 & 17 & 1 & 16 & 48 & 80 & 111 & 128 & 127 & 112 & 80 & 48 & 17 & 0 \\
33 & 21 & 1 & 20 & 60 & 100 & 139 & 160 & 159 & 140 & 100 & 60 & 21 & 0 \\
34 & 21 & 1 & 16 & 52 & 88 & 123 & 144 & 143 & 128 & 92 & 56 & 21 & 0 \\
35 & 25 & 1 & 24 & 72 & 120 & 167 & 192 & 191 & 168 & 120 & 72 & 25 & 0 \\
36 & 17 & 1 & 20 & 56 & 92 & 127 & 144 & 143 & 124 & 88 & 52 & 17 & 0 \\
37 & 40 & 1 & 18 & 75 & 132 & 188 & 228 & 227 & 210 & 153 & 96 & 40 & 0 \\
38 & 28 & 1 & 18 & 63 & 108 & 152 & 180 & 179 & 162 & 117 & 72 & 28 & 0 \\
39 & 33 & 1 & 25 & 80 & 136 & 191 & 223 & 223 & 199 & 144 & 88 & 33 & 1 \\
40 & 25 & 1 & 24 & 72 & 120 & 167 & 192 & 191 & 168 & 120 & 72 & 25 & 0 \\
41 & 51 & 1 & 20 & 90 & 160 & 229 & 280 & 279 & 260 & 190 & 120 & 51 & 0 \\
42 & 25 & 1 & 24 & 72 & 120 & 167 & 192 & 191 & 168 & 120 & 72 & 25 & 0 
 \end{tabular}

 \newpage
\begin{tabular}{cc}
\begin{tabular}{lcccccccc}
$n$ &  $k_0$ & $k_1$ & $k_2$ & $k_3 $& $k_4$ & $k_5$ & $k_6$ & $k_7$\\ \hline
4 & 1 & 1 & 1 & 1 & 0 & 0 & 0 & 0 \\
5 & 1 & 2 & 2 & 2 & 1 & 0 & 0 & 0 \\
6 & 1 & 2 & 2 & 2 & 1 & 0 & 0 & 0 \\
7 & 1 & 3 & 4 & 4 & 3 & 1 & 0 & 0 \\
8 & 1 & 3 & 4 & 4 & 3 & 1 & 0 & 0 \\
9 & 1 & 4 & 6 & 6 & 5 & 2 & 0 & 0 \\
10 & 1 & 4 & 6 & 6 & 5 & 2 & 0 & 0 \\
11 & 1 & 5 & 9 & 10 & 9 & 5 & 1 & 0 \\
12 & 1 & 5 & 8 & 8 & 7 & 3 & 0 & 0 \\
13 & 1 & 6 & 12 & 14 & 13 & 8 & 2 & 0 \\
14 & 1 & 6 & 11 & 12 & 11 & 6 & 1 & 0 \\
15 & 1 & 8 & 15 & 16 & 15 & 8 & 1 & 0 \\
16 & 1 & 7 & 14 & 16 & 15 & 9 & 2 & 0 \\
17 & 1 & 8 & 19 & 24 & 23 & 16 & 5 & 0 \\
18 & 1 & 8 & 16 & 18 & 17 & 10 & 2 & 0 \\
19 & 1 & 9 & 23 & 30 & 29 & 21 & 7 & 0 \\
20 & 1 & 10 & 21 & 24 & 23 & 14 & 3 & 0 \\
21 & 1 & 12 & 27 & 32 & 31 & 20 & 5 & 0 \\
22 & 1 & 10 & 24 & 30 & 29 & 20 & 6 & 0 \\
23 & 1 & 12 & 32 & 43 & 43 & 32 & 12 & 1 
\end{tabular}
&
\begin{minipage}{6cm}\begin{center} Decomposition numbers for $$(f_n)_*\OO_{\MMb_1(n)_{(3)}}$$ for decompositions into $k_i$ copies of $$(f_2)_*\OO_{\MMb_1(2)_{(3)}}\tensor \omega^{\tensor -i}$$ \end{center}\end{minipage}\\\\

\begin{tabular}{lccccccc}
$n$ &  $k_0$ & $k_1$ & $k_2$ & $k_3 $& $k_4$ & $k_5$  \\ \hline
5 & 1 & 1 & 1 & 0 & 0 & 0 \\
6 & 1 & 1 & 1 & 0 & 0 & 0 \\
7 & 1 & 2 & 2 & 1 & 0 & 0 \\
8 & 1 & 2 & 2 & 1 & 0 & 0 \\
9 & 1 & 3 & 3 & 2 & 0 & 0 \\
10 & 1 & 3 & 3 & 2 & 0 & 0 \\
11 & 1 & 4 & 5 & 4 & 1 & 0 \\
12 & 1 & 4 & 4 & 3 & 0 & 0 \\
13 & 1 & 5 & 7 & 6 & 2 & 0 \\
14 & 1 & 5 & 6 & 5 & 1 & 0 \\
15 & 1 & 7 & 8 & 7 & 1 & 0 \\
16 & 1 & 6 & 8 & 7 & 2 & 0 \\
17 & 1 & 7 & 12 & 11 & 5 & 0 \\
18 & 1 & 7 & 9 & 8 & 2 & 0 \\
19 & 1 & 8 & 15 & 14 & 7 & 0 \\
20 & 1 & 9 & 12 & 11 & 3 & 0 \\
21 & 1 & 11 & 16 & 15 & 5 & 0 \\
22 & 1 & 9 & 15 & 14 & 6 & 0 \\
23 & 1 & 11 & 21 & 21 & 11 & 1 

\end{tabular}
& \begin{minipage}{7cm}\begin{center} Decomposition numbers for $$(f_n)_*\OO_{\MMb_1(n)_{(2)}}$$ for decompositions into $k_i$ copies of $$(f_3)_*\OO_{\MMb_1(3)_{(2)}}\tensor \omega^{\tensor -i}$$\end{center}\end{minipage}
\end{tabular}

\bibliographystyle{alpha}
\bibliography{../Chromatic}
\end{document}